\documentclass[11pt]{amsart}
\usepackage{amsmath}
\usepackage{amsfonts}
\usepackage{amssymb}
\usepackage[all,cmtip]{xy}           %xypic macro for latex2.09

\usepackage{dsfont}
\usepackage{bbm}
\usepackage{bbding}
\usepackage{txfonts}
\usepackage[shortlabels]{enumitem}
\usepackage{ifpdf}
\ifpdf
\usepackage[colorlinks,linkcolor=black,final,backref=page,hyperindex]{hyperref}
\else
\usepackage[colorlinks,linkcolor=black,final,backref=page,hyperindex,hypertex]{hyperref}
\fi
\usepackage{amscd}
\usepackage{tikz-cd}
\usepackage[active]{srcltx}
\usepackage{mathrsfs}
\usepackage{ulem}

\makeatletter
\newtheorem{theorem}{Theorem}[section]
\newtheorem{proposition}[theorem]{Proposition}
\newtheorem{lemma}[theorem]{Lemma}
\newtheorem{corollary}[theorem]{Corollary}
\theoremstyle{definition} 
\newtheorem{example}[theorem]{Example}
\newtheorem{definition}[theorem]{Definition}
\newtheorem{remark}[theorem]{Remark}

\newcommand{\vv}{\mathsf{v}}
\newcommand{\ww}{\mathsf{w}}

\newcommand{\qa}{kQ/I}

\newcommand{\soc}{\mathrm{soc}}

\newcommand{\red}{\mathrm{red}}

\begin{document}
	
	\title[Two-term tilting complexes of biserial fractional Brauer graph algebras]{
	Two-term tilting complexes of biserial fractional Brauer graph algebras }

	\author{Bohan Xing} 
	\address{(Bohan Xing)
		School of Mathematical Sciences,
		Laboratory of Mathematics and Complex Systems,
		Beijing Normal University,
		Beijing 100875,
		P.R.China}
	% \email{wd0843@163.com}
	% \email{ymliu@bnu.edu.cn}
	\email{bhxing@mail.bnu.edu.cn}

	\date{\today}

	\begin{abstract}
	Brauer graph algebras form a classical class of symmetric algebras with well-structured combinatorial properties and geometric models. Recently, they have been generalized to biserial fractional Brauer graph algebras, which can be regarded as a self-injective version of the classical Brauer graph algebras. In this paper, we show that the skew group algebras of biserial fractional Brauer graph algebras induced by the Nakayama automorphism are in fact skew-Brauer graph algebras. We then study two-term tilting complexes and Kauer moves for biserial fractional Brauer graph algebras. Moreover, we prove that a biserial fractional Brauer graph algebra is tilting-discrete if and only if its reduced form (which is a Brauer graph algebra) is tilting-discrete. Finally, we show that tilting-discrete biserial fractional Brauer graph algebras are closed under derived equivalence.
	\end{abstract}

	\renewcommand{\thefootnote}{\alph{footnote}}
\setcounter{footnote}{-1} 
\renewcommand{\thefootnote}{\alph{footnote}}
\setcounter{footnote}{-1} 
	\maketitle
	
	%\tableofcontents

	\allowdisplaybreaks
	
	\section{Introduction}

	Brauer graph algebras (abbr. BGAs) originate from the modular representation theory of finite groups. A subclass, Brauer tree algebras, first appear in the work of Janusz \cite{J} and Dade \cite{D}. The general case, BGAs, is subsequently defined by Donovan and Freislich in \cite{DF}. In general, BGAs form a class of symmetric algebras equipped with a combinatorial structure given by a pair $(\Gamma, m)$, where $\Gamma$ is a ribbon graph (i.e., a graph embedded in an orientable surface) and $m \colon V(\Gamma) \to \mathbb{Z}_{>0}$ is a multiplicity function. They are precisely the symmetric special biserial algebras \cite{Sch}, and the representation-finite ones coincide with Brauer tree algebras, which are derived equivalent to symmetric Nakayama algebras (see for example \cite{SS}).

	In order to study self-injective algebras via a similar combinatorial approach, Li and Liu \cite{LL} generalized BGAs to fractional BGAs. In general, fractional BGAs are neither special biserial nor of tame representation type \cite{X}. Consequently, in \cite{LL,LL3}, Li and Liu studied fractional BGAs of type MS, which are self-injective and special biserial. Here ``MS'' stands for “multiserial and self-injective” \cite{GS1}. In this article, we simply refer to such MS-type algebras as {\it biserial fractional BGAs} (abbr. biserial FBGAs), since the biserial part of fractional BGAs is special biserial and coincides with fractional BGAs of type MS. Same as BGAs, these algebras are defined from combinatorial data --- a ribbon graph equipped with a \textit{degree function} --- called a biserial fractional Brauer graph (abbr. biserial FBG), satisfying a certain condition (see Definition \ref{fms-BG}). We illustrate these concepts with the following example.

	\begin{example}\label{exa:preproj-A3}
		Consider the following ribbon graph $\Gamma$, consisting of two vertices $\vv$ and $\ww$ and three edges $\bar{1},\bar{2},\bar{3}$, embedded in a torus with two punctures. Without loss of generality, we assume that the cyclic ordering around each vertex is clockwise. We use $i,i'$ to denote the two distinct half-edges that correspond to the edge $\bar{i}$.

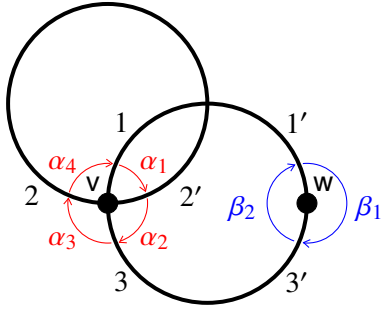
		\begin{center}   
\begin{tikzpicture}[x=0.75pt,y=0.75pt,yscale=-1,xscale=1]
%uncomment if require: \path (0,235); %set diagram left start at 0, and has height of 235

%Shape: Circle [id:dp8262589623190342] 
\draw  [fill={rgb, 255:red, 0; green, 0; blue, 0 }  ,fill opacity=1 ] (141,157) .. controls (141,154.24) and (143.24,152) .. (146,152) .. controls (148.76,152) and (151,154.24) .. (151,157) .. controls (151,159.76) and (148.76,162) .. (146,162) .. controls (143.24,162) and (141,159.76) .. (141,157) -- cycle ;
%Shape: Circle [id:dp01868024686061953] 
\draw  [line width=1.5]  (146,157) .. controls (146,129.39) and (168.39,107) .. (196,107) .. controls (223.61,107) and (246,129.39) .. (246,157) .. controls (246,184.61) and (223.61,207) .. (196,207) .. controls (168.39,207) and (146,184.61) .. (146,157) -- cycle ;
%Shape: Circle [id:dp9090944636447427] 
\draw  [fill={rgb, 255:red, 0; green, 0; blue, 0 }  ,fill opacity=1 ] (241,157) .. controls (241,154.24) and (243.24,152) .. (246,152) .. controls (248.76,152) and (251,154.24) .. (251,157) .. controls (251,159.76) and (248.76,162) .. (246,162) .. controls (243.24,162) and (241,159.76) .. (241,157) -- cycle ;
%Shape: Circle [id:dp6722146453306068] 
\draw  [line width=1.5]  (96,107) .. controls (96,79.39) and (118.39,57) .. (146,57) .. controls (173.61,57) and (196,79.39) .. (196,107) .. controls (196,134.61) and (173.61,157) .. (146,157) .. controls (118.39,157) and (96,134.61) .. (96,107) -- cycle ;
%Shape: Arc [id:dp5101488299886188] 
\draw  [draw opacity=0] (126.75,151.54) .. controls (129.13,143.15) and (136.85,137) .. (146,137) .. controls (146.42,137) and (146.84,137.01) .. (147.25,137.04) -- (146,157) -- cycle ; 
\draw [red]  (126.75,151.54) .. controls (129.13,143.15) and (136.85,137) .. (146,137) .. controls (146.42,137) and (146.84,137.01) .. (147.25,137.04) ;  
\draw [red]  (145.31,134.88) -- (148.24,137.13) -- (144.98,138.86) ;

%Shape: Arc [id:dp5181964745321894] 
\draw  [draw opacity=0] (147.85,176.92) .. controls (147.24,176.97) and (146.62,177) .. (146,177) .. controls (134.95,177) and (126,168.05) .. (126,157) .. controls (126,156.14) and (126.05,155.3) .. (126.16,154.47) -- (146,157) -- cycle ; 
\draw [red]  (147.85,176.92) .. controls (147.24,176.97) and (146.62,177) .. (146,177) .. controls (134.95,177) and (126,168.05) .. (126,157) .. controls (126,156.14) and (126.05,155.3) .. (126.16,154.47) ;  
\draw [red]  (123.93,157.41) -- (126.33,154.61) -- (127.9,157.94) ;

%Shape: Arc [id:dp39416535042659895] 
\draw  [draw opacity=0] (151.68,137.82) .. controls (157.98,139.68) and (162.98,144.56) .. (165.02,150.78) -- (146,157) -- cycle ; 
\draw [red]  (151.68,137.82) .. controls (157.98,139.68) and (162.98,144.56) .. (165.02,150.78) ;  
\draw [red]  (165.72,147.41) -- (164.96,151.02) -- (161.99,148.84) ;

%Shape: Arc [id:dp016647716993849704] 
\draw  [draw opacity=0] (165.85,154.56) .. controls (165.95,155.36) and (166,156.17) .. (166,157) .. controls (166,165.85) and (160.25,173.36) .. (152.28,175.99) -- (146,157) -- cycle ; 
\draw [red]  (165.85,154.56) .. controls (165.95,155.36) and (166,156.17) .. (166,157) .. controls (166,165.85) and (160.25,173.36) .. (152.28,175.99) ;  
\draw [red]  (155.59,176.86) -- (151.91,176.61) -- (153.65,173.36) ;

%Shape: Arc [id:dp8241553012004375] 
\draw  [draw opacity=0] (239.57,175.94) .. controls (231.68,173.27) and (226,165.8) .. (226,157) .. controls (226,148.18) and (231.71,140.69) .. (239.64,138.03) -- (246,157) -- cycle ; 
\draw [blue]  (239.57,175.94) .. controls (231.68,173.27) and (226,165.8) .. (226,157) .. controls (226,148.18) and (231.71,140.69) .. (239.64,138.03) ;  
\draw [blue]  (236.15,137.29) -- (239.81,137.76) -- (237.87,140.9) ;

%Shape: Arc [id:dp14343189840461412] 
\draw  [draw opacity=0] (243.94,137.1) .. controls (244.62,137.04) and (245.31,137) .. (246,137) .. controls (257.05,137) and (266,145.95) .. (266,157) .. controls (266,168.05) and (257.05,177) .. (246,177) .. controls (245.37,177) and (244.75,176.97) .. (244.13,176.91) -- (246,157) -- cycle ; 
\draw [blue]  (243.94,137.1) .. controls (244.62,137.04) and (245.31,137) .. (246,137) .. controls (257.05,137) and (266,145.95) .. (266,157) .. controls (266,168.05) and (257.05,177) .. (246,177) .. controls (245.37,177) and (244.75,176.97) .. (244.13,176.91) ;  
\draw [blue]  (247.03,179.23) -- (244.31,176.74) -- (247.69,175.28) ;

% Text Node
\draw (148,110) node [anchor=north west][inner sep=0.75pt]   [align=left] {$1$};
% Text Node
\draw (234,110) node [anchor=north west][inner sep=0.75pt]   [align=left] {$1'$};
% Text Node
\draw (103,147) node [anchor=north west][inner sep=0.75pt]   [align=left] {$2$};
% Text Node
\draw (181,147) node [anchor=north west][inner sep=0.75pt]   [align=left] {$2'$};
% Text Node
\draw (134,142) node [anchor=north west][inner sep=0.75pt]   [align=left] {$\vv$};
% Text Node
\draw (248,142) node [anchor=north west][inner sep=0.75pt]   [align=left] {$\ww$};
% Text Node
\draw (148,190) node [anchor=north west][inner sep=0.75pt]   [align=left] {$3$};
% Text Node
\draw (234,190) node [anchor=north west][inner sep=0.75pt]   [align=left] {$3'$};
% Text Node
\draw (161,132) node [anchor=north west][inner sep=0.75pt]   [align=left] {\textcolor{red}{$\alpha_1$}};
% Text Node
\draw (161,170) node [anchor=north west][inner sep=0.75pt]   [align=left] {\textcolor{red}{$\alpha_2$}};
% Text Node
\draw (115,170) node [anchor=north west][inner sep=0.75pt]   [align=left] {\textcolor{red}{$\alpha_3$}};
% Text Node
\draw (115,132) node [anchor=north west][inner sep=0.75pt]   [align=left] {\textcolor{red}{$\alpha_4$}};
% Text Node
\draw (270,152) node [anchor=north west][inner sep=0.75pt]   [align=left] {\textcolor{blue}{$\beta_1$}};
% Text Node
\draw (206,152) node [anchor=north west][inner sep=0.75pt]   [align=left] {\textcolor{blue}{$\beta_2$}};
\end{tikzpicture}
		\end{center}
	As in the case of BGAs, the above ribbon graph $\Gamma$ naturally induces a quiver $Q_\Gamma$ as follows, whose vertices correspond to the edges of $\Gamma$, and whose arrows are determined by the cyclic ordering around the vertices of $\Gamma$.
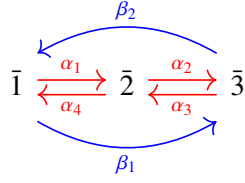
	$$\begin{tikzcd}
\bar{1} 
\arrow[r, shift left, draw=red, "\textcolor{red}{\alpha_1}"] 
\arrow[rr, bend right, shift right=3, draw=blue, "\textcolor{blue}{\beta_1}"'] 
& \bar{2} 
\arrow[r, shift left, draw=red, "\textcolor{red}{\alpha_2}"] 
\arrow[l, shift left, draw=red, "\textcolor{red}{\alpha_4}"] 
& \bar{3} 
\arrow[l, shift left, draw=red, "\textcolor{red}{\alpha_3}"] 
\arrow[ll, bend right, shift right=3, draw=blue, "\textcolor{blue}{\beta_2}"']
\end{tikzcd}$$
Here $\alpha$ denotes the arrows induced by the cyclic ordering around the vertex $\vv$, while $\beta$ denotes the arrows induced by the cyclic ordering around the vertex $\ww$.

In contrast to the multiplicity function in the setting of BGAs, which specifies how many times a path may wind around a given vertex, the degree function in biserial FBGAs records the maximal number of steps around a vertex, that is, the maximal length of such a path. Therefore, for a BGA, the degree $d(\vv)$ of each vertex $\vv$ in the ribbon graph is given by the product $d(\vv)=m(\vv)\,val(\vv)$, where $val(\vv)$ denotes the valency of $\vv$. Thus, one may regard $(\Gamma,d)$ as defining a BGA provided that the degree function $d$ is divisible by the valency function $val$.

{\it Thus, to avoid confusion with the multiplicity function $m$, throughout this paper we consider a pair $(\Gamma,d)$, where $d$ is understood to be the degree function unless otherwise specified.}

For instance, if we consider the degree function $d$ on $\Gamma$ defined by $$d(\vv)=2,\quad d(\ww)=1,$$ then $(\Gamma,d)$ forms a biserial FBG, and the associated algebra admits the following commutative relations:
$$
\textcolor{red}{\alpha_4}\textcolor{red}{\alpha_1}
=
\textcolor{red}{\alpha_2}\textcolor{red}{\alpha_3},
\quad
\textcolor{blue}{\beta_1}
=
\textcolor{red}{\alpha_1}\textcolor{red}{\alpha_2},
\quad
\textcolor{blue}{\beta_2}
=
\textcolor{red}{\alpha_3}\textcolor{red}{\alpha_4}.
$$
As in the case of BGAs, we further impose ``gentle relations'' of the form $\textcolor{red}{\alpha}\textcolor{blue}{\beta}
=
\textcolor{blue}{\beta}\textcolor{red}{\alpha}
=
0$. When the degree function $d$ satisfies certain additional conditions (ensuring that the paths appearing in the above relations have the same source and terminus), the resulting algebra $A$ associated with $(\Gamma,d)$ is a finite-dimensional self-injective special biserial algebra. Its Gabriel quiver is obtained from $Q_\Gamma$ by deleting the arrows $\beta_1$ and $\beta_2$. Moreover, the indecomposable projective $A$-modules are given by
			$$P_1=\begin{array}{*{1}{l}}
			1\\2\\3
		\end{array},\quad P_2=\begin{array}{*{1}{lll}}
			&2&\\1&&3\\&2&
		\end{array},\quad P_3=\begin{array}{*{1}{l}}
			3\\2\\1
		\end{array}.$$
Note that $A$ is the preprojective algebra of type $A_3$. The Nakayama automorphism $\nu_A$ of $A$, given by
$$
e_{\bar{1}}\longleftrightarrow e_{\bar{1}},\quad
e_{\bar{2}}\longleftrightarrow e_{\bar{3}},\quad
\alpha_1\longleftrightarrow \alpha_3,\quad
\alpha_2\longleftrightarrow \alpha_4,
$$
is naturally corresponding to the following automorphism of $\Gamma$:
$$
1\longleftrightarrow 1',\quad
2\longleftrightarrow 3,\quad
2'\longleftrightarrow 3',
$$
which is called the {\it Nakayama automorphism} $\nu$ of $(\Gamma,d)$ in \cite{LL}. 
	\end{example}

Consider the graph of orbits $\Gamma/\langle\nu\rangle$ (see Subsection~\ref{subsec:reduced-form}). It may contain orbifold edges, and its surface model is an orbifold surface in the sense of \cite{AP2,BSW,So2}, see also Definition \ref{def:orbifold}. If $\Gamma/\langle\nu\rangle$ has no orbifold edges, then the biserial FBG $(\Gamma,d)$ is called {\it admissible}. Otherwise, $(\Gamma,d)$ is called {\it non-admissible}. This terminology reflects the admissibility of the action of $\langle\nu\rangle$ on the set of half-edges; see~\cite{BG,LL3}. For a non-admissible pair $(\Gamma,d)$, there exists a ribbon graph $\widehat{\Gamma/\langle \nu \rangle}$ equipped with a natural orientation-preserving involution $\phi$ (which induces a $\mathbb{Z}/2\mathbb{Z}$-action on the BGA associated with $\widehat{\Gamma/\langle \nu \rangle}$), and is a double cover of $\Gamma/\langle\nu\rangle$. The {\it reduced form} of $(\Gamma,d)$ is defined by
$$
\Gamma_{\red} =
\begin{cases}
\Gamma/\langle\nu\rangle, & \text{if } (\Gamma,d) \text{ is admissible}, \\[2mm]
\widehat{\Gamma/\langle\nu\rangle}, & \text{if } (\Gamma,d) \text{ is non-admissible}.
\end{cases}
$$
Moreover, the degree function $d$ naturally induces a degree function on $\Gamma_{\red}$, such that $(\Gamma_{\red}, d)$ is a classical Brauer graph. It is shown in \cite[Theorem~2.29]{LL3} that a biserial FBGA and its reduced form have the same representation type.

\begin{figure}[ht]
		\centering
		\begin{center}

\tikzset{every picture/.style={line width=0.75pt}} %set default line width to 0.75pt        

\begin{tikzpicture}[x=0.75pt,y=0.75pt,yscale=-1,xscale=1]
%uncomment if require: \path (0,235); %set diagram left start at 0, and has height of 235

%Shape: Circle [id:dp6781095359559848] 
\draw  [fill={rgb, 255:red, 0; green, 0; blue, 0 }  ,fill opacity=1 ] (130,105) .. controls (130,102.24) and (132.24,100) .. (135,100) .. controls (137.76,100) and (140,102.24) .. (140,105) .. controls (140,107.76) and (137.76,110) .. (135,110) .. controls (132.24,110) and (130,107.76) .. (130,105) -- cycle ;
%Shape: Circle [id:dp5884929003271584] 
\draw  [fill={rgb, 255:red, 0; green, 0; blue, 0 }  ,fill opacity=1 ] (200,105) .. controls (200,102.24) and (202.24,100) .. (205,100) .. controls (207.76,100) and (210,102.24) .. (210,105) .. controls (210,107.76) and (207.76,110) .. (205,110) .. controls (202.24,110) and (200,107.76) .. (200,105) -- cycle ;
%Straight Lines [id:da7001582911036106] 
\draw [line width=1.5]    (135,105) -- (205,105) ;
%Straight Lines [id:da3805853912603774] 
\draw [line width=1.5]    (135,105) -- (95,105) ;
%Shape: Circle [id:dp15098600447861554] 
\draw  [fill={rgb, 255:red, 0; green, 0; blue, 0 }  ,fill opacity=1 ] (437,105) .. controls (437,102.24) and (439.24,100) .. (442,100) .. controls (444.76,100) and (447,102.24) .. (447,105) .. controls (447,107.76) and (444.76,110) .. (442,110) .. controls (439.24,110) and (437,107.76) .. (437,105) -- cycle ;
%Shape: Circle [id:dp427621613162551] 
\draw  [fill={rgb, 255:red, 0; green, 0; blue, 0 }  ,fill opacity=1 ] (507,105) .. controls (507,102.24) and (509.24,100) .. (512,100) .. controls (514.76,100) and (517,102.24) .. (517,105) .. controls (517,107.76) and (514.76,110) .. (512,110) .. controls (509.24,110) and (507,107.76) .. (507,105) -- cycle ;
%Straight Lines [id:da8423306478603518] 
\draw [line width=1.5]    (442,105) -- (512,105) ;
%Shape: Circle [id:dp33986886963768237] 
\draw  [fill={rgb, 255:red, 0; green, 0; blue, 0 }  ,fill opacity=1 ] (367,105) .. controls (367,102.24) and (369.24,100) .. (372,100) .. controls (374.76,100) and (377,102.24) .. (377,105) .. controls (377,107.76) and (374.76,110) .. (372,110) .. controls (369.24,110) and (367,107.76) .. (367,105) -- cycle ;
%Shape: Circle [id:dp7885688341917547] 
\draw  [fill={rgb, 255:red, 0; green, 0; blue, 0 }  ,fill opacity=1 ] (577,105) .. controls (577,102.24) and (579.24,100) .. (582,100) .. controls (584.76,100) and (587,102.24) .. (587,105) .. controls (587,107.76) and (584.76,110) .. (582,110) .. controls (579.24,110) and (577,107.76) .. (577,105) -- cycle ;
%Straight Lines [id:da48512865516690185] 
\draw [line width=1.5]    (512,105) -- (582,105) ;
%Straight Lines [id:da3967878794273161] 
\draw [line width=1.5]    (372,105) -- (442,105) ;

% Text Node
\draw (86,98) node [anchor=north west][inner sep=0.75pt]  [font=\Large] [align=left] {$\times$};
\end{tikzpicture}
		\end{center}
					\caption{Left: the graph of orbits $\Gamma/\langle \nu \rangle$ of the biserial FBG in Example~\ref{exa:preproj-A3}, where the left edge is an orbifold edge; Right: the corresponding reduced form $\Gamma_{\red}=\widehat{\Gamma/\langle \nu \rangle}$.}
		\label{fig:red-form-preproj-A3}	
	\end{figure}
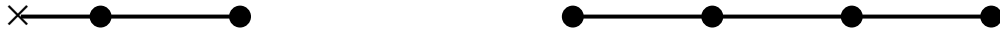

In particular, \cite{LL3} shows that the representation-finite biserial FBGAs coincide with the representation-finite self-injective algebras of type~A (for a more explicit characterization of these algebras, see \cite{Asa1,Asa2}). More precisely, representation-finite admissible biserial FBGAs coincide with algebras that are derived equivalent to self-injective Nakayama algebras, whereas representation-finite non-admissible biserial FBGAs coincide with algebras that are derived equivalent to M\"{o}bius\footnote{Here, the term ``M\"obius'' refers to the fact that the stable  Auslander--Reiten quiver of these algebras is of M\"obius type, following Riedtmann's notation \cite{Ried}. This terminology is not related to the surface models of biserial FBGAs.
} algebras.

In fact, we observe that the $(\Gamma/\langle \nu \rangle,d)$ arising from the definition of the reduced form in~\cite{LL3} is precisely the skew-Brauer graph (abbr. skew-BG) defined in~\cite{EGV,So2}, which can be used to define the corresponding skew-Brauer graph algebra (abbr. skew-BGA). Using the characterization of skew group algebras in~\cite{RR}, we show the following.

\begin{proposition}[see Proposition~\ref{prop:adm-skew=BGA} and \ref{prop:nonadm-skew=skew-BGA}]
Let $A$ be the biserial FBGA associated with a biserial FBG $(\Gamma,d)$, and let $\nu_A$ be the Nakayama automorphism of $A$ induced by the Nakayama automorphism $\nu$ of $(\Gamma,d)$. 
Let $B$ be the (skew-)BGA corresponding to the (skew-)BG $(\Gamma/\langle \nu \rangle,d)$.

\begin{enumerate}
\item If $(\Gamma,d)$ is admissible, then $B$ is a BGA, and the skew-group algebra $A\langle \nu_A\rangle$ is Morita equivalent to $B$.

\item If $(\Gamma,d)$ is non-admissible, then $B$ is a skew-BGA. Moreover, if $\operatorname{char} k \neq 2$, then the skew-group algebra $A\langle \nu_A\rangle$ is Morita equivalent to $B$.
\end{enumerate}
\end{proposition}

	Meanwhile, tilting theory (and more generally, silting theory) is a fundamental tool in the study of derived categories of finite-dimensional algebras \cite{AI,AIR}. Through such equivalences, one can compare homological properties, classify algebras up to derived equivalence, and analyze how their representation theory behave under these equivalences. This framework is especially interesting and often effective for classes of algebras admitting combinatorial models, where two-term silting complexes often have explicit descriptions; see, for example, \cite{AAC,AY,BC, CS}.
	
	We establish a bijective correspondence between two-term tilting complexes of biserial FBGAs and a certain subset of the two-term tilting complexes of their reduced forms.

\begin{theorem}[see Theorem~\ref{thm:adm-case} and \ref{thm:non-adm-case}]
Let $A$ be a biserial FBGA with associated biserial FBG $(\Gamma,d)$, and let $A_{\red}$ be the BGA associated with the reduced form $(\Gamma_{\red},d)$. 
\begin{enumerate}
	\item If $(\Gamma,d)$ is admissible, then there is a poset isomorphism between
\begin{enumerate}
  \item the set $2\text{-}\mathrm{tilt}\,A$ of isomorphism classes of basic two-term tilting complexes of $A$, and
  \item the set $2\text{-}\mathrm{tilt}\,A_{\red}$ of isomorphism classes of basic two-term tilting complexes of $A_{\red}$.
\end{enumerate}
If $(\Gamma,d)$ is non-admissible, then there is a poset isomorphism between
\begin{enumerate}
  \item the set $2\text{-}\mathrm{tilt}\,A$ of isomorphism classes of basic two-term tilting complexes of $A$, and
  \item the set $2\text{-}\mathrm{tilt}^{\phi}\,A_{\red}$ of isomorphism classes of basic $\phi$-stable  two-term tilting complexes of $A_{\red}$, where $\phi$ denotes the natural involution of $A_{\red}$.
\end{enumerate}

\item The following statements are equivalent:
\begin{enumerate}
  \item $A$ is tilting-discrete;
  \item $A_{\red}$ is tilting-discrete;
  \item $\Gamma_{\red}$ contains at most one odd cycle and no even cycles.
\end{enumerate}
\end{enumerate}
\end{theorem}

\noindent We note that the above result does not depend on the characteristic of the base field, since by \cite{AAC,AY}, two-term presilting complexes correspond to certain walks on the associated ribbon graph. Indeed, we show that there is a correspondence between signed walks on the two ribbon graphs $\Gamma$ and $\Gamma_{\red}$, which induces a correspondence between two-term tilting complexes.
If we further assume that $\operatorname{char} k \neq 2$, then the correspondence between the above two-term tilting complexes can be extended to a correspondence between $n$-term tilting complexes via the skew group algebra approach in~\cite{KKKMM}; see the proof of Theorem~\ref{thm:adm-case} and Remark~\ref{rmk:non-adm}.

	As noted in \cite{AD,AK}, it is interesting to find algebras that are tilting-discrete but not silting-discrete. In light of the above results, we can easily produce many self-injective algebras that are tilting-discrete but not silting-discrete; see Example~\ref{exa:fms-til-dis-not-sil-dis}

	Moreover, we describe irreducible tilting mutations of a biserial FBGA $A$ associated with $(\Gamma,d)$ in terms of Kauer moves (Figure~\ref{fig:gen-Kauer-move}) \cite{K,So1,So2}, using local moves both on its original ribbon graph $\Gamma$ (Figure~\ref{fig:1-KM-fms}) and on its graph of orbits $\Gamma / \langle \nu \rangle$ (Subsection \ref{subsec:move-on-red}).
More precisely, this follows the same rule as in the case of (skew-)BGAs: these mutations correspond to mutations of the (orbifold) ribbon graph (see Figures~\ref{fig:Kauer-moves} and~\ref{fig:Kauer-moves-orbifold}).

\begin{figure}[ht]
		\centering
		\begin{center}       
\tikzset{every picture/.style={line width=0.75pt}} %set default line width to 0.75pt        

\begin{tikzpicture}[x=0.75pt,y=0.75pt,yscale=-1,xscale=1]
%uncomment if require: \path (0,235); %set diagram left start at 0, and has height of 235

%Straight Lines [id:da26120890939224806] 
\draw [line width=1.5]    (50,70) -- (160,70) ;
%Straight Lines [id:da3109930443019404] 
\draw [line width=1.5]    (50,70) -- (29.5,119.5) ;
%Straight Lines [id:da608494482658684] 
\draw [color={rgb, 255:red, 0; green, 0; blue, 255 }  ,draw opacity=1 ][line width=1.5]    (50,70) -- (52,119.83) ;
%Straight Lines [id:da8630010536852626] 
\draw [color={rgb, 255:red, 0; green, 0; blue, 255 }  ,draw opacity=1 ][line width=1.5]    (50,70) -- (92.67,101.17) ;
%Straight Lines [id:da5735624814000166] 
\draw [line width=1.5]    (160,70) -- (179.5,117.5) ;
%Shape: Circle [id:dp944514451712889] 
\draw  [fill={rgb, 255:red, 0; green, 0; blue, 0 }  ,fill opacity=1 ] (45,70) .. controls (45,67.24) and (47.24,65) .. (50,65) .. controls (52.76,65) and (55,67.24) .. (55,70) .. controls (55,72.76) and (52.76,75) .. (50,75) .. controls (47.24,75) and (45,72.76) .. (45,70) -- cycle ;
%Shape: Circle [id:dp892927410071584] 
\draw  [fill={rgb, 255:red, 0; green, 0; blue, 0 }  ,fill opacity=1 ] (155,70) .. controls (155,67.24) and (157.24,65) .. (160,65) .. controls (162.76,65) and (165,67.24) .. (165,70) .. controls (165,72.76) and (162.76,75) .. (160,75) .. controls (157.24,75) and (155,72.76) .. (155,70) -- cycle ;
%Straight Lines [id:da23034143496433357] 
\draw [line width=1.5]    (381,70) -- (491,70) ;
%Straight Lines [id:da5260724095071576] 
\draw [line width=1.5]    (381,70) -- (364.5,119.5) ;
%Straight Lines [id:da6326141495525821] 
\draw [color={rgb, 255:red, 0; green, 0; blue, 255 }  ,draw opacity=1 ][line width=1.5]    (491,70) -- (491.5,120.5) ;
%Straight Lines [id:da8413689387795369] 
\draw [color={rgb, 255:red, 0; green, 0; blue, 255 }  ,draw opacity=1 ][line width=1.5]    (491,70) -- (447.5,101.5) ;
%Straight Lines [id:da42451166764865444] 
\draw [line width=1.5]    (491,70) -- (512.5,119.5) ;
%Shape: Circle [id:dp6280299551034454] 
\draw  [fill={rgb, 255:red, 0; green, 0; blue, 0 }  ,fill opacity=1 ] (376,70) .. controls (376,67.24) and (378.24,65) .. (381,65) .. controls (383.76,65) and (386,67.24) .. (386,70) .. controls (386,72.76) and (383.76,75) .. (381,75) .. controls (378.24,75) and (376,72.76) .. (376,70) -- cycle ;
%Shape: Circle [id:dp8326764224338483] 
\draw  [fill={rgb, 255:red, 0; green, 0; blue, 0 }  ,fill opacity=1 ] (486,70) .. controls (486,67.24) and (488.24,65) .. (491,65) .. controls (493.76,65) and (496,67.24) .. (496,70) .. controls (496,72.76) and (493.76,75) .. (491,75) .. controls (488.24,75) and (486,72.76) .. (486,70) -- cycle ;
%Shape: Arc [id:dp25834406834941737] 
\draw  [draw opacity=0][dash pattern={on 0.84pt off 2.51pt}] (41.49,82.35) .. controls (37.57,79.65) and (35,75.12) .. (35,70) .. controls (35,61.72) and (41.72,55) .. (50,55) .. controls (57.34,55) and (63.45,60.27) .. (64.75,67.24) -- (50,70) -- cycle ; \draw  [dash pattern={on 0.84pt off 2.51pt}] (41.49,82.35) .. controls (37.57,79.65) and (35,75.12) .. (35,70) .. controls (35,61.72) and (41.72,55) .. (50,55) .. controls (57.34,55) and (63.45,60.27) .. (64.75,67.24) ;  
%Shape: Arc [id:dp9050851315370134] 
\draw  [draw opacity=0][dash pattern={on 0.84pt off 2.51pt}] (145.05,68.76) .. controls (145.68,61.05) and (152.13,55) .. (160,55) .. controls (168.28,55) and (175,61.72) .. (175,70) .. controls (175,75.25) and (172.31,79.86) .. (168.23,82.55) -- (160,70) -- cycle ; \draw  [dash pattern={on 0.84pt off 2.51pt}] (145.05,68.76) .. controls (145.68,61.05) and (152.13,55) .. (160,55) .. controls (168.28,55) and (175,61.72) .. (175,70) .. controls (175,75.25) and (172.31,79.86) .. (168.23,82.55) ;  
%Shape: Arc [id:dp127017280311293] 
\draw  [draw opacity=0] (163.36,84.62) .. controls (162.28,84.87) and (161.15,85) .. (160,85) .. controls (152.6,85) and (146.45,79.64) .. (145.22,72.58) -- (160,70) -- cycle ; \draw   (163.36,84.62) .. controls (162.28,84.87) and (161.15,85) .. (160,85) .. controls (152.6,85) and (146.45,79.64) .. (145.22,72.58) ;  
\draw   (143.99,75.78) -- (145.3,72.33) -- (147.9,74.95) ;
%Shape: Arc [id:dp15340836121418344] 
\draw  [draw opacity=0][dash pattern={on 0.84pt off 2.51pt}] (375.19,83.83) .. controls (369.79,81.56) and (366,76.23) .. (366,70) .. controls (366,61.72) and (372.72,55) .. (381,55) .. controls (388.95,55) and (395.45,61.18) .. (395.97,68.99) -- (381,70) -- cycle ; \draw  [dash pattern={on 0.84pt off 2.51pt}] (375.19,83.83) .. controls (369.79,81.56) and (366,76.23) .. (366,70) .. controls (366,61.72) and (372.72,55) .. (381,55) .. controls (388.95,55) and (395.45,61.18) .. (395.97,68.99) ;  
%Shape: Arc [id:dp17858050004005088] 
\draw  [draw opacity=0][dash pattern={on 0.84pt off 2.51pt}] (60.99,80.21) .. controls (58.73,82.64) and (55.67,84.32) .. (52.23,84.83) -- (50,70) -- cycle ; \draw  [color={rgb, 255:red, 0; green, 0; blue, 255 }  ,draw opacity=1 ][dash pattern={on 0.84pt off 2.51pt}] (60.99,80.21) .. controls (58.73,82.64) and (55.67,84.32) .. (52.23,84.83) ;  
%Shape: Arc [id:dp7307018410450798] 
\draw  [draw opacity=0] (64.87,72) .. controls (64.6,73.97) and (63.96,75.81) .. (63.02,77.46) -- (50,70) -- cycle ; \draw   (64.87,72) .. controls (64.6,73.97) and (63.96,75.81) .. (63.02,77.46) ;  
\draw   (65.94,75.58) -- (62.86,77.62) -- (62.27,73.98) ;
%Shape: Arc [id:dp3863727068943652] 
\draw  [draw opacity=0] (48.9,84.96) .. controls (47.92,84.89) and (46.96,84.72) .. (46.04,84.47) -- (50,70) -- cycle ; \draw   (48.9,84.96) .. controls (47.92,84.89) and (46.96,84.72) .. (46.04,84.47) ;  
\draw   (47.91,86.71) -- (45.38,84.03) -- (48.87,82.83) ;
%Shape: Arc [id:dp677423576252286] 
\draw  [draw opacity=0] (395.88,71.87) .. controls (394.96,79.27) and (388.65,85) .. (381,85) .. controls (380.03,85) and (379.08,84.91) .. (378.16,84.73) -- (381,70) -- cycle ; \draw   (395.88,71.87) .. controls (394.96,79.27) and (388.65,85) .. (381,85) .. controls (380.03,85) and (379.08,84.91) .. (378.16,84.73) ;  
\draw   (380.99,87.18) -- (378.18,84.8) -- (381.51,83.22) ;
%Shape: Arc [id:dp2868395774633654] 
\draw  [draw opacity=0][dash pattern={on 0.84pt off 2.51pt}] (476.15,67.84) .. controls (477.2,60.58) and (483.45,55) .. (491,55) .. controls (499.28,55) and (506,61.72) .. (506,70) .. controls (506,75.72) and (502.79,80.7) .. (498.08,83.23) -- (491,70) -- cycle ; \draw  [dash pattern={on 0.84pt off 2.51pt}] (476.15,67.84) .. controls (477.2,60.58) and (483.45,55) .. (491,55) .. controls (499.28,55) and (506,61.72) .. (506,70) .. controls (506,75.72) and (502.79,80.7) .. (498.08,83.23) ;  
%Shape: Arc [id:dp650702956885336] 
\draw  [draw opacity=0][dash pattern={on 0.84pt off 2.51pt}] (489.41,84.92) .. controls (485.17,84.47) and (481.45,82.25) .. (479.01,79.02) -- (491,70) -- cycle ; \draw  [color={rgb, 255:red, 0; green, 0; blue, 255 }  ,draw opacity=1 ][dash pattern={on 0.84pt off 2.51pt}] (489.41,84.92) .. controls (485.17,84.47) and (481.45,82.25) .. (479.01,79.02) ;  
%Shape: Arc [id:dp9936274571620323] 
\draw  [draw opacity=0] (478.23,77.87) .. controls (477.02,75.91) and (476.25,73.66) .. (476.05,71.23) -- (491,70) -- cycle ; \draw   (478.23,77.87) .. controls (477.02,75.91) and (476.25,73.66) .. (476.05,71.23) ;  
\draw   (474.86,74.7) -- (476.13,71.24) -- (478.76,73.82) ;
%Shape: Arc [id:dp5730115688555413] 
\draw  [draw opacity=0] (496.15,84.09) .. controls (495.01,84.51) and (493.8,84.79) .. (492.54,84.92) -- (491,70) -- cycle ; \draw   (496.15,84.09) .. controls (495.01,84.51) and (493.8,84.79) .. (492.54,84.92) ;  
\draw   (496.11,86.17) -- (492.69,84.79) -- (495.37,82.25) ;

% Text Node
\draw (254,58) node [anchor=north west][inner sep=0.75pt]  [font=\LARGE] [align=left] {$ \mathrel{\substack{
\overset{\mu_P^-}{\longrightarrow}\\[-0.3ex]
\underset{\mu_{P}^+}{\longleftarrow}
}} $};
\end{tikzpicture}
		\end{center}
		\caption{The Kauer move of a cyclically consecutive set of half-edges (in blue).}
		\label{fig:gen-Kauer-move}	
	\end{figure}
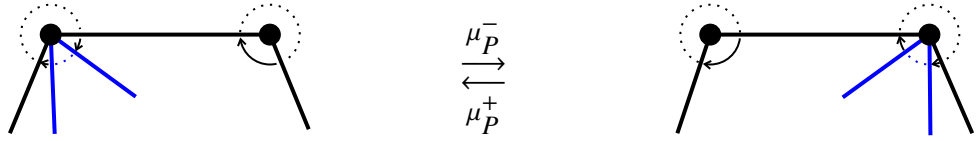

	Since BGAs are closed under derived equivalence \cite{AZ}, it is natural to have the following conjecture.
	\medskip
	
	\noindent\textbf{Conjecture 1.3.}\; If a basic algebra $A$ is derived equivalent to a biserial FBGA, then $A$ is also a biserial FBGA.
		\medskip

	\noindent We show that this conjecture hold for tilting-discrete cases in Corollary \ref{cor:der-closed-tilt-dis}.

	\medskip
	\textbf{Outline.}\; In Section~\ref{sec:BGA}, we review fundamental concepts and key results on biserial FBGAs. 
In Section~\ref{sec:skew-BGA-and-bfbga}, we recall the definitions of skew group algebras and skew-BGAs, and show that the skew group algebras of biserial FBGAs induced by the Nakayama automorphism are in fact skew-BGAs. 
In Section~\ref{sec:kauer-move}, we review silting theory and Kauer moves for BGAs. We also study two-term silting complexes of biserial FBGAs in several special cases. Furthermore, we describe irreducible tilting mutations of a biserial FBGA $A$ in terms of Kauer moves. 
In Section~\ref{sec:2-term-tilt}, we construct a bijection between two-term tilting complexes of a biserial FBGA $A$ and $\phi$-stable two-term tilting complexes of its reduced form $A_{\mathrm{red}}$. We also give necessary and sufficient conditions for a biserial FBGA to be tilting-discrete. 
In the appendix, we provide step-by-step verifications of the tilting mutations, which are straightforward but lengthy. 

	\section*{Acknowledgments}
	The author sincerely thanks Aaron Chan for his supervision at Nagoya University, many fruitful discussions, and valuable suggestions and comments on this work. The author is very grateful to Dong Yang for sharing his lecture notes and for patiently answering his questions on silting theory. The author also thanks Nengqun Li, Yuming Liu and Zhengfang Wang for helpful discussions. This work is supported by the China Scholarship Council (No. 202506040127).

	\section*{Notation and Conventions}
	
	Throughout we assume that $k$ is an algebraically closed field and all algebras considered are finite-dimensional $k$-algebras. Unless stated otherwise, all modules will be finitely generated right modules. Furthermore, we say that $A$ is a bound quiver algebra, if $A\cong \qa$, where $Q$ is a finite quiver and $I$ is an admissible ideal in the path algebra $kQ$. We denote by $s(p)$ the source vertex of a path $p$ and by $t(p)$ its terminus vertex. We will write paths from left to right, for example, $p=\alpha_{1}\alpha_{2}\cdots\alpha_{n}$ is a path with starting arrow $\alpha_{1}$ and ending arrow $\alpha_{n}$. A path is called a cycle if $s(p)=t(p)$.
	By abuse of notation we sometimes view an element in $kQ$ as an element in the quotient $\qa$ if no confusion can arise.

	%In this paper, we study the indecomposable $A$-modules $M$ via their Loewy structure, which is represented by a diagram where the $i$-th row corresponds to the simple summands of the completely reducible module $\rad^{i-1}(A)M/\rad^{i}(A)M$ with $\rad(A)$ the Jacobson radical of $A$. Each number in the diagram denotes a distinct simple module in $A$. For further details, see for example in \cite[Page 174]{Ben}.

	Recall that the {\it derived category} $\mathcal{D}(A)$ of an algebra $A$ is obtained from the homotopy category $\mathcal{K}(A)$ by inverting quasi-isomorphisms. Two algebras are said to be {\it derived equivalent} if their derived categories are equivalent as triangulated categories. By Rickard's theorem~\cite{Ric}, this is equivalent to an equivalence between their bounded homotopy categories $\mathcal{K}^b(\mathrm{proj}\,A)$ of finitely generated projective modules.
	
	A $k$-algebra $A$ is called {\it self-injective} if $A_A$ is an injective $A$-module; and {\it symmetric} if $_AA_A\cong DA$ as $A$-$A$-bimodules. Indeed, symmetric algebras are naturally self-injective (see for example in \cite{Z}).
	
	A $k$-algebra $A$ is {\it special biserial} if it is isomorphic to an algebra of the form $\qa$ where $kQ$ is a path algebra and $I$ is an admissible ideal such that the following properties hold.
	
	\begin{enumerate}[(1)]
		\item At every vertex $i$ in $Q$, there are at most two arrows starting at $i$ and there are at most two arrows ending at $i$.
		
		\item For every arrow $\alpha$ in $Q$, there exists at most one arrow $\beta$ such that $\beta\alpha\notin I$ and there exists at most one arrow $\gamma$ such that $\alpha\gamma\notin I$.
	\end{enumerate}

	\section{Biserial fractional Brauer graph algebras}\label{sec:BGA}

	\subsection{Ribbon graphs}
	\
	
	Ribbon graphs combinatorially encode the structure of orientable surfaces (see for example in \cite[Section 1.1]{OZ}). A key feature of ribbon graphs is the cyclic ordering of (half-)edges at each vertex, which captures the orientation data of the underlying surface. We begin this section by recalling their formal definition.
	
	\begin{definition}\label{def:ribbon-graph}
		A {\it ribbon graph} is a tuple $\Gamma=(V,H,s,\iota,\rho)$, where
		\begin{enumerate}[(1)]
			\item $V$ (also denoted by $V(\Gamma)$) is a finite set whose elements are called vertices;
			
			\item $H$ (also denoted by $H(\Gamma)$) is a finite set whose elements are called half-edges;
			
			\item $s: H\rightarrow V$ is a functions;
			
			\item $\iota: H\rightarrow H$ is an involution without fixed points;
			
			\item $\rho: H\rightarrow H$ is a permutation whose cycles correspond to the sets $H_\vv:=s^{-1}(\vv)$, $\vv\in V$.
			
		\end{enumerate}
		The $\iota$-orbits are called the edges of $\Gamma$. In particular, if we set
\[
E(\Gamma) := H/\langle \iota \rangle,
\]
then $(V(\Gamma), E(\Gamma))$ is the underlying combinatorial graph of $\Gamma$.
	\end{definition}
	
	For a ribbon graph $\Gamma=(V,H,s,\iota,\rho)$, we introduce the following notation. 
For each half-edge $h \in H$, we write
\[
h^+ := \rho(h) \quad \text{and} \quad h^- := \rho^{-1}(h)
\]
for the successor and predecessor of $h$, respectively. We denote by
\[
\bar{h} := \{h, \iota(h)\}
\]
the edge associated with $h$ in the underlying combinatorial graph of $\Gamma$.
For each vertex $\vv \in V$, the valency of $\vv$ is defined by
\[
val(\vv) := \bigl|\{h \in H \mid s(h) = \vv\}\bigr|.
\]
In particular, a loop (that is, an edge $\{h, \iota(h)\}$ with $s(h) = s(\iota(h))$) contributes twice to $val(\vv)$.

	Unless stated otherwise, we will assume that $\Gamma$ is connected, i.e. its underlying graph is connected.

	\subsection{Biserial fractional Brauer graph algebras}\label{subsec:def-fms-BGA}
	\
	
	In this section, we review some basic knowledge about biserial fractional Brauer graph algebras, introduced in \cite{LL}. 
	
	\begin{definition}\textnormal{(cf. \cite[Definition 2.1]{LL3})}\label{fms-BG}
		A \textit{biserial fractional Brauer graph} (abbr. biserial FBG) is a pair $(\Gamma,d)$ consisting of a ribbon graph $\Gamma$ together with a {\it degree function} $d : V(\Gamma) \to \mathbb{Z}_{>0}$ (whose values are referred to as the {\it degrees}), such that for each half-edge $h \in H(\Gamma)$,
		\begin{equation}
\iota\!\left(\rho^{d(s(h))}(h)\right) \;=\; \rho^{d(s(\iota(h)))}(\iota(h)).
\tag{SI}
\end{equation}
	\end{definition}

We often omit $d$ from the notation and simply refer to $\Gamma$ as a biserial FBG when no confusion arises. For a vertex $\vv$, we define the {\it multiplicity} of $\vv$ (also called the fractional degree in \cite{LL}) to be the rational number
\[
m(\vv) := \frac{d(\vv)}{val(\vv)}.
\]
It follows immediately from the definition that a biserial FBG is a (classical) Brauer graph (abbr. BG) (for an explicit definition, see for example in \cite{SS,OZ}) if and only if the multiplicity of every vertex is an integer. A vertex $\vv$ is called {\it truncated} if $d(\vv)=1$.
	Denote by $\nu$ the map
\begin{equation*}
\nu: H \longrightarrow H, \qquad
h \longmapsto \rho^{d(s(h))}h,
\end{equation*}
which is called the {\it Nakayama automorphism} of $(\Gamma,d)$.
	
	To each biserial FBG $(\Gamma,d)$, one can associate a ($2$-regular) quiver $Q=Q_\Gamma$ and an ideal of relations $I=I_{\Gamma,d}$ in the path algebra $kQ$ as follows.
	
	\begin{enumerate}[(1)]
		\item The vertices of $Q$ correspond to the edges of $\Gamma$ and for every $h\in H$, there is an arrow $\alpha_h:\bar{h}\rightarrow\overline{h^+}$. The assignment $\alpha_h\mapsto \alpha_{h^+}$ defines a permutation $\rho=\rho_\Gamma$ of the arrows of $Q$ whose orbits are in bijection with vertices in $\Gamma$. If $\alpha=\alpha_h$, set $d(\alpha):=d(s(h))$.
		
		\item The ideal $I$ is generated by the following set of relations:
		\begin{enumerate}[(i)]
			\item $$\alpha\rho(\alpha)\cdots\rho^{d(\alpha)-1}(\alpha)=\beta\rho(\beta)\cdots\rho^{d(\beta)-1}(\beta),$$
			where $\alpha,\beta\in Q_1$ and $s(\alpha)=s(\beta)$, that is $\alpha$ and $\beta$ start at the same edge of $\Gamma$.
			
			\item $$\alpha\beta=0,$$
			where $\alpha,\beta\in Q_1$ are composable and $\rho(\alpha)\neq \beta$;
		\end{enumerate}
	\end{enumerate}
	
	\begin{definition}\textnormal{(cf. \cite[Definition 2.2 and Proposition 5.5]{LL3})}\label{fms-BGA}
		A $k$-algebra $A$ is called a {\it biserial fractional Brauer graph algebra} (abbr. biserial FBGA) if there exists a biserial FBG $(\Gamma,d)$ such that $A\cong kQ_\Gamma/I_{\Gamma,d}$ as $k$-algebras.
	\end{definition}

	\begin{remark}\label{remark:truncated}
\begin{enumerate}[(1)]
\item
For each edge connected with some truncated vertex in $\Gamma$, the first type of relation becomes non-admissible and has the form $\alpha \rho(\alpha)\cdots \rho^{d(\alpha)-1}(\alpha)=\beta$.
We usually omit this relation by deleting the arrow $\beta$ from the quiver $Q_\Gamma$.

\item
By the preceding remark, and as in the case of BGAs (see, for example, \cite{SS}), a biserial FBGA $A = kQ_{\mathrm{adm}}/I_{\mathrm{adm}}$ can equivalently be defined in terms of its Gabriel quiver $Q_{\mathrm{adm}}$ and an admissible ideal $I_{\mathrm{adm}}$. The quiver $Q_{\mathrm{adm}}$ is obtained from $Q_\Gamma$ by deleting all arrows around truncated vertices. The ideal $I_{\mathrm{adm}}$ is generated by the following three types of relations:
\begin{enumerate}
\item[\textnormal{(I)}]
$$
\alpha\rho(\alpha)\cdots\rho^{d(\alpha)-1}(\alpha)
=
\beta\rho(\beta)\cdots\rho^{d(\beta)-1}(\beta),
$$
whenever $s(\alpha)=s(\beta)=\bar h$ and both vertices incident to the edge $\bar h$ are non-truncated;

\item[\textnormal{(II)}]
$$
\alpha\rho(\alpha)\cdots\rho^{d(\alpha)}(\alpha)=0,
$$
for any arrow $\alpha$ in $Q_{\mathrm{adm}}$;

\item[\textnormal{(III)}]
$$
\alpha\beta=0
$$
for all paths $\alpha\beta$ with $\beta\neq \rho(\alpha)$.
\end{enumerate}
\end{enumerate}
\end{remark}	

	Note that $A$ is indecomposable as a ring if and only if $\Gamma$ is connected, and $A$ is a (classical) Brauer graph algebra (abbr. BGA) if and only if $m(\vv)$ is an integer for every vertex $\vv \in V(\Gamma)$.
	Moreover, biserial FBGAs satisfy the following elementary properties; see~\cite{LL} if one needs details.

	\begin{proposition}\label{prop:bfBGA-basic}
Let $A$ be a biserial FBGA with associated biserial FBG $(\Gamma,d)$, and let $\nu$ be the Nakayama automorphism of $(\Gamma,d)$. Then $A$ is self-injective and special biserial.
Moreover, the Nakayama automorphism of $A$ is given by the map
\begin{equation*}
\nu_A \colon A \to A,\qquad 
\nu_A(\bar{h})=\overline{\rho^{-d(s(h))}(h)}, \quad 
\nu_A(\alpha_h)=\rho^{-d(s(h))}(\alpha_h),
\end{equation*}
which is induced by the Nakayama automorphism of $(\Gamma,d)$. 
Here $\bar{h}$ is identified with the primitive idempotent $e_{\bar{h}}$.
\end{proposition}

	\subsection{Some numerical invariants associated to vertices}\label{subsec:numbers-of-bFBGA}
	\
	
	We begin by defining some numerical invariants associated to the vertices of biserial FBGs.

\begin{definition}
Let $\Gamma = (\Gamma, d)$ be a biserial FBG with $\nu$ the Nakayama automorphism of $\Gamma$.
\begin{enumerate}
    \item For each vertex $\vv \in V$, denote by $o(\vv)$ the number of $\langle \nu \rangle$-orbits of half-edges incident with the vertex $\vv$.
    \item Let $G_\vv$ be any $\langle \nu \rangle$-orbit of half-edges around $\vv$. We set $n(\vv)=|G_\vv|$, which is independent of the choice of $G_\vv$.
\end{enumerate}
\end{definition}

By definition, for a biserial FBG $(\Gamma,d)$, both the Nakayama permutation $\nu$ and the function $o$ are determined by the degree function $d$. Moreover, we have $$val(\vv) = o(\vv)\,n(\vv).$$ We also note that $n(\vv)=1$ for every $\vv \in V(\Gamma)$ if $\Gamma$ is a BG.

\begin{lemma}\label{lem:n(v)-property}
	If there is an edge connecting $\vv_1$ and $\vv_2$, then $n(\vv_1) = n(\vv_2)$. In fact, for a connected biserial FBG $(\Gamma, d)$, $n(\vv)$ coincides with the order of the cyclic group $\langle \nu \rangle$ generated by the Nakayama automorphism.
\end{lemma}

\begin{proof}
	Let $\{h_1, \ldots, h_n\}$ be a $\langle \nu \rangle$-orbit of half-edges around $\vv_1$, such that for each $h_i$, we have $s(\iota(h_i)) = \vv_2$. By the definition of a biserial FBG, the set $\{\iota(h_1), \ldots, \iota(h_n)\}$ forms a $\langle \nu \rangle$-orbit of half-edges around $\vv_2$. Hence $n(\vv_1) = n(\vv_2) = n$. By the definition of the Nakayama automorphism of $(\Gamma,d)$, $\nu$ acts as a cyclic permutation of length $n(\vv)$ on $G_\vv$. Hence $\nu^{n(\vv)}=\mathrm{id}$. Therefore, $n(\vv)$ coincides with the order of the cyclic group $\langle \nu \rangle$ generated by the Nakayama automorphism.
\end{proof}

Suppose $s(h_1)=\vv=s(h_2)$. Then we call the sequence
\[
h_1 = h^0 \xrightarrow{\rho} h^1 \xrightarrow{\rho} \cdots \xrightarrow{\rho} h^{k-1} \xrightarrow{\rho} h^k = h_2,
\]
with $k \le val(\vv)$, of consecutive half-edges a {\it fan}.
Moreover, if $h_1$ and $h_2$ lie in the same $\langle \nu \rangle$-orbit and the half-edges $h^0,\ldots,h^{k-1}$ lie in distinct $\langle \nu \rangle$-orbits, then we call such a fan $F$ a {\it $\nu$-fan}. In this case, we have $k=o(\vv)$.

\begin{lemma}
	For all $\vv \in V(\Gamma)$, $d(\vv)$ is divisible by $o(\vv)$, and $F(\vv) := \dfrac{d(\vv)}{o(\vv)} \in \mathbb{Z}$ equals the number of $\nu$-fans encountered under the action of the Nakayama automorphism $\nu$, when moving from a half-edge $h$ to $\nu(h)$.
\end{lemma}

\begin{proof}
	Consider a half-edge $h \in H$ such that $s(h) = \vv$. Since $h$ and $\rho^{d(\vv)}(h)$ lie in the same $\langle \nu \rangle$-orbit, the angle from $h$ to $\rho^{d(\vv)}(h)$ is always an integer multiple of the $\nu$-fan around $\vv$. This integer, which equals $\dfrac{d(\vv)}{o(\vv)}$, does not depend on the choice of $h$, because both $d(\vv)$ and $o(\vv)$ are fixed for $\vv$.
\end{proof}
 
The above definitions yield an equivalent description of the multiplicity function as follows.

\begin{proposition}\label{prop:multiplicity-eqv-def}
	For each vertex $\vv \in V$, we have
	\[
	m(\vv) = \dfrac{F(\vv)}{n(\vv)}.
	\]
\end{proposition}

\begin{proof}
	By the definition of the multiplicity,
	\[
	m(\vv) = \dfrac{d(\vv)}{val(\vv)} 
	= \dfrac{o(\vv)\,F(\vv)}{o(\vv)\,n(\vv)} 
	= \dfrac{F(\vv)}{n(\vv)}.
	\]
\end{proof}

\subsection{Reduced forms of biserial FBGAs}\label{subsec:reduced-form}
\

Roughly speaking, the reduced form of a biserial FBG is essentially given by the graph of orbits $\Gamma / \langle \nu \rangle$. However, since $\Gamma / \langle \nu \rangle$ is in general not a ribbon graph, but an orbifold ribbon graph.

\begin{definition}\label{def:orbifold}
	An {\it orbifold ribbon graph} is a tuple $\Gamma=(V,H,s,\iota,\rho)$ satisfying conditions {\normalfont(1)}--{\normalfont(5)} of a ribbon graph in Definition \ref{def:ribbon-graph}, except that the involution $\iota$ is allowed to have fixed points. We denote by $H_\times \subseteq H$ the set of fixed points of $\iota$. The elements of $H_\times$ are called {\it orbifold edges}. Each orbifold edge is drawn as an edge connecting a vertex in $V$ to a cross symbol~$\times$, where the symbol~$\times$ is not regarded as a vertex of $V$.
\end{definition}

Just as a ribbon graph naturally fattens to an orientable ribbon surface, an orbifold ribbon graph naturally fattens to an orbifold ribbon surface \cite{AP2,BSW}. We will show that an orbifold ribbon graph can be obtained as a quotient of a biserial FBG by its Nakayama automorphism. 

For each biserial FBG $(\Gamma = (V, H, s, \iota, \rho), d)$, define 
\[
(\Gamma / \langle \nu \rangle,d) := ((V, H', s', \iota', \rho'),d),
\]
where $V$ coincides with the same vertex set of $\Gamma$ and $d$ is the same function defined on $V$, and 
\begin{enumerate}
    \item $H' = H / \langle \nu \rangle = \{ [h] \mid h \in H \}$ is the set of $\langle \nu \rangle$-orbits in $H$;
    \item $s'([h]) = s(h)$;
    \item $\iota'([h]) = [\iota(h)]$;
    \item $\rho'([h]) = [\rho(h)]$.
\end{enumerate}
In fact, $\Gamma / \langle \nu \rangle$ is an orbifold ribbon graph. 
If $\Gamma/\langle\nu\rangle$ has no orbifold edges, then the biserial FBG $(\Gamma,d)$, equivalently the associated biserial FBGA $A$, is called {\it admissible}. Otherwise, $(\Gamma,d)$ and $A$ are called {\it non-admissible}. This terminology follows from the notion of admissibility of the action of $\langle\nu\rangle$ on the set of half-edges; see~\cite{LL3}.
By definition, each edge or orbifold edge in $\Gamma / \langle \nu \rangle$ corresponds to a $\langle \nu_A \rangle$-orbit of simple modules in $A$.

Now we construct a ribbon graph from an orbifold ribbon graph.
Let $\Gamma=(V,H,s,\iota,\rho)$ be an orbifold ribbon graph.
Define $\widehat{\Gamma}
:= (\widehat{V},\widehat{H},\widehat{s},\widehat{\iota},\widehat{\rho})$ as follows.

\begin{enumerate}
    \item If $\iota$ has no fixed point, then let 
    \[
        \widehat{\Gamma} = \Gamma.
    \]
    
    \item If $\iota$ has some fixed points, then:
        \begin{enumerate}
            \item $\widehat{V} = V_1 \sqcup V_2$, where $V_1$ and $V_2$ are two copies of $V$. For $\vv \in V$, we denote by $\vv_i$ the element of $V_i$ corresponding to $\vv$.

            \item $\widehat{H} = H_1 \sqcup H_2$, where $H_1$ and $H_2$ are two copies of $H$. For $h \in H$, we denote by $h_i$ the element of $H_i$ corresponding to $h$.
            
            \item $\widehat{s}(h_i) = s(h)$.
            
            \item 
            \[
            \widehat{\iota}(h_i) =
            \begin{cases}
                \iota(h)_i, & \text{if } \iota(h) \neq h, \\[2mm]
                h_{3-i}, & \text{if } \iota(h) = h.
            \end{cases}
            \]
            
            \item $\widehat{\rho}(h_i) = \rho(h)_i$.
        \end{enumerate}
\end{enumerate}
In the case where $\iota$ has fixed points, there exists a natural involution $\phi$ on $\widehat{\Gamma}$ which is given by $\vv_1\mapsto \vv_2$ and $h_1\mapsto h_2$.

Now, for each biserial FBG $(\Gamma,d)$, we define its {\it reduced form} to be
$$
\Gamma_{\mathrm{red}} := (\widehat{\Gamma / \langle \nu \rangle}, d'),
$$
where $d'$ is the function on the vertex set of $\widehat{\Gamma / \langle \nu \rangle}$ naturally induced by $d$, that is,
$d'(\vv_i) = d(\vv)$ for every $\vv_i$ corresponding to the vertex $\vv$ of $\Gamma$.
By construction, for each vertex $\vv$ in $\Gamma / \langle \nu \rangle$ we have $n(\vv) = 1$, that is each $\langle\nu\rangle$-orbit has exactly one element, the multiplicity of each vertex $\vv$ is given by
\[
m(\vv) = \frac{F(\vv)}{n(\vv)} = F(\vv) \in \mathbb{Z}_+.
\]
Therefore, this is a BG.

For the following statement, we recall that a Brauer graph is called a {\it Brauer tree} if its associated ribbon graph $\Gamma$ is a tree and the multiplicity function $m$ assigns the value $1$ to all vertices, except possibly for a single vertex $\vv$, called the {\it exceptional vertex}, for which $m(\vv)$ is referred to as the {\it exceptional multiplicity}. The BGA associated with a Brauer tree is called a {\it Brauer tree algebra}. In fact, a BGA is representation-finite if and only if it is a Brauer tree algebra (see for example \cite{SS}).

Let $A$ denote the biserial FBGA associated with $(\Gamma,d)$. We define $A_{\mathrm{red}}$ to be the BGA associated with $(\Gamma_{\mathrm{red}},d)$, which we also call the {\it reduced form} of $A$. It is shown in~\cite[Theorem~2.29]{LL3} that a biserial FBGA is representation-finite (resp.\ domestic) if and only if its reduced form is representation-finite (resp.\ domestic). 

Moreover, in the representation-finite case, although we do not make use of the following theorem in this paper, we include it to help the reader better understand which classical algebras biserial FBGAs correspond to. For a more explicit description of representation-finite self-injective algebras of type~A, we refer the reader to \cite{Ried,Asa1,Asa2}, where the necessary terminology and background can be found.

\begin{theorem}\textnormal{(cf. \cite[Theorem 4.7 and Theorem 4.8]{LL3})}\label{thm:rep-fin-fms}
Let $A$ be a biserial FBGA with associated biserial FBG $(\Gamma,d)$. Then the following are equivalent:
\begin{enumerate}[(1)]
    \item $A$ is representation-finite.

    \item $A$ is a representation-finite self-injective algebra of type~A.

    \item The reduced form $A_{\mathrm{red}}$ is a Brauer tree algebra with $n$ edges and exceptional multiplicity $m$.
\end{enumerate}
When these equivalent conditions hold, one of the following cases occurs:
\begin{enumerate}[(1)]
    \item $\Gamma / \langle \nu \rangle$ is a Brauer tree. In this case, the stable Auslander--Reiten quiver of $A$ is isomorphic to $     \mathbb{Z}A_{mn}/\langle \tau^{nr}\rangle$
    for some positive integer $r$, where $\tau$ denotes the translation of $\mathbb{Z}A_{mn}$. Moreover, $A$ is stably equivalent to a self-injective Nakayama algebra.

    \item $\Gamma / \langle \nu \rangle$ is a tree with a unique orbifold edge, and every vertex has multiplicity~$1$. In this case, the stable Auslander--Reiten quiver of $A$ is isomorphic to
    $
        \mathbb{Z}A_n/\langle \tau^{nr}\phi\rangle
    $
    for some positive integer $r$, where $\phi$ denotes the involution of $\mathbb{Z}A_n$. Moreover, $A$ is stably equivalent to a Möbius algebra.
\end{enumerate}
\end{theorem}

\section{Skew group algebras and biserial FBGAs}\label{sec:skew-BGA-and-bfbga}

By Proposition~\ref{prop:bfBGA-basic}, for any biserial FBGA $A$, the Nakayama automorphism $\nu_A$ is naturally induced by a quiver morphism of its quiver. Hence the action of the cyclic group $G=\langle\nu_A\rangle$ on $A$ gives rise naturally to the skew group algebra $AG$. In this section, we investigate the relationship among $A$, $A_{\mathrm{red}}$, and $AG$.

\subsection{Skew group algebras and skew-BGAs}
\

We first recall the definition of skew group algebras.

\begin{definition}\label{def:skew group-alg}
	Let $A$ be a $k$-algebra and let a finite group $G$ act on $A$ by algebra automorphisms. The {\it skew group algebra} $A G$ is defined as follows.

As a $k$-vector space,
\[
AG = A\otimes_k kG,
\]
where $kG$ is the group algebra of $G$. We write elements as $ag$ for $a\in A$ and $g\in G$. The multiplication is given by
\[
(a\otimes g)(b\otimes h)=ag(b)\otimes gh
\]
for all $a,b\in A$ and $g,h\in G$.
\end{definition}

By applying skew group algebras to specific algebras, analogous to the generalization from gentle algebras to skew-gentle algebras, a $\mathbb{Z}/2\mathbb{Z}$ action on BGA is used in \cite{EGV,So2} to define skew-BGAs, which are algebras defined via orbifold ribbon graphs together with an (integer-valued) multiplicity function. We now recall their definitions.

\begin{definition}\textnormal{(cf. \cite[Definition 4.1]{EGV} and \cite[Definition 3.1]{So2})}\label{skew-BG}
		A \textit{skew-Brauer graph} (abbr. skew-BG) is a pair $(\Gamma,d)$ consisting of an orbifold ribbon graph $\Gamma$ (see Definition \ref{def:orbifold}) together with a {\it degree function} $d : V(\Gamma) \to \mathbb{Z}_{>0}$ (whose values are referred to as the {\it degrees}), such that for each half-edge $v \in V(\Gamma)$, $val(v)\mid d(v)$.
\end{definition}

\begin{remark}
	We note that our definition of skew-BGs differs slightly from that in~\cite{EGV}. In~\cite[Definition 4.1]{EGV}, an orbifold edge is regarded as an edge connecting an ordinary vertex and a {\it distinguished vertex} (corresponding to our symbol \(\times\)). These distinguished vertices are viewed as vertices of multiplicity~one. In contrast, we regard an orbifold edge simply as a half-edge, and define multiplicities only on the ordinary vertices.
\end{remark}

To each skew-BG $(\Gamma,d)$, one can also associate a quiver $Q=Q_\Gamma$ and an ideal of relations $I=I_{\Gamma,d}$ in the path algebra $kQ$ as follows.
	
	\begin{enumerate}[(1)]
		\item For the edge set $E(\Gamma)=H/\langle \iota\rangle$, there is a natural disjoint decomposition
\[
E_\circ \sqcup E_\times,
\]
where $E_\circ$ consists of the $\iota$-orbits containing two distinct half-edges, while $E_\times$ consists of half-edges fixed by $\iota$. This induces, in a natural way, a partition of the half-edge set
\[
H = H_\circ \sqcup H_\times,
\]
where $H_\circ$ consists of half-edges belonging to edges in $E_\circ$, and $H_\times$ consists of half-edges fixed by $\iota$. The vertex set $Q_0$ of the quiver $Q$ is given by
\[
Q_0:=E_\circ \sqcup E_\times^+ \sqcup E_\times^-,
\]
where $E_\times^+$ and $E_\times^-$ are two copies of $E_\times$. For each $e\in E_\times$, we denote by $e^+\in E_\times^+$ and $e^-\in E_\times^-$ the corresponding elements.

\item The set of arrows of the quiver $Q$ is determined by the following rule. Consider two half-edges $h_1,h_2\in H(\Gamma)$ such that $\rho(h_1)=h_2$, and let $e_1=\overline{h_1}$ and $e_2=\overline{h_2}$ be the corresponding edges in $\Gamma$.

\begin{enumerate}
\item If both $e_1$ and $e_2$ do not lie in $E_\times$, then there is a unique arrow in $Q$,
\[
\alpha_{h_1}: e_1 \to e_2.
\]

\item If $e_1\in E_\times$ and $e_2\in E_\circ$, then there are two arrows in $Q$,
\[
{}^{+}\alpha_{h_1}: e_1^+ \to e_2
\quad \text{and} \quad
{}^{-}\alpha_{h_1}: e_1^- \to e_2.
\]

\item If $e_1\in E_\circ$ and $e_2\in E_\times$, then there are two arrows in $Q$,
\[
\alpha_{h_1}^{+}: e_1 \to e_2^+
\quad \text{and} \quad
\alpha_{h_1}^{-}: e_1 \to e_2^-.
\]

\item If both $e_1$ and $e_2$ lie in $E_\times$, then there are four corresponding arrows in $Q$, denoted by
\[
{}^{i}\alpha_{h_1}^{j}: e_1^i \to e_2^j,
\qquad i,j\in\{+,-\}.
\]
\end{enumerate}

\item With above notations, every arrow in $Q$ can be written in the form
\[
{}^i\alpha_h^j,
\]
where $h$ is a half-edge of $\Gamma$ and $i,j\in \{\varnothing,+,-\}$. Moreover, the arrows
\[
{}^i\alpha_h^j \quad \text{and} \quad {}^{i'}\alpha_{\rho(h)}^{j'}
\]
are composable in $Q$ whenever $j=i'$.

Thus, for each vertex $e\in Q_0$ corresponding to an $\iota$-orbit $[h]$ in $H(\Gamma)$, we define the following cycle in the quiver $Q$:
\[
C_{(e,h)}
=
{}^{i_1}\alpha_h^{i_2}\,
{}^{i_2}\alpha_{\rho(h)}^{i_3}
\cdots
{}^{i_{\mathrm{val}(s(h))}}
\alpha_{\rho^{\mathrm{val}(s(h))-1}(h)}^{i_1}.
\]
Note that the cycle $C_{(e,h)}$ defined in this way is not necessarily unique in $Q$. However, by the relations in the ideal $I$ defined below, the cycle associated with a given pair $(e,h)$ becomes unique in the quotient algebra $kQ/I$.

\item The ideal $I$ is generated by the following relations:
\begin{enumerate}[(i)]
\item
\[
{}^i\alpha_h^+\,{}^{+}\alpha_{\rho(h)}^{j}
=
{}^i\alpha_h^-\,{}^{-}\alpha_{\rho(h)}^{j},
\]
for each half-edge $h$ such that $\rho(h)\in H_\times$, and for all $i,j\in\{\varnothing,+,-\}$. These relations ensure that the cycle $C_{(e,h)}$ defined above is well-defined in $kQ/I$ for every pair $(e,h)$.

\item
\[
C_{(e,h_1)}^{m(s(h_1))}
=
C_{(e,h_2)}^{m(s(h_2))},
\]
where $\{h_1,h_2\}$ is the edge in $E_\circ$ corresponding to the vertex $e$ of $\Gamma$, and $m=\frac{d}{\mathrm{val}}$
is the associated integer-valued multiplicity function of $(\Gamma,d)$.
			
			\item $${}^i\alpha_{h}\alpha_{\iota\rho(h)}^j=0,$$
			for each half-edge $h$ such that $\rho(h)\in H_\circ$, and for all $i,j\in\{\varnothing,+,-\}$. 

	\item $${}^{i_1'}\alpha_h^{i_2}{}^{i_2}\alpha_{\rho(h)}^{i_3}
\cdots
{}^{i_{\mathrm{val}(s(h))}}
\alpha_{\rho^{\mathrm{val}(s(h))-1}(h)}^{i_1}\;C_{(e,h)}^{m(s(h))-1}=0$$
	for all $C_{(e,h)}
=
{}^{i_1}\alpha_h^{i_2}\,
{}^{i_2}\alpha_{\rho(h)}^{i_3}
\cdots
{}^{i_{\mathrm{val}(s(h))}}
\alpha_{\rho^{\mathrm{val}(s(h))-1}(h)}^{i_1}$ with $h\in H_\times$ and $i_1'\neq i_1$.

\item $$C_{(e,h)}^{m(s(h))}\;{}^{i_1}\alpha_h^{i_2}=0$$
	for all $C_{(e,h)}
=
{}^{i_1}\alpha_h^{i_2}\,
{}^{i_2}\alpha_{\rho(h)}^{i_3}
\cdots
{}^{i_{\mathrm{val}(s(h))}}
\alpha_{\rho^{\mathrm{val}(s(h))-1}(h)}^{i_1}$.
		\end{enumerate}
	\end{enumerate}
	
	\begin{definition}\textnormal{(cf. \cite[Definition 4.3]{EGV} and \cite[Definition 3.5]{So2})}\label{skew-BGA}
		A $k$-algebra $A$ is called a {\it skew-Brauer graph algebra} (abbr. skew-BGA) if there exists a skew-BG $(\Gamma,d)$ such that $A\cong kQ_\Gamma/I_{\Gamma,d}$ as $k$-algebras.
	\end{definition}

	\begin{remark}
		We note that our notation follows~\cite{EGV} and differs slightly from that of~\cite{So2}. In particular, the commutativity relations in~\cite{So2} (corresponding to relation~(ii) in our ideal \(I\)) involve coefficients given by certain powers of \(2\), but these coefficients can be eliminated via an algebra isomorphism under the $\operatorname{char} k \neq 2$ assumption.
	\end{remark}

	We illustrate the above construction with the following example.

	\begin{example} (Example \ref{exa:preproj-A3} revisited) Consider the skew-BG $(\Gamma/\langle \nu\rangle,\mathrm{val})$ induced from the biserial FBG $(\Gamma,d)$ in Example~1, where the orbifold ribbon graph $\Gamma/\langle \nu\rangle$ is given as follows.
	\begin{center}

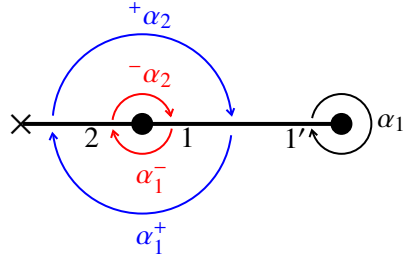
\tikzset{every picture/.style={line width=0.75pt}} %set default line width to 0.75pt        

\begin{tikzpicture}[x=0.75pt,y=0.75pt,yscale=-1,xscale=1]
%uncomment if require: \path (0,235); %set diagram left start at 0, and has height of 235

%Shape: Circle [id:dp8934929695145333] 
\draw  [fill={rgb, 255:red, 0; green, 0; blue, 0 }  ,fill opacity=1 ] (184.5,115.5) .. controls (184.5,112.74) and (186.74,110.5) .. (189.5,110.5) .. controls (192.26,110.5) and (194.5,112.74) .. (194.5,115.5) .. controls (194.5,118.26) and (192.26,120.5) .. (189.5,120.5) .. controls (186.74,120.5) and (184.5,118.26) .. (184.5,115.5) -- cycle ;
%Straight Lines [id:da09464826976213714] 
\draw [line width=1.5]    (189.5,115.5) -- (289.5,115.5) ;
%Shape: Circle [id:dp36278708481485955] 
\draw  [fill={rgb, 255:red, 0; green, 0; blue, 0 }  ,fill opacity=1 ] (284.5,115.5) .. controls (284.5,112.74) and (286.74,110.5) .. (289.5,110.5) .. controls (292.26,110.5) and (294.5,112.74) .. (294.5,115.5) .. controls (294.5,118.26) and (292.26,120.5) .. (289.5,120.5) .. controls (286.74,120.5) and (284.5,118.26) .. (284.5,115.5) -- cycle ;
%Straight Lines [id:da034514994500294494] 
\draw [line width=1.5]    (129.5,115.5) -- (189.5,115.5) ;
%Shape: Arc [id:dp7077787550393215] 
\draw  [draw opacity=0] (174.63,113.48) .. controls (175.62,106.15) and (181.9,100.5) .. (189.5,100.5) .. controls (196.46,100.5) and (202.31,105.24) .. (204,111.66) -- (189.5,115.5) -- cycle ; \draw  [color={rgb, 255:red, 255; green, 0; blue, 0 }  ,draw opacity=1 ] (174.63,113.48) .. controls (175.62,106.15) and (181.9,100.5) .. (189.5,100.5) .. controls (196.46,100.5) and (202.31,105.24) .. (204,111.66) ;  
%Shape: Arc [id:dp06291935051305986] 
\draw  [draw opacity=0] (144.69,111.35) .. controls (146.78,88.44) and (166.05,70.5) .. (189.5,70.5) .. controls (212.67,70.5) and (231.75,88.01) .. (234.23,110.51) -- (189.5,115.5) -- cycle ; \draw  [color={rgb, 255:red, 0; green, 0; blue, 255 }  ,draw opacity=1 ] (144.69,111.35) .. controls (146.78,88.44) and (166.05,70.5) .. (189.5,70.5) .. controls (212.67,70.5) and (231.75,88.01) .. (234.23,110.51) ;  
%Shape: Arc [id:dp07544771630883917] 
\draw  [draw opacity=0] (234.14,121.19) .. controls (231.35,143.36) and (212.43,160.5) .. (189.5,160.5) .. controls (166.42,160.5) and (147.39,143.12) .. (144.8,120.73) -- (189.5,115.5) -- cycle ; \draw  [color={rgb, 255:red, 0; green, 0; blue, 255 }  ,draw opacity=1 ] (234.14,121.19) .. controls (231.35,143.36) and (212.43,160.5) .. (189.5,160.5) .. controls (166.42,160.5) and (147.39,143.12) .. (144.8,120.73) ;  
%Shape: Arc [id:dp5727848764613193] 
\draw  [draw opacity=0] (204.31,117.86) .. controls (203.18,125.03) and (196.98,130.5) .. (189.5,130.5) .. controls (182.85,130.5) and (177.21,126.17) .. (175.24,120.18) -- (189.5,115.5) -- cycle ; \draw  [color={rgb, 255:red, 255; green, 0; blue, 0 }  ,draw opacity=1 ] (204.31,117.86) .. controls (203.18,125.03) and (196.98,130.5) .. (189.5,130.5) .. controls (182.85,130.5) and (177.21,126.17) .. (175.24,120.18) ;  
\draw  [color={rgb, 255:red, 0; green, 0; blue, 255 }  ,draw opacity=1 ] (235.68,107.57) -- (234.37,111.02) -- (231.77,108.4) ;
%Shape: Arc [id:dp47009079000127096] 
\draw  [draw opacity=0] (274.78,112.61) .. controls (276.13,105.71) and (282.2,100.5) .. (289.5,100.5) .. controls (297.78,100.5) and (304.5,107.22) .. (304.5,115.5) .. controls (304.5,123.78) and (297.78,130.5) .. (289.5,130.5) .. controls (282.5,130.5) and (276.62,125.71) .. (274.97,119.22) -- (289.5,115.5) -- cycle ; \draw   (274.78,112.61) .. controls (276.13,105.71) and (282.2,100.5) .. (289.5,100.5) .. controls (297.78,100.5) and (304.5,107.22) .. (304.5,115.5) .. controls (304.5,123.78) and (297.78,130.5) .. (289.5,130.5) .. controls (282.5,130.5) and (276.62,125.71) .. (274.97,119.22) ;  
\draw  [color={rgb, 255:red, 255; green, 0; blue, 0 }  ,draw opacity=1 ] (205.02,107.9) -- (203.71,111.35) -- (201.1,108.73) ;
\draw   (273.48,122.15) -- (274.68,118.66) -- (277.36,121.2) ;
\draw  [color={rgb, 255:red, 255; green, 0; blue, 0 }  ,draw opacity=1 ] (173.89,122.54) -- (174.95,119.01) -- (177.73,121.43) ;
\draw  [color={rgb, 255:red, 0; green, 0; blue, 255 }  ,draw opacity=1 ] (142.99,123.68) -- (144.47,120.3) -- (146.94,123.05) ;

% Text Node
\draw (121,109) node [anchor=north west][inner sep=0.75pt]  [font=\Large] [align=left] {$\times$};
% Text Node
\draw (159,116) node [anchor=north west][inner sep=0.75pt]   [align=left] {$2$};
% Text Node
\draw (207,116) node [anchor=north west][inner sep=0.75pt]   [align=left] {$1$};
% Text Node
\draw (259,116) node [anchor=north west][inner sep=0.75pt]   [align=left] {$1'$};
% Text Node
\draw (181,82.33) node [anchor=north west][inner sep=0.75pt]   [align=left] {$\color{red}{}^-\alpha_2$};
% Text Node
\draw (185,133) node [anchor=north west][inner sep=0.75pt]   [align=left] {$\color{red}\alpha_1^-$};
% Text Node
\draw (181,53) node [anchor=north west][inner sep=0.75pt]   [align=left] {$\color{blue}{}^+\alpha_2$};
% Text Node
\draw (185,163.33) node [anchor=north west][inner sep=0.75pt]   [align=left] {$\color{blue}\alpha_1^+$};
% Text Node
\draw (306,109) node [anchor=north west][inner sep=0.75pt]   [align=left] {$\alpha_{1'}$};

\end{tikzpicture}

	\end{center}
	Thus the skew-BGA $A=kQ/I$ is given by the quiver $Q=Q_{\Gamma/\langle\nu\rangle}$,
$$
\begin{tikzcd}
2^+ \arrow[r, "{}^+\alpha_2", shift left,blue] 
& {\{1,1'\}} \arrow[l, "\alpha_1^+", shift left,blue] 
\arrow["\alpha_{1'}"', loop, distance=2em, in=125, out=55] 
\arrow[r, "\alpha_1^-"', shift right,red] 
& 2^- \arrow[l, "{}^-\alpha_2"', shift right,red]
\end{tikzcd}
$$
with relations in $I$ given by
\begin{enumerate}[(i)]
\item 
${\color{blue}\alpha_1^+}{\color{blue}{}^+\alpha_2}
=
{\color{red}\alpha_1^-}{\color{red}{}^-\alpha_2}$;

\item 
$\alpha_{1'}
=
{\color{blue}\alpha_1^+}{\color{blue}{}^+\alpha_2}$, $\alpha_{1'}
=
{\color{red}\alpha_1^-}{\color{red}{}^-\alpha_2}$;

\item 
$\alpha_{1'}{\color{blue}\alpha_1^+}
=
\alpha_{1'}{\color{red}\alpha_1^-}
=
{\color{blue}{}^+\alpha_2}\alpha_{1'}
=
{\color{red}{}^-\alpha_2}\alpha_{1'}
=0$;

\item 
${\color{red}{}^-\alpha_2}{\color{blue}\alpha_1^+}
=
{\color{blue}{}^+\alpha_2}{\color{red}\alpha_1^-}
=0$;

\item 
${\color{blue}\alpha_1^+}{\color{blue}{}^+\alpha_2}{\color{blue}\alpha_1^+}
=
{\color{blue}{}^+\alpha_2}{\color{blue}\alpha_1^+}{\color{blue}{}^+\alpha_2}
=
{\color{red}\alpha_1^-}{\color{red}{}^-\alpha_2}{\color{red}\alpha_1^-}
=
{\color{red}{}^-\alpha_2}{\color{red}\alpha_1^-}{\color{red}{}^-\alpha_2}
=
\alpha_{1'}^2
=0$.
\end{enumerate}
Consider the BGA $B$ associated with $(\widehat{\Gamma/\langle \nu \rangle},\mathrm{val})$ as shown in Figure~\ref{fig:red-form-preproj-A3}. There is a natural $\mathbb Z/2\mathbb Z$-action on it induced by the involution $\phi$ defined in Subsection~\ref{subsec:reduced-form}. In fact, we have the following Morita equivalence
\[
B\langle\phi\rangle \stackrel{\mathrm{Morita}}{\sim} B \cong A,
\]
following~\cite[Section~2.3]{RR} and also~\cite[Subsection~1.2]{AP}.
\end{example}
	
In fact, the above Morita equivalence holds in general.

\begin{proposition}\textnormal{(\cite[Proposition 3.9]{So2})}\label{prop:red-skew-BGA}
	Let $A$ be the skew-BGA associated with a skew-BG $(\Gamma,d)$, where $\Gamma$ has at least one orbifold edge. Let $B$ be the BGA associated with the BG $(\widehat{\Gamma},d)$ together with the natural involution $\phi$ on $B$. Assume that $\operatorname{char} k \neq 2$. Then $A$ is Morita equivalent to the skew group algebra $B\langle\phi\rangle$.
\end{proposition}

\subsection{Skew group algebras of biserial FBGAs}
\

Let $A$ be the biserial FBGA associated with the biserial FBG $(\Gamma,d)$. The Nakayama automorphism $\nu_A$ of $A$ is naturally induced by the Nakayama automorphism $\nu$ on $(\Gamma,d)$. We now consider the skew group algebra $AG$, where $G=\langle \nu_A\rangle$ is the cyclic group generated by $\nu_A$.

\subsubsection{Admissible cases}
\

We note that when the action of $G$ is admissible (that is, for any half-edge \(h\), we have \(\nu^{n}(h)\neq \iota(h)\) for all \(n\in\mathbb{Z}\)), the orbit graph $\Gamma/\langle \nu \rangle$ is in fact a ribbon graph, and hence $(\Gamma/\langle \nu \rangle,d)$ is a BG. Therefore, the following proposition holds.

\begin{proposition}\label{prop:adm-skew=BGA}
	Let $A$ be the biserial FBGA associated with a biserial FBG $(\Gamma,d)$, and suppose the Nakayama automorphism $\nu$ of $(\Gamma,d)$ acts admissibly on $\Gamma$. Let $A_{\mathrm{red}}$ be the BGA associated with the BG $(\Gamma/\langle \nu \rangle,d)$. Then $A_{\mathrm{red}}$ is Morita equivalent to the skew group algebra $A\langle \nu_A\rangle$.
\end{proposition}

\begin{proof}
	This equivalence can also be verified via the description of the quiver of the skew group algebras for cyclic groups by Reiten and Riedtmann \cite[Section~2.3]{RR}. This description requires that the orders of the stabilizers under the action of the group \(G\) are invertible in the field \(k\). In our setting, the admissibility of $G$-asction implies that the required invertibility condition is automatically satisfied, and the description applies directly in our situation.
\end{proof}

\subsubsection{Non-admissible cases}
\

We note that when the action of $G$ is non-admissible, the graph of orbits $\Gamma/\langle \nu \rangle$ is an orbifold ribbon graph, and hence $(\Gamma/\langle \nu \rangle,d)$ is a skew-BG. 

We also note that each orbifold edge in $\Gamma/\langle \nu \rangle$ corresponds to a $\nu$-orbit $\{e_1,\dots,e_m\}$ in $\Gamma$, which consists of $m$ distinct loops in $E(\Gamma)$. We note that $m$ is constant since, by Lemma \ref{lem:n(v)-property}, every $\nu$-orbit in a connected biserial FBG $(\Gamma,d)$ contains the same number of half-edges. Therefore, for each vertex $\vv$ of $\Gamma$, it follows from Lemma \ref{lem:n(v)-property} that $n(\vv)=2m$, that is, the order of the group $G$ is $2m$. We now proceed to prove the main result of this section.

\begin{proposition}\label{prop:nonadm-skew=skew-BGA}
Let $A$ be the biserial FBGA associated with a biserial FBG $(\Gamma,d)$, and suppose that the Nakayama automorphism $\nu$ of $(\Gamma,d)$ acts non-admissibly on $\Gamma$. Let $B$ be the skew-BGA corresponding to the skew-BG $(\Gamma/\langle \nu \rangle,d)$. Assume that $\operatorname{char} k \neq 2$. Then $B$ is Morita equivalent to the skew group algebra $A\langle \nu_A\rangle$.
\end{proposition}

\begin{proof}
	By the preceding discussion and Proposition \ref{prop:bfBGA-basic}, we may assume that $G=\langle \nu_A\rangle$ is a cyclic group of order $2m$. Following \cite{RR}, we first construct the basic algebra of $AG$.

	By the definition of a biserial FBG, there are two types of $\nu$-orbits in $E(\Gamma)$: those consisting of $m$ loops, and those consisting of $2m$ edges joining the same pair of vertices in $\Gamma$. In the corresponding cases, the stabilizers under the action of the group G have orders $2$ and $1$, respectively.
	
	Denote by $(E(\Gamma)/\langle \nu \rangle)_{\times}$ the set of orbits of the first type, and by $(E(\Gamma)/\langle \nu \rangle)_{\circ}$ the set of orbits of the second type. Note that
\[
E(\Gamma)/\langle \nu \rangle \;=\; (E(\Gamma)/\langle \nu \rangle)_{\times} \,\sqcup\, (E(\Gamma)/\langle \nu \rangle)_{\circ}.
\]
 We identify the edges of $\Gamma$ with the primitive idempotents of $A$. Let $g=\nu^m$. Then for any $[e]\in E(\Gamma)/\langle \nu \rangle$, we have the following idempotents in $AG$. Note that the assumption \(\operatorname{char} k \neq 2\) is necessary here. 

\begin{itemize}
\item If $[e]\in (E(\Gamma)/\langle \nu \rangle)_{\times}$, then there are two corresponding idempotents given by
\[
e_{[e]}^{+} := \left( \sum_{e'\in [e]} e' \right)\otimes \frac{1+g}{2}, 
\qquad
e_{[e]}^{-} := \left( \sum_{e'\in [e]} e' \right)\otimes \frac{1-g}{2}.
\]

\item If $[e]\in (E(\Gamma)/\langle \nu \rangle)_{\circ}$, then there is a corresponding idempotent given by
\[
e_{[e]} := \left( \sum_{e'\in [e]} e' \right)\otimes \frac{1+g}{2}.
\]
\end{itemize}

Next, consider an arrow $\alpha$ from $e$ to $e'$ in $Q_\Gamma$. We consider arrows up to the action of $\nu$, that is, we fix a representative for each $\nu$-orbit of arrows in $Q_\Gamma$. For any $[e],[e']\in E(\Gamma)/\langle \nu \rangle$, we define arrows as follows.

\begin{itemize}
\item[(1)] If $[e],[e']\in (E(\Gamma)/\langle \nu \rangle)_{\times}$, then there are four arrows:
\begin{itemize}
\item[(i)] from $e_{[e]}^{+}$ to $e_{[e']}^{+}$,
\[
{}^{+}\alpha_{[e]}^{+} := \left( \sum_{i=0}^{2m-1} \nu^{i}(\alpha) \right)\otimes \frac{1+g}{2};
\]
\item[(ii)] from $e_{[e]}^{+}$ to $e_{[e']}^{-}$,
\[
{}^{+}\alpha_{[e]}^{-} := \left( \sum_{i=0}^{m-1} \bigl(\nu^{i}(\alpha)-\nu^{m+i}(\alpha)\bigr) \right)\otimes \frac{1-g}{2};
\]
\item[(iii)] from $e_{[e]}^{-}$ to $e_{[e']}^{+}$,
\[
{}^{-}\alpha_{[e]}^{+} := \left( \sum_{i=0}^{m-1} \bigl(\nu^{i}(\alpha)-\nu^{m+i}(\alpha)\bigr) \right)\otimes \frac{1+g}{2};
\]
\item[(iv)] from $e_{[e]}^{-}$ to $e_{[e']}^{-}$,
\[
{}^{-}\alpha_{[e]}^{-} := \left( \sum_{i=0}^{2m-1} \nu^{i}(\alpha) \right)\otimes \frac{1-g}{2}.
\]
\end{itemize}

\item[(2)] If $[e],[e']\in (E(\Gamma)/\langle \nu \rangle)_{\circ}$, then there is an arrow from $e_{[e]}$ to $e_{[e']}$,
\[
\alpha_{[e]} := \left( \sum_{i=0}^{2m-1} \nu^{i}(\alpha) \right)\otimes \frac{1+g}{2}.
\]

\item[(3)] If $[e]\in (E(\Gamma)/\langle \nu \rangle)_{\times}$ and $[e']\in (E(\Gamma)/\langle \nu \rangle)_{\circ}$, then there are two arrows:
\begin{itemize}
\item[(i)] from $e_{[e]}^{+}$ to $e_{[e']}$,
\[
{}^{+}\alpha_{[e]} := \left( \sum_{i=0}^{2m-1} \nu^{i}(\alpha) \right)\otimes \frac{1+g}{2};
\]
\item[(ii)] from $e_{[e]}^{-}$ to $e_{[e']}$,
\[
{}^{-}\alpha_{[e]} := \left( \sum_{i=0}^{m-1} \bigl(\nu^{i}(\alpha)-\nu^{m+i}(\alpha)\bigr) \right)\otimes \frac{1+g}{2}.
\]
\end{itemize}

\item[(4)] If $[e]\in (E(\Gamma)/\langle \nu \rangle)_{\circ}$ and $[e']\in (E(\Gamma)/\langle \nu \rangle)_{\times}$, then there are two arrows:
\begin{itemize}
\item[(i)] from $e_{[e]}$ to $e_{[e']}^{+}$,
\[
\alpha_{[e]}^{+} := \left( \sum_{i=0}^{2m-1} \nu^{i}(\alpha) \right)\otimes \frac{1+g}{2};
\]
\item[(ii)] from $e_{[e]}$ to $e_{[e']}^{-}$,
\[
\alpha_{[e]}^{-} := \left( \sum_{i=0}^{m-1} \bigl(\nu^{i}(\alpha)-\nu^{m+i}(\alpha)\bigr) \right)\otimes \frac{1-g}{2}.
\]
\end{itemize}
\end{itemize}
Let
\[
f := \sum_{[e]\in (E(\Gamma)/\langle \nu \rangle)_{\times}} \bigl(e_{[e]}^{+}+e_{[e]}^{-}\bigr)
\;+\;
\sum_{[e]\in (E(\Gamma)/\langle \nu \rangle)_{\circ}} e_{[e]}.
\]
Then the above vertices and arrows, together with the multiplication inherited from $A$, define a basic algebra $fAGf$, which is Morita equivalent to $AG$, following the Reiten and Riedtmann's description in \cite[Section~2.3]{RR}.

Now consider $B$, the skew-BGA corresponding to the skew-BG $(\Gamma/\langle \nu \rangle,d)$. 
Since $E(\Gamma/\langle \nu \rangle)$ is in bijection with $E(\Gamma)/\langle \nu \rangle$, 
we fix such a bijection and denote it by $\theta$. 
Then there is a natural algebra homomorphism from $B$ to $fAGf$, induced by a quiver morphism, given as follows:
\[
B \longrightarrow fAGf\colon\qquad 
e_{e}^{i} \longmapsto e_{\theta(e)}^{i}, \quad 
{}^{j}\alpha_{e}^{k} \longmapsto {}^{j}\alpha_{\theta(e)}^{k},
\]
where $e\in E(\Gamma/\langle \nu \rangle)$ and $i,j,k\in\{+,-,\varnothing\}$. This homomorphism is in fact an isomorphism of algebras, and the verification is the same as in \cite[Proposition~3.9]{So2}. Therefore, in this case, $B$ is Morita equivalent to the skew group algebra $A\langle \nu_A\rangle$.
\end{proof}

	\section{Kauer moves for biserial FBGAs}\label{sec:kauer-move}

	\subsection{Silting theory}\label{subsec:silt-theory}
	\

	We first recall some basic notions from silting theory of finite-dimensional algebras. Let $A$ be a finite-dimensional algebra over a field $k$. We write $\mathcal{K}^b(\mathrm{proj}\, A)$ for the bounded homotopy category of finitely generated projective $A$-modules.

\begin{definition}\textnormal{(cf. \cite{KV,W})}
A complex $P \in \mathcal{K}^b(\mathrm{proj}\, A)$ is called {\it presilting} (resp. {\it pretilting}) if
$$
\mathrm{Hom}_{\mathcal{K}^b(\mathrm{proj}\, A)}(P, P[i]) = 0, 
\qquad \text{for all } i>0\;\;\text{(resp. $i\neq 0$)}.
$$
Moreover, a presilting (resp. pretilting) complex $P$ is called {\it silting} (resp. {\it tilting}) if
$$
\mathrm{thick}(P) = \mathcal{K}^b(\mathrm{proj}\, A),
$$
that is, $P$ generates $\mathcal{K}^b(\mathrm{proj}\, A)$ as a triangulated category.
\end{definition}

There is a natural partial order on the set of basic silting complexes, defined as follows.

\begin{definition} 
For basic silting complexes $T$ and $T'$, we write $T \ge T'$ if
$$
\mathrm{Hom}_{\mathcal{K}^b(\mathrm{proj}\,A)}(T, T'[i]) = 0
\quad \text{for all } i>0.
$$
\end{definition}

By~\cite{AI}, this relation defines a partial order on the set of basic silting complexes. Let $\mathrm{Hasse}(\mathrm{silt}\,A)$ denote the Hasse quiver of this poset. We then have the following definitions (under the equivalent definition given by \cite[Theorem 1.2]{AM}).

\begin{definition}
Let $A$ be a finite-dimensional self-injective algebra.
\begin{enumerate}[(1)]
    \item We say that $A$ is {\it silting-discrete} if, for any basic silting complex
    $T \in \mathcal{K}^b(\mathrm{proj}\,A)$, there are only finitely many basic silting
    complexes $U$ such that $T \ge U \ge T[1].$
   
    \item We say that $A$ is {\it tilting-discrete} if, for any basic tilting complex
    $T \in \mathcal{K}^b(\mathrm{proj}\,A)$, there are only finitely many basic tilting
    complexes $U$ such that $ T \ge U \ge T[1].$
\end{enumerate}
\end{definition}

Recall that a complex $P$ in $\mathcal{K}^b(\mathrm{proj}\, A)$ is called {\it $n$-term} if it is isomorphic to a complex of the form
$$
\cdots\longrightarrow P^{-1} \longrightarrow P^{0}\longrightarrow\cdots,
$$
with $P^{i}=0$ for all $i\neq 0,-1,\cdots,-(n-1)$, that is, $A\geq P\geq A[n-1]$

Now let $\nu$ be a triangle auto-equivalence of $\mathcal{K}^b(\mathrm{proj}\, A)$, then a complex $P\in \mathcal{K}^b(\mathrm{proj}\, A)$ is called {\it $\nu$-stable} if $\nu P\cong P$. In fact, we have the following theorem.

\begin{theorem}\textnormal{(cf. \cite[Theorem 2.1]{AR} and \cite[Theorem A.4]{Ai})}\label{prop:sel-inj-silting}
Let $A$ be a self-injective algebra with Nakayama automorphism $\nu_A$, and let 
$T \in \mathcal{K}^b(\mathrm{proj}\, A)$ be a silting complex. Then $T$ is a tilting complex if and only if $T$ is $\nu_A$-stable.
\end{theorem}

We denote by $2\text{-}\mathrm{presilt}\,A$ (resp.\ $2\text{-}\mathrm{silt}\,A$, $2\text{-}\mathrm{tilt}\,A$, $2\text{-}\mathrm{presilt}^{\nu}\,A$, $2\text{-}\mathrm{tilt}^{\nu}\,A$) the set of basic $2$-term presilting (resp.\ silting, tilting, $\nu$-stable presilting, $\nu$-stable tilting) complexes of $A$.

\begin{definition}\label{def:mutation}
Let $P = X \oplus P'$ be a basic two-term silting complex. The {\it left mutation} of $P$ at $X$ is defined as follows.
Take a minimal left $\mathrm{add}(P')$-approximation $f\colon X\to E$ and complete it to a triangle
$$
X \xrightarrow{f} E \rightarrow X' \rightarrow X[1].
$$
Then
$$
\mu_X^{-}(P) := X' \oplus P'
$$
is called the left mutation of $P$ at $X$. The right mutation $\mu_X^{+}(P)$ is defined dually. We say that such a mutation is {\it irreducible} if $X$ is indecomposable, and {\it $\nu$-irreducible} if $X$ is indecomposable as a $\nu$-stable object; that is, $X$ is $\nu$-stable and none of its proper direct summands is $\nu$-stable.
\end{definition}

For a given finite-dimensional algebra $A$, mutations starting from the tilting complex $A$ were originally introduced by Rickard \cite{Ric2} and later generalized by Okuyama \cite{Oku}; the resulting complexes are known as {\it Okuyama--Rickard complexes}. This result was generalized to the silting setting in~\cite{AI}.

Consider the (left) mutation quiver of $A$, whose vertices correspond to basic silting complexes. 
There is an arrow $T \to T'$ if $T'=\mu_X^{-}(T)$ for some indecomposable direct summand $X$ of $T$. 
By \cite[Theorem 2.35]{AI}, this quiver coincides with the Hasse quiver $\mathrm{Hasse}(\mathrm{silt}\,A)$ of $\mathrm{silt}\,A$.

Note that $\mathrm{Hasse}(2\text{-}\mathrm{silt}\,A)$ is the full subquiver of $\mathrm{Hasse}(\mathrm{silt}\,A)$ whose vertices correspond to basic $2$-term silting complexes. $\mathrm{Hasse}(2\text{-}\mathrm{silt}\,A)$ is an $n$-regular graph, where $n$ is the number of indecomposable direct summands of $A$, see~\cite[Corollary~3.8]{AIR}. Then we define the connectivity properties as follows (for the $\nu$-stable case, see, for example in \cite{CKL,AK}).

\begin{definition}
Let $A$ be a self-injective algebra. We say that $A$ is {\it $\nu$-stable-silting-connected} if for any two basic $\nu$-stable-silting complexes $T,T' \in \mathcal{K}^b(\mathrm{proj}\,A)$, there exists a finite sequence of  $\nu$-irreducible silting mutations transforming $T$ into $T'$.
\end{definition}

Now, based on existing results, we can give a preliminary description of $2$-term silting complexes for biserial FBGAs in some special cases.

\begin{proposition}\label{prop:2-silt-fms}
	Let $A$ be a biserial FBGA with associated biserial FBG $(\Gamma,d)$ such that $m(\vv)\geq 1$ for all vertices $\vv$ of $\Gamma$. Let $A'$ be the BGA associated with the same ribbon graph $\Gamma$. Then there is a bijection between
\begin{enumerate}
  		\item the set $2\text{-}\mathrm{silt}\,A$ of isomorphism classes of basic two-term silting complexes of $A$, and
  		\item the set $2\text{-}\mathrm{silt}\,A'=2\text{-}\mathrm{tilt}\,A'$ of isomorphism classes of basic two-term silting, equivalently, tilting complexes of $A'$.
	\end{enumerate}
	Moreover, $A$ is silting-discrete if and only if $\Gamma$ contains at most one odd cycle and no even cycles. In this case, $A$ is either a BGA or a local algebra.
\end{proposition}

\begin{proof}
	By definition, every arrow $\alpha$ defines a cycle
		$$C_\alpha=\alpha\rho(\alpha)\cdots\rho^l(\alpha)$$
		where $l+1$ denotes the cardinality of the $\rho$-orbit of $\alpha$. Since for each vertex $\vv\in V(\Gamma)$ the element $$\sum_{h\in H \text{ with } s(h)=\vv} C_{\alpha_h}$$ is in both the center of $A$ and the center of $A'$, it follows that there exists a $k$-algebra isomorphism
$$
B := A \Big/ \Big\langle \sum_{h\in H \text{ with } s(h)=\vv} C_{\alpha_h} \;\Big|\; \vv\in V \Big\rangle
\cong
A' \Big/ \Big\langle \sum_{h\in H \text{ with } s(h)=\vv} C_{\alpha_h} \;\Big|\; \vv\in V \Big\rangle .
$$
As shown in \cite[Theorem~11]{EJR} (see also \cite[Corollary~7.31]{AY}),
the sets $2\text{-}\mathrm{silt}\,A$ and $2\text{-}\mathrm{silt}\,A'$ both coincide with the set $2\text{-}\mathrm{silt}\,B$ of isomorphism classes of basic two-term silting complexes over $B$. 

Suppose that $A$ is silting-discrete. Then $2\text{-}\mathrm{silt}\,A$ is finite, and hence so is $2\text{-}\mathrm{silt}\,A'$. By \cite[Theorem~6.7]{AAC}, this is equivalent to the graph $\Gamma$ containing at most one odd cycle and no even cycles.
Conversely, assume that $\Gamma$ contains at most one odd cycle and no even cycles. If $m(\vv)$ is not an integer for some vertex $\vv$, then by the (SI) condition in the definition of biserial FBG and Lemma \ref{lem:n(v)-property}, $\Gamma$ is either a single loop or contains at least one $2$-cycle. Since $\Gamma$ contains no even cycles, in this case, $A$ is a local algebra, and hence silting-discrete by \cite[Corollary~3.12]{AI}. If $m(\vv)$ is an integer for every vertex $\vv$, then $A$ is a BGA. Therefore, by \cite[Theorem~6.7]{AAC}, $A$ is silting-discrete.
\end{proof}

The above condition on the multiplicity is necessary, as we can see from the following example.

\begin{example}\label{exa:silt-dis-but-have-2-cycles}
	Consider the biserial FBG $(\Gamma,d)$ where $\Gamma$ is the ribbon graph
	\begin{center}

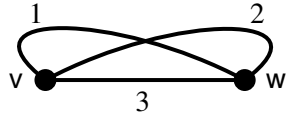
\tikzset{every picture/.style={line width=0.75pt}} %set default line width to 0.75pt        

\begin{tikzpicture}[x=0.75pt,y=0.75pt,yscale=-1,xscale=1]
%uncomment if require: \path (0,235); %set diagram left start at 0, and has height of 235

%Shape: Circle [id:dp8262589623190342] 
\draw  [fill={rgb, 255:red, 0; green, 0; blue, 0 }  ,fill opacity=1 ] (80,95) .. controls (80,92.24) and (82.24,90) .. (85,90) .. controls (87.76,90) and (90,92.24) .. (90,95) .. controls (90,97.76) and (87.76,100) .. (85,100) .. controls (82.24,100) and (80,97.76) .. (80,95) -- cycle ;
%Shape: Circle [id:dp9090944636447427] 
\draw  [fill={rgb, 255:red, 0; green, 0; blue, 0 }  ,fill opacity=1 ] (180,95) .. controls (180,92.24) and (182.24,90) .. (185,90) .. controls (187.76,90) and (190,92.24) .. (190,95) .. controls (190,97.76) and (187.76,100) .. (185,100) .. controls (182.24,100) and (180,97.76) .. (180,95) -- cycle ;
%Straight Lines [id:da01645424468666179] 
\draw [line width=1.5]    (85,95) -- (185,95) ;
%Curve Lines [id:da7083498694035377] 
\draw [line width=1.5]    (85,95) .. controls (150.25,60.25) and (230.25,61.25) .. (185,95) ;
%Curve Lines [id:da5110135380410727] 
\draw [line width=1.5]    (85,95) .. controls (39.75,60.25) and (119.75,60.25) .. (185,95) ;

% Text Node
\draw (65.25,90) node [anchor=north west][inner sep=0.75pt]   [align=left] {$\vv$};
% Text Node
\draw (194.25,90) node [anchor=north west][inner sep=0.75pt]   [align=left] {$\ww$};
% Text Node
\draw (75.25,55) node [anchor=north west][inner sep=0.75pt]   [align=left] {$1$};
% Text Node
\draw (186.25,55) node [anchor=north west][inner sep=0.75pt]   [align=left] {$2$};
% Text Node
\draw (128.75,99.5) node [anchor=north west][inner sep=0.75pt]   [align=left] {$3$};
\end{tikzpicture}

	\end{center}
	with clockwise cyclic order around each vertex, and the degree function $d$ is defined by $d(\vv)=d(\ww)=1$. Then $m(\vv)=m(\ww)=\frac{1}{3}$. The associated biserial FBGA is isomorphic to the Nakayama algebra (also a gentle algebra) $A=kQ/I$ where $Q$ is the quiver 
	$$
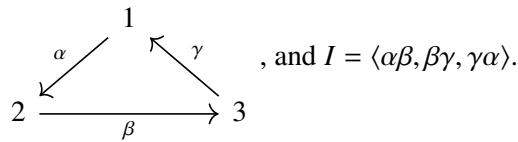
\begin{tikzcd}
                       & 1 \arrow[ld, "\alpha"'] &                         \\
2 \arrow[rr, "\beta"'] &                         & 3 \arrow[lu, "\gamma"']
\end{tikzcd}\text{, and }I=\langle \alpha\beta,\beta\gamma,\gamma\alpha\rangle.$$
By \cite{Ada}, $A$ is silting-discrete. However, $\Gamma$ contains even cycles. We remark that, for a gentle algebra $A$ and a silting complex $T$, the dg endomorphism algebra of $T$ is quasi-isomorphic to another graded gentle algebra \cite{APS,CS,JSW}.

Moreover, the above algebra is derived equivalent to the graded algebra $B=kQ'/I'$, where the quiver $Q'$ is
\[\begin{tikzcd}
                                   & 1 \arrow[ld, "\beta", shift left] \arrow[rd, "\gamma"] &   \\
2 \arrow[ru, "\alpha", shift left] &                                                        & 3
\end{tikzcd}
\text{ with $I'=\langle \alpha\beta,\beta\alpha\rangle$, $|\alpha|=|\gamma|=0$, $|\beta|=-1$.}\]
 This dg algebra is the dg endomorphism algebra of the silting mutation $\mu^-(A)$ of $A$ at vertex~$1$. Therefore, unlike in the gentle algebra case, the differential graded algebras arising from silting complexes of biserial FBGAs are not, in general, obtained simply by adding a grading to another biserial FBGA.
\end{example}

		\subsection{Kauer moves for BGAs}
	\

Let $I$ and $I'$ be disjoint sets of indecomposable projective $A$-modules such that 
$I \cup I'$ is a complete set of representatives of the isomorphism classes of indecomposable projective $A$-modules. For a set $\mathcal S$ of indecomposable projective modules, define
\[
P(\mathcal S):=\bigoplus_{P\in \mathcal S} P.
\]
Then
\[
A=P(I)\oplus P(I').
\]
Hence the left mutation $\mu^-_{P(I)}(A)$ is a silting complex by \cite{AI}. By Theorem \ref{prop:sel-inj-silting}, the complex $\mu^-_{P(I)}(A)$ is tilting if and only if $I$ is $\nu_A$-stable.

Now let $A$ be the BGA associated with a ribbon graph $\Gamma$. Since $A$ is symmetric, its Nakayama automorphism is trivial, and hence each indecomposable projective $A$-module forms a $\nu_A$-stable orbit by itself. In particular, for each indecomposable projective module $P$, the endomorphism algebra
\[
B:=\mathrm{End}_{\mathcal{K}^b(\mathrm{proj}\,A)}(\mu^-_{P}(A))
\]
was described by Kauer \cite{K}. He proved that $B$ is again a BGA associated with the ribbon graph
\[
\Gamma'=(\Gamma\setminus s)\cup s',
\]
where $s$ is the edge of $\Gamma$ corresponding to $P$, and $s'$ is obtained from $\Gamma$ by performing one of the local moves illustrated in Figure~\ref{fig:Kauer-moves}. We note that the multiplicity of each vertex is preserved.

	\begin{figure}[ht]
		\centering
		\begin{center}       
			\begin{tikzpicture}[x=0.75pt,y=0.75pt,yscale=-1,xscale=1]
				%uncomment if require: \path (0,479); %set diagram left start at 0, and has height of 479
				
				%Shape: Circle [id:dp22226166249641044] 
				\draw  [fill={rgb, 255:red, 0; green, 0; blue, 0 }  ,fill opacity=1 ] (50,55) .. controls (50,52.24) and (52.24,50) .. (55,50) .. controls (57.76,50) and (60,52.24) .. (60,55) .. controls (60,57.76) and (57.76,60) .. (55,60) .. controls (52.24,60) and (50,57.76) .. (50,55) -- cycle ;
				%Shape: Circle [id:dp6485957611984443] 
				\draw  [fill={rgb, 255:red, 0; green, 0; blue, 0 }  ,fill opacity=1 ] (100,55) .. controls (100,52.24) and (102.24,50) .. (105,50) .. controls (107.76,50) and (110,52.24) .. (110,55) .. controls (110,57.76) and (107.76,60) .. (105,60) .. controls (102.24,60) and (100,57.76) .. (100,55) -- cycle ;
				%Straight Lines [id:da0673474500039486] 
				\draw [line width=1.5]    (55,55) -- (75,90) -- (105,55) ;
				%Shape: Circle [id:dp10432059954586159] 
				\draw  [fill={rgb, 255:red, 0; green, 0; blue, 0 }  ,fill opacity=1 ] (70,90) .. controls (70,87.24) and (72.24,85) .. (75,85) .. controls (77.76,85) and (80,87.24) .. (80,90) .. controls (80,92.76) and (77.76,95) .. (75,95) .. controls (72.24,95) and (70,92.76) .. (70,90) -- cycle ;
				%Shape: Circle [id:dp886807605750549] 
				\draw  [fill={rgb, 255:red, 0; green, 0; blue, 0 }  ,fill opacity=1 ] (150,55) .. controls (150,52.24) and (152.24,50) .. (155,50) .. controls (157.76,50) and (160,52.24) .. (160,55) .. controls (160,57.76) and (157.76,60) .. (155,60) .. controls (152.24,60) and (150,57.76) .. (150,55) -- cycle ;
				%Shape: Circle [id:dp9929471825613261] 
				\draw  [fill={rgb, 255:red, 0; green, 0; blue, 0 }  ,fill opacity=1 ] (200,55) .. controls (200,52.24) and (202.24,50) .. (205,50) .. controls (207.76,50) and (210,52.24) .. (210,55) .. controls (210,57.76) and (207.76,60) .. (205,60) .. controls (202.24,60) and (200,57.76) .. (200,55) -- cycle ;
				%Straight Lines [id:da8944069118075275] 
				\draw [line width=1.5]    (155,55) -- (175,90) -- (205,55) ;
				%Shape: Circle [id:dp39909523582360173] 
				\draw  [fill={rgb, 255:red, 0; green, 0; blue, 0 }  ,fill opacity=1 ] (170,90) .. controls (170,87.24) and (172.24,85) .. (175,85) .. controls (177.76,85) and (180,87.24) .. (180,90) .. controls (180,92.76) and (177.76,95) .. (175,95) .. controls (172.24,95) and (170,92.76) .. (170,90) -- cycle ;
				%Straight Lines [id:da039273848985515114] 
				\draw [line width=1.5]    (75,90) -- (175,90) ;
				%Straight Lines [id:da2634289074516307] 
				\draw [color={rgb, 255:red, 0; green, 0; blue, 255 }  ,draw opacity=1 ][line width=1.5]    (175,90) -- (175,155.5) ;
				%Shape: Circle [id:dp5555539095631987] 
				\draw  [fill={rgb, 255:red, 0; green, 0; blue, 0 }  ,fill opacity=1 ] (300.25,190.08) .. controls (300.26,192.84) and (298.04,195.09) .. (295.28,195.11) .. controls (292.51,195.12) and (290.26,192.89) .. (290.25,190.13) .. controls (290.23,187.37) and (292.46,185.12) .. (295.22,185.11) .. controls (297.98,185.09) and (300.23,187.32) .. (300.25,190.08) -- cycle ;
				%Shape: Circle [id:dp17611778165090142] 
				\draw  [fill={rgb, 255:red, 0; green, 0; blue, 0 }  ,fill opacity=1 ] (250.25,190.34) .. controls (250.26,193.1) and (248.04,195.35) .. (245.28,195.37) .. controls (242.51,195.38) and (240.26,193.16) .. (240.25,190.39) .. controls (240.24,187.63) and (242.46,185.38) .. (245.22,185.37) .. controls (247.98,185.35) and (250.24,187.58) .. (250.25,190.34) -- cycle ;
				%Straight Lines [id:da7982591909607111] 
				\draw [line width=1.5]    (295.25,190.11) -- (275.07,155.21) -- (245.25,190.37) ;
				%Shape: Circle [id:dp45756942985381555] 
				\draw  [fill={rgb, 255:red, 0; green, 0; blue, 0 }  ,fill opacity=1 ] (280.07,155.19) .. controls (280.08,157.95) and (277.85,160.2) .. (275.09,160.21) .. controls (272.33,160.23) and (270.08,158) .. (270.07,155.24) .. controls (270.05,152.48) and (272.28,150.23) .. (275.04,150.21) .. controls (277.8,150.2) and (280.05,152.42) .. (280.07,155.19) -- cycle ;
				%Shape: Circle [id:dp36589728520424525] 
				\draw  [fill={rgb, 255:red, 0; green, 0; blue, 0 }  ,fill opacity=1 ] (200.25,190.6) .. controls (200.26,193.37) and (198.04,195.62) .. (195.28,195.63) .. controls (192.52,195.64) and (190.26,193.42) .. (190.25,190.66) .. controls (190.24,187.9) and (192.46,185.64) .. (195.22,185.63) .. controls (197.99,185.62) and (200.24,187.84) .. (200.25,190.6) -- cycle ;
				%Shape: Circle [id:dp8575087725431652] 
				\draw  [fill={rgb, 255:red, 0; green, 0; blue, 0 }  ,fill opacity=1 ] (150.25,190.87) .. controls (150.27,193.63) and (148.04,195.88) .. (145.28,195.89) .. controls (142.52,195.91) and (140.27,193.68) .. (140.25,190.92) .. controls (140.24,188.16) and (142.46,185.91) .. (145.22,185.89) .. controls (147.99,185.88) and (150.24,188.1) .. (150.25,190.87) -- cycle ;
				%Straight Lines [id:da22470054015043983] 
				\draw [line width=1.5]    (195.25,190.63) -- (175.07,155.74) -- (145.25,190.89) ;
				%Shape: Circle [id:dp7921261977334222] 
				\draw  [fill={rgb, 255:red, 0; green, 0; blue, 0 }  ,fill opacity=1 ] (180.07,155.71) .. controls (180.08,158.47) and (177.85,160.72) .. (175.09,160.74) .. controls (172.33,160.75) and (170.08,158.52) .. (170.07,155.76) .. controls (170.05,153) and (172.28,150.75) .. (175.04,150.74) .. controls (177.8,150.72) and (180.05,152.95) .. (180.07,155.71) -- cycle ;
				%Straight Lines [id:da6675942361657408] 
				\draw [line width=1.5]    (275.07,155.21) -- (175.07,155.74) ;
				%Shape: Circle [id:dp6284420282406071] 
				\draw  [fill={rgb, 255:red, 0; green, 0; blue, 0 }  ,fill opacity=1 ] (418,55) .. controls (418,52.24) and (420.24,50) .. (423,50) .. controls (425.76,50) and (428,52.24) .. (428,55) .. controls (428,57.76) and (425.76,60) .. (423,60) .. controls (420.24,60) and (418,57.76) .. (418,55) -- cycle ;
				%Shape: Circle [id:dp03130012886927225] 
				\draw  [fill={rgb, 255:red, 0; green, 0; blue, 0 }  ,fill opacity=1 ] (468,55) .. controls (468,52.24) and (470.24,50) .. (473,50) .. controls (475.76,50) and (478,52.24) .. (478,55) .. controls (478,57.76) and (475.76,60) .. (473,60) .. controls (470.24,60) and (468,57.76) .. (468,55) -- cycle ;
				%Straight Lines [id:da4962668232216847] 
				\draw [line width=1.5]    (423,55) -- (443,90) -- (473,55) ;
				%Shape: Circle [id:dp683582884243922] 
				\draw  [fill={rgb, 255:red, 0; green, 0; blue, 0 }  ,fill opacity=1 ] (438,90) .. controls (438,87.24) and (440.24,85) .. (443,85) .. controls (445.76,85) and (448,87.24) .. (448,90) .. controls (448,92.76) and (445.76,95) .. (443,95) .. controls (440.24,95) and (438,92.76) .. (438,90) -- cycle ;
				%Shape: Circle [id:dp397672260285046] 
				\draw  [fill={rgb, 255:red, 0; green, 0; blue, 0 }  ,fill opacity=1 ] (518,55) .. controls (518,52.24) and (520.24,50) .. (523,50) .. controls (525.76,50) and (528,52.24) .. (528,55) .. controls (528,57.76) and (525.76,60) .. (523,60) .. controls (520.24,60) and (518,57.76) .. (518,55) -- cycle ;
				%Shape: Circle [id:dp06657193331954336] 
				\draw  [fill={rgb, 255:red, 0; green, 0; blue, 0 }  ,fill opacity=1 ] (568,55) .. controls (568,52.24) and (570.24,50) .. (573,50) .. controls (575.76,50) and (578,52.24) .. (578,55) .. controls (578,57.76) and (575.76,60) .. (573,60) .. controls (570.24,60) and (568,57.76) .. (568,55) -- cycle ;
				%Straight Lines [id:da29987409339029214] 
				\draw [line width=1.5]    (523,55) -- (543,90) -- (573,55) ;
				%Shape: Circle [id:dp12399283165637831] 
				\draw  [fill={rgb, 255:red, 0; green, 0; blue, 0 }  ,fill opacity=1 ] (538,90) .. controls (538,87.24) and (540.24,85) .. (543,85) .. controls (545.76,85) and (548,87.24) .. (548,90) .. controls (548,92.76) and (545.76,95) .. (543,95) .. controls (540.24,95) and (538,92.76) .. (538,90) -- cycle ;
				%Straight Lines [id:da3401186486796579] 
				\draw [line width=1.5]    (443,90) -- (543,90) ;
				%Shape: Circle [id:dp4275305196375967] 
				\draw  [fill={rgb, 255:red, 0; green, 0; blue, 0 }  ,fill opacity=1 ] (668.25,190.08) .. controls (668.26,192.84) and (666.04,195.09) .. (663.28,195.11) .. controls (660.51,195.12) and (658.26,192.89) .. (658.25,190.13) .. controls (658.23,187.37) and (660.46,185.12) .. (663.22,185.11) .. controls (665.98,185.09) and (668.23,187.32) .. (668.25,190.08) -- cycle ;
				%Shape: Circle [id:dp09576828796536874] 
				\draw  [fill={rgb, 255:red, 0; green, 0; blue, 0 }  ,fill opacity=1 ] (618.25,190.34) .. controls (618.26,193.1) and (616.04,195.35) .. (613.28,195.37) .. controls (610.51,195.38) and (608.26,193.16) .. (608.25,190.39) .. controls (608.24,187.63) and (610.46,185.38) .. (613.22,185.37) .. controls (615.98,185.35) and (618.24,187.58) .. (618.25,190.34) -- cycle ;
				%Straight Lines [id:da46689319724200384] 
				\draw [line width=1.5]    (663.25,190.11) -- (643.07,155.21) -- (613.25,190.37) ;
				%Shape: Circle [id:dp7979982659667229] 
				\draw  [fill={rgb, 255:red, 0; green, 0; blue, 0 }  ,fill opacity=1 ] (648.07,155.19) .. controls (648.08,157.95) and (645.85,160.2) .. (643.09,160.21) .. controls (640.33,160.23) and (638.08,158) .. (638.07,155.24) .. controls (638.05,152.48) and (640.28,150.23) .. (643.04,150.21) .. controls (645.8,150.2) and (648.05,152.42) .. (648.07,155.19) -- cycle ;
				%Shape: Circle [id:dp05121931006424041] 
				\draw  [fill={rgb, 255:red, 0; green, 0; blue, 0 }  ,fill opacity=1 ] (568.25,190.6) .. controls (568.26,193.37) and (566.04,195.62) .. (563.28,195.63) .. controls (560.52,195.64) and (558.26,193.42) .. (558.25,190.66) .. controls (558.24,187.9) and (560.46,185.64) .. (563.22,185.63) .. controls (565.99,185.62) and (568.24,187.84) .. (568.25,190.6) -- cycle ;
				%Shape: Circle [id:dp742376625234417] 
				\draw  [fill={rgb, 255:red, 0; green, 0; blue, 0 }  ,fill opacity=1 ] (518.25,190.87) .. controls (518.27,193.63) and (516.04,195.88) .. (513.28,195.89) .. controls (510.52,195.91) and (508.27,193.68) .. (508.25,190.92) .. controls (508.24,188.16) and (510.46,185.91) .. (513.22,185.89) .. controls (515.99,185.88) and (518.24,188.1) .. (518.25,190.87) -- cycle ;
				%Straight Lines [id:da7471466121888295] 
				\draw [line width=1.5]    (563.25,190.63) -- (543.07,155.74) -- (513.25,190.89) ;
				%Shape: Circle [id:dp692218505072326] 
				\draw  [fill={rgb, 255:red, 0; green, 0; blue, 0 }  ,fill opacity=1 ] (548.07,155.71) .. controls (548.08,158.47) and (545.85,160.72) .. (543.09,160.74) .. controls (540.33,160.75) and (538.08,158.52) .. (538.07,155.76) .. controls (538.05,153) and (540.28,150.75) .. (543.04,150.74) .. controls (545.8,150.72) and (548.05,152.95) .. (548.07,155.71) -- cycle ;
				%Straight Lines [id:da6195120294149439] 
				\draw [line width=1.5]    (643.07,155.21) -- (543.07,155.74) ;
				%Curve Lines [id:da947257585904822] 
				\draw [color={rgb, 255:red, 0; green, 0; blue, 255 }  ,draw opacity=1 ][line width=1.5]    (643.04,150.21) .. controls (626,112.5) and (474,130.5) .. (443,90) ;
				
				% Text Node
				\draw (178,115) node [anchor=north west][inner sep=0.75pt]   [align=left] {$s'$};
				% Text Node
				\draw (335,110) node [anchor=north west][inner sep=0.75pt] [font=\LARGE]  [align=left] {$ \mathrel{\substack{
\overset{\mu_{s'}^+}{\longrightarrow}\\[-0.3ex]
\underset{\mu_{s}^-}{\longleftarrow}
}} $};
				% Text Node
				\draw (574,102) node [anchor=north west][inner sep=0.75pt]   [align=left] {$s$};
				% Text Node
				\draw (71,50) node [anchor=north west][inner sep=0.75pt]   [align=left] {$\cdots$};
				% Text Node
				\draw (173,50) node [anchor=north west][inner sep=0.75pt]   [align=left] {$\cdots$};
				% Text Node
				\draw (166,188) node [anchor=north west][inner sep=0.75pt]   [align=left] {$\cdots$};
				% Text Node
				\draw (270,188) node [anchor=north west][inner sep=0.75pt]   [align=left] {$\cdots$};
				% Text Node
				\draw (442,50) node [anchor=north west][inner sep=0.75pt]   [align=left] {$\cdots$};
				% Text Node
				\draw (543,50) node [anchor=north west][inner sep=0.75pt]   [align=left] {$\cdots$};
				% Text Node
				\draw (533,189) node [anchor=north west][inner sep=0.75pt]   [align=left] {$\cdots$};
				% Text Node
				\draw (636,185) node [anchor=north west][inner sep=0.75pt]   [align=left] {$\cdots$};

			\end{tikzpicture}

			\smallskip

				\begin{tikzpicture}[x=0.75pt,y=0.75pt,yscale=-1,xscale=1]
				%uncomment if require: \path (0,479); %set diagram left start at 0, and has height of 479
				
				%Shape: Circle [id:dp39909523582360173] 
				\draw  [fill={rgb, 255:red, 0; green, 0; blue, 0 }  ,fill opacity=1 ] (170,90) .. controls (170,87.24) and (172.24,85) .. (175,85) .. controls (177.76,85) and (180,87.24) .. (180,90) .. controls (180,92.76) and (177.76,95) .. (175,95) .. controls (172.24,95) and (170,92.76) .. (170,90) -- cycle ;
				%Straight Lines [id:da2634289074516307] 
				\draw [color={rgb, 255:red, 0; green, 0; blue, 255 }  ,draw opacity=1 ][line width=1.5]    (175,90) -- (175,155.5) ;
				%Shape: Circle [id:dp5555539095631987] 
				\draw  [fill={rgb, 255:red, 0; green, 0; blue, 0 }  ,fill opacity=1 ] (300.25,190.08) .. controls (300.26,192.84) and (298.04,195.09) .. (295.28,195.11) .. controls (292.51,195.12) and (290.26,192.89) .. (290.25,190.13) .. controls (290.23,187.37) and (292.46,185.12) .. (295.22,185.11) .. controls (297.98,185.09) and (300.23,187.32) .. (300.25,190.08) -- cycle ;
				%Shape: Circle [id:dp17611778165090142] 
				\draw  [fill={rgb, 255:red, 0; green, 0; blue, 0 }  ,fill opacity=1 ] (250.25,190.34) .. controls (250.26,193.1) and (248.04,195.35) .. (245.28,195.37) .. controls (242.51,195.38) and (240.26,193.16) .. (240.25,190.39) .. controls (240.24,187.63) and (242.46,185.38) .. (245.22,185.37) .. controls (247.98,185.35) and (250.24,187.58) .. (250.25,190.34) -- cycle ;
				%Straight Lines [id:da7982591909607111] 
				\draw [line width=1.5]    (295.25,190.11) -- (275.07,155.21) -- (245.25,190.37) ;
				%Shape: Circle [id:dp45756942985381555] 
				\draw  [fill={rgb, 255:red, 0; green, 0; blue, 0 }  ,fill opacity=1 ] (280.07,155.19) .. controls (280.08,157.95) and (277.85,160.2) .. (275.09,160.21) .. controls (272.33,160.23) and (270.08,158) .. (270.07,155.24) .. controls (270.05,152.48) and (272.28,150.23) .. (275.04,150.21) .. controls (277.8,150.2) and (280.05,152.42) .. (280.07,155.19) -- cycle ;
				%Shape: Circle [id:dp36589728520424525] 
				\draw  [fill={rgb, 255:red, 0; green, 0; blue, 0 }  ,fill opacity=1 ] (200.25,190.6) .. controls (200.26,193.37) and (198.04,195.62) .. (195.28,195.63) .. controls (192.52,195.64) and (190.26,193.42) .. (190.25,190.66) .. controls (190.24,187.9) and (192.46,185.64) .. (195.22,185.63) .. controls (197.99,185.62) and (200.24,187.84) .. (200.25,190.6) -- cycle ;
				%Shape: Circle [id:dp8575087725431652] 
				\draw  [fill={rgb, 255:red, 0; green, 0; blue, 0 }  ,fill opacity=1 ] (150.25,190.87) .. controls (150.27,193.63) and (148.04,195.88) .. (145.28,195.89) .. controls (142.52,195.91) and (140.27,193.68) .. (140.25,190.92) .. controls (140.24,188.16) and (142.46,185.91) .. (145.22,185.89) .. controls (147.99,185.88) and (150.24,188.1) .. (150.25,190.87) -- cycle ;
				%Straight Lines [id:da22470054015043983] 
				\draw [line width=1.5]    (195.25,190.63) -- (175.07,155.74) -- (145.25,190.89) ;
				%Shape: Circle [id:dp7921261977334222] 
				\draw  [fill={rgb, 255:red, 0; green, 0; blue, 0 }  ,fill opacity=1 ] (180.07,155.71) .. controls (180.08,158.47) and (177.85,160.72) .. (175.09,160.74) .. controls (172.33,160.75) and (170.08,158.52) .. (170.07,155.76) .. controls (170.05,153) and (172.28,150.75) .. (175.04,150.74) .. controls (177.8,150.72) and (180.05,152.95) .. (180.07,155.71) -- cycle ;
				%Straight Lines [id:da6675942361657408] 
				\draw [line width=1.5]    (275.07,155.21) -- (175.07,155.74) ;
				%Shape: Circle [id:dp4275305196375967] 
				\draw  [fill={rgb, 255:red, 0; green, 0; blue, 0 }  ,fill opacity=1 ] (594.25,194.08) .. controls (594.26,196.84) and (592.04,199.09) .. (589.28,199.11) .. controls (586.51,199.12) and (584.26,196.89) .. (584.25,194.13) .. controls (584.23,191.37) and (586.46,189.12) .. (589.22,189.11) .. controls (591.98,189.09) and (594.23,191.32) .. (594.25,194.08) -- cycle ;
				%Shape: Circle [id:dp09576828796536874] 
				\draw  [fill={rgb, 255:red, 0; green, 0; blue, 0 }  ,fill opacity=1 ] (544.25,194.34) .. controls (544.26,197.1) and (542.04,199.35) .. (539.28,199.37) .. controls (536.51,199.38) and (534.26,197.16) .. (534.25,194.39) .. controls (534.24,191.63) and (536.46,189.38) .. (539.22,189.37) .. controls (541.98,189.35) and (544.24,191.58) .. (544.25,194.34) -- cycle ;
				%Straight Lines [id:da46689319724200384] 
				\draw [line width=1.5]    (589.25,194.11) -- (569.07,159.21) -- (539.25,194.37) ;
				%Shape: Circle [id:dp7979982659667229] 
				\draw  [fill={rgb, 255:red, 0; green, 0; blue, 0 }  ,fill opacity=1 ] (574.07,159.19) .. controls (574.08,161.95) and (571.85,164.2) .. (569.09,164.21) .. controls (566.33,164.23) and (564.08,162) .. (564.07,159.24) .. controls (564.05,156.48) and (566.28,154.23) .. (569.04,154.21) .. controls (571.8,154.2) and (574.05,156.42) .. (574.07,159.19) -- cycle ;
				%Shape: Circle [id:dp05121931006424041] 
				\draw  [fill={rgb, 255:red, 0; green, 0; blue, 0 }  ,fill opacity=1 ] (494.25,194.6) .. controls (494.26,197.37) and (492.04,199.62) .. (489.28,199.63) .. controls (486.52,199.64) and (484.26,197.42) .. (484.25,194.66) .. controls (484.24,191.9) and (486.46,189.64) .. (489.22,189.63) .. controls (491.99,189.62) and (494.24,191.84) .. (494.25,194.6) -- cycle ;
				%Shape: Circle [id:dp742376625234417] 
				\draw  [fill={rgb, 255:red, 0; green, 0; blue, 0 }  ,fill opacity=1 ] (444.25,194.87) .. controls (444.27,197.63) and (442.04,199.88) .. (439.28,199.89) .. controls (436.52,199.91) and (434.27,197.68) .. (434.25,194.92) .. controls (434.24,192.16) and (436.46,189.91) .. (439.22,189.89) .. controls (441.99,189.88) and (444.24,192.1) .. (444.25,194.87) -- cycle ;
				%Straight Lines [id:da7471466121888295] 
				\draw [line width=1.5]    (489.25,194.63) -- (469.07,159.74) -- (439.25,194.89) ;
				%Shape: Circle [id:dp692218505072326] 
				\draw  [fill={rgb, 255:red, 0; green, 0; blue, 0 }  ,fill opacity=1 ] (474.07,159.71) .. controls (474.08,162.47) and (471.85,164.72) .. (469.09,164.74) .. controls (466.33,164.75) and (464.08,162.52) .. (464.07,159.76) .. controls (464.05,157) and (466.28,154.75) .. (469.04,154.74) .. controls (471.8,154.72) and (474.05,156.95) .. (474.07,159.71) -- cycle ;
				%Straight Lines [id:da6195120294149439] 
				\draw [line width=1.5]    (569.07,159.21) -- (469.07,159.74) ;
				%Shape: Circle [id:dp3160252607813181] 
				\draw  [fill={rgb, 255:red, 0; green, 0; blue, 0 }  ,fill opacity=1 ]
(464.07,90) .. controls (464.07,87.24) and (466.31,85) .. (469.07,85)
.. controls (471.83,85) and (474.07,87.24) .. (474.07,90)
.. controls (474.07,92.76) and (471.83,95) .. (469.07,95)
.. controls (466.31,95) and (464.07,92.76) .. (464.07,90) -- cycle ;
				%\draw  [fill={rgb, 255:red, 0; green, 0; blue, 0 }  ,fill opacity=1 ] (564,90) .. controls (564,87.24) and (566.24,85) .. (569,85) .. controls (571.76,85) and (574,87.24) .. (574,90) .. controls (574,92.76) and (571.76,95) .. (569,95) .. controls (566.24,95) and (564,92.76) .. (564,90) -- cycle ;
				%Straight Lines [id:da682935712642375] 
				%\draw [color={rgb, 255:red, 0; green, 0; blue, 255 }  ,draw opacity=1 ][line width=1.5]    (569,90) -- (569,155.5) ;
				\draw [color={rgb,255:red,0;green,0;blue,255}, draw opacity=1, line width=1.5]
(469.07,90)
.. controls (475,120) and (550,125)
.. (569.07,159.21);
				
				%\draw [color={rgb, 255:red, 0; green, 0; blue, 255 }  ,draw opacity=1 ][line width=1.5]    (643.04,150.21) .. controls (626,112.5) and (474,130.5) .. (443,90) ;
				
				% Text Node
				\draw (178,115) node [anchor=north west][inner sep=0.75pt]   [align=left] {$s'$};
				% Text Node
				\draw (342,130) node [anchor=north west][inner sep=0.75pt]  [font=\LARGE] [align=left] {$ \mathrel{\substack{
\overset{\mu_{s'}^+}{\longrightarrow}\\[-0.3ex]
\underset{\mu_{s}^-}{\longleftarrow}
}} $};
				% Text Node
				\draw (166,188) node [anchor=north west][inner sep=0.75pt]   [align=left] {$\cdots$};
				% Text Node
				\draw (270,188) node [anchor=north west][inner sep=0.75pt]   [align=left] {$\cdots$};
				% Text Node
				\draw (459,193) node [anchor=north west][inner sep=0.75pt]   [align=left] {$\cdots$};
				% Text Node
				\draw (562,189) node [anchor=north west][inner sep=0.75pt]   [align=left] {$\cdots$};
				% Text Node
				\draw (520,115) node [anchor=north west][inner sep=0.75pt]   [align=left] {$s$};

			\end{tikzpicture}

					\smallskip   
			
			\begin{tikzpicture}[x=0.75pt,y=0.75pt,yscale=-1,xscale=1]
				%uncomment if require: \path (0,479); %set diagram left start at 0, and has height of 479
				
				%Shape: Circle [id:dp5555539095631987] 
				\draw  [fill={rgb, 255:red, 0; green, 0; blue, 0 }  ,fill opacity=1 ] (300.25,190.08) .. controls (300.26,192.84) and (298.04,195.09) .. (295.28,195.11) .. controls (292.51,195.12) and (290.26,192.89) .. (290.25,190.13) .. controls (290.23,187.37) and (292.46,185.12) .. (295.22,185.11) .. controls (297.98,185.09) and (300.23,187.32) .. (300.25,190.08) -- cycle ;
				%Shape: Circle [id:dp17611778165090142] 
				\draw  [fill={rgb, 255:red, 0; green, 0; blue, 0 }  ,fill opacity=1 ] (250.25,190.34) .. controls (250.26,193.1) and (248.04,195.35) .. (245.28,195.37) .. controls (242.51,195.38) and (240.26,193.16) .. (240.25,190.39) .. controls (240.24,187.63) and (242.46,185.38) .. (245.22,185.37) .. controls (247.98,185.35) and (250.24,187.58) .. (250.25,190.34) -- cycle ;
				%Straight Lines [id:da7982591909607111] 
				\draw [line width=1.5]    (295.25,190.11) -- (275.07,155.21) -- (245.25,190.37) ;
				%Shape: Circle [id:dp45756942985381555] 
				\draw  [fill={rgb, 255:red, 0; green, 0; blue, 0 }  ,fill opacity=1 ] (280.07,155.19) .. controls (280.08,157.95) and (277.85,160.2) .. (275.09,160.21) .. controls (272.33,160.23) and (270.08,158) .. (270.07,155.24) .. controls (270.05,152.48) and (272.28,150.23) .. (275.04,150.21) .. controls (277.8,150.2) and (280.05,152.42) .. (280.07,155.19) -- cycle ;
				%Shape: Circle [id:dp36589728520424525] 
				\draw  [fill={rgb, 255:red, 0; green, 0; blue, 0 }  ,fill opacity=1 ] (200.25,190.6) .. controls (200.26,193.37) and (198.04,195.62) .. (195.28,195.63) .. controls (192.52,195.64) and (190.26,193.42) .. (190.25,190.66) .. controls (190.24,187.9) and (192.46,185.64) .. (195.22,185.63) .. controls (197.99,185.62) and (200.24,187.84) .. (200.25,190.6) -- cycle ;
				%Shape: Circle [id:dp8575087725431652] 
				\draw  [fill={rgb, 255:red, 0; green, 0; blue, 0 }  ,fill opacity=1 ] (150.25,190.87) .. controls (150.27,193.63) and (148.04,195.88) .. (145.28,195.89) .. controls (142.52,195.91) and (140.27,193.68) .. (140.25,190.92) .. controls (140.24,188.16) and (142.46,185.91) .. (145.22,185.89) .. controls (147.99,185.88) and (150.24,188.1) .. (150.25,190.87) -- cycle ;
				%Straight Lines [id:da22470054015043983] 
				\draw [line width=1.5]    (195.25,190.63) -- (175.07,155.74) -- (145.25,190.89) ;
				%Shape: Circle [id:dp7921261977334222] 
				\draw  [fill={rgb, 255:red, 0; green, 0; blue, 0 }  ,fill opacity=1 ] (180.07,155.71) .. controls (180.08,158.47) and (177.85,160.72) .. (175.09,160.74) .. controls (172.33,160.75) and (170.08,158.52) .. (170.07,155.76) .. controls (170.05,153) and (172.28,150.75) .. (175.04,150.74) .. controls (177.8,150.72) and (180.05,152.95) .. (180.07,155.71) -- cycle ;
				%Straight Lines [id:da6675942361657408] 
				\draw [line width=1.5]    (275.07,155.21) -- (175.07,155.74) ;
				%Shape: Circle [id:dp4275305196375967] 
				\draw  [fill={rgb, 255:red, 0; green, 0; blue, 0 }  ,fill opacity=1 ] (594.25,194.08) .. controls (594.26,196.84) and (592.04,199.09) .. (589.28,199.11) .. controls (586.51,199.12) and (584.26,196.89) .. (584.25,194.13) .. controls (584.23,191.37) and (586.46,189.12) .. (589.22,189.11) .. controls (591.98,189.09) and (594.23,191.32) .. (594.25,194.08) -- cycle ;
				%Shape: Circle [id:dp09576828796536874] 
				\draw  [fill={rgb, 255:red, 0; green, 0; blue, 0 }  ,fill opacity=1 ] (544.25,194.34) .. controls (544.26,197.1) and (542.04,199.35) .. (539.28,199.37) .. controls (536.51,199.38) and (534.26,197.16) .. (534.25,194.39) .. controls (534.24,191.63) and (536.46,189.38) .. (539.22,189.37) .. controls (541.98,189.35) and (544.24,191.58) .. (544.25,194.34) -- cycle ;
				%Straight Lines [id:da46689319724200384] 
				\draw [line width=1.5]    (589.25,194.11) -- (569.07,159.21) -- (539.25,194.37) ;
				%Shape: Circle [id:dp7979982659667229] 
				\draw  [fill={rgb, 255:red, 0; green, 0; blue, 0 }  ,fill opacity=1 ] (574.07,159.19) .. controls (574.08,161.95) and (571.85,164.2) .. (569.09,164.21) .. controls (566.33,164.23) and (564.08,162) .. (564.07,159.24) .. controls (564.05,156.48) and (566.28,154.23) .. (569.04,154.21) .. controls (571.8,154.2) and (574.05,156.42) .. (574.07,159.19) -- cycle ;
				%Shape: Circle [id:dp05121931006424041] 
				\draw  [fill={rgb, 255:red, 0; green, 0; blue, 0 }  ,fill opacity=1 ] (494.25,194.6) .. controls (494.26,197.37) and (492.04,199.62) .. (489.28,199.63) .. controls (486.52,199.64) and (484.26,197.42) .. (484.25,194.66) .. controls (484.24,191.9) and (486.46,189.64) .. (489.22,189.63) .. controls (491.99,189.62) and (494.24,191.84) .. (494.25,194.6) -- cycle ;
				%Shape: Circle [id:dp742376625234417] 
				\draw  [fill={rgb, 255:red, 0; green, 0; blue, 0 }  ,fill opacity=1 ] (444.25,194.87) .. controls (444.27,197.63) and (442.04,199.88) .. (439.28,199.89) .. controls (436.52,199.91) and (434.27,197.68) .. (434.25,194.92) .. controls (434.24,192.16) and (436.46,189.91) .. (439.22,189.89) .. controls (441.99,189.88) and (444.24,192.1) .. (444.25,194.87) -- cycle ;
				%Straight Lines [id:da7471466121888295] 
				\draw [line width=1.5]    (489.25,194.63) -- (469.07,159.74) -- (439.25,194.89) ;
				%Shape: Circle [id:dp692218505072326] 
				\draw  [fill={rgb, 255:red, 0; green, 0; blue, 0 }  ,fill opacity=1 ] (474.07,159.71) .. controls (474.08,162.47) and (471.85,164.72) .. (469.09,164.74) .. controls (466.33,164.75) and (464.08,162.52) .. (464.07,159.76) .. controls (464.05,157) and (466.28,154.75) .. (469.04,154.74) .. controls (471.8,154.72) and (474.05,156.95) .. (474.07,159.71) -- cycle ;
				%Straight Lines [id:da6195120294149439] 
				\draw [line width=1.5]    (569.07,159.21) -- (469.07,159.74) ;
				%Curve Lines [id:da32475312783451193] 
				\draw [color={rgb, 255:red, 0; green, 0; blue, 255 }  ,draw opacity=1 ][line width=1.5]    (175.07,155.74) .. controls (271,78.5) and (72,78.5) .. (175.07,155.74) ;
				%Curve Lines [id:da5252910438579086] 
				\draw [color={rgb, 255:red, 0; green, 0; blue, 255 }  ,draw opacity=1 ][line width=1.5]    (569.07,159.21) .. controls (665,81.98) and (466,81.98) .. (569.07,159.21) ;
				
				% Text Node
				\draw (342,130) node [anchor=north west][inner sep=0.75pt]  [font=\LARGE] [align=left] {$ \mathrel{\substack{
\overset{\mu_{s'}^+}{\longrightarrow}\\[-0.3ex]
\underset{\mu_{s}^-}{\longleftarrow}
}} $};
				% Text Node
				\draw (166,188) node [anchor=north west][inner sep=0.75pt]   [align=left] {$\cdots$};
				% Text Node
				\draw (270,188) node [anchor=north west][inner sep=0.75pt]   [align=left] {$\cdots$};
				% Text Node
				\draw (459,193) node [anchor=north west][inner sep=0.75pt]   [align=left] {$\cdots$};
				% Text Node
				\draw (562,189) node [anchor=north west][inner sep=0.75pt]   [align=left] {$\cdots$};
				% Text Node
				\draw (167,82) node [anchor=north west][inner sep=0.75pt]   [align=left] {$s'$};
				% Text Node
				\draw (562,85) node [anchor=north west][inner sep=0.75pt]   [align=left] {$s$};

			\end{tikzpicture}
			\smallskip	
		\end{center}
					\caption{Kauer moves on ribbon graphs with clockwise cyclic orderings at all vertices.}
		\label{fig:Kauer-moves}	
	\end{figure}
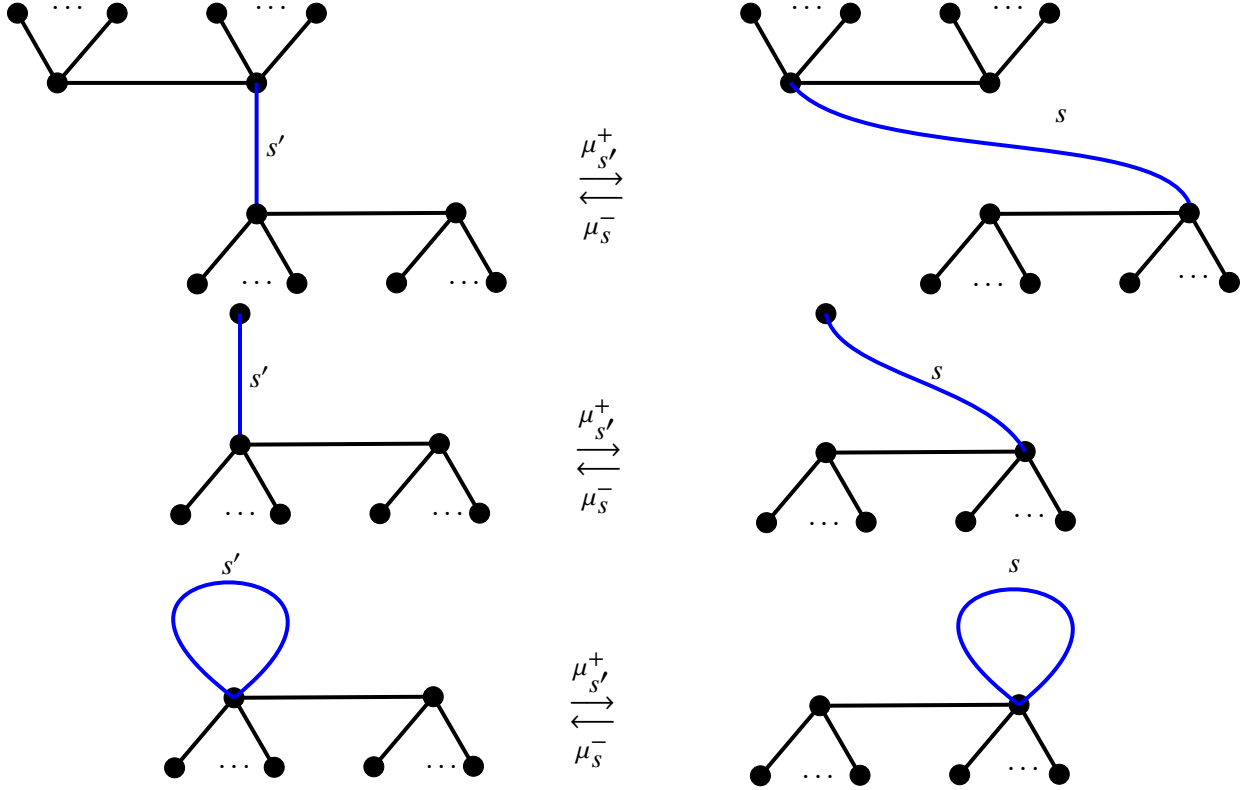	
	We call such a local move in Figure~\ref{fig:Kauer-moves} a {\it Kauer move} at $s$, which provides an explicit description of irreducible tilting mutations of a Brauer graph algebra.

More generally, Soto recently gave an explicit description of tilting mutations induced by general Okuyama--Rickard complexes (equivalently, left mutations) \cite{So1,So2}, called {\it generalized Kauer moves}. More precisely, let $H'$ be an $\iota$-stable set of half-edges of $\Gamma$, corresponding to a set $I$ of indecomposable projective $A$-modules. Then the endomorphism algebra
$\mathrm{End}_{\mathcal{K}^b(\mathrm{proj}\,A)}(\mu^-_{P(I)}(A))$
is again a BGA associated with the Brauer graph obtained by applying local moves that simultaneously shift all cyclically consecutive subsets of $H'$, as illustrated in Figure~\ref{fig:gen-Kauer-move}. For an explicit definition of these local moves, see \cite[Definition~1.13]{So1}.

	\subsection{Kauer moves for biserial FBGAs}\label{subsec:KM-FMS-BGA}
	\

We now begin to study irreducible tilting mutations (that is, $\nu_A$-irreducible mutations as defined in Definition~\ref{def:mutation}) of a given biserial FBGA $A$ with associated biserial FBG $(\Gamma,d)$, where $\nu_A$ denotes the Nakayama permutation of $A$. By Proposition~\ref{prop:bfBGA-basic}, $\nu_A$ is induced by the Nakayama automorphism
\[
\nu: H \longrightarrow H, \qquad
h \longmapsto \rho^{d(s(h))}h
\]
of $(\Gamma,d)$. Therefore, studying irreducible tilting mutations is equivalent to considering mutations along each $\langle \nu \rangle$-orbit of edges in $\Gamma$, or equivalently, along the corresponding pair of $\langle \nu \rangle$-orbits of half-edges in $\Gamma$, namely
\[
\{h_1,\ldots,h_n\}
\quad\text{and}\quad
\{\iota(h_1),\ldots,\iota(h_n)\}.
\]
The following lemma gives a description of the $\nu$-orbits in the graph $\Gamma$.

\begin{lemma}\label{lem:nu-orbit}
	Let $(\Gamma,d)$ be a biserial FBG, and let $\{\overline{h_1},\ldots,\overline{h_n}\}$
be a $\nu$-orbit of edges. Assume that $\Gamma$ contains edges other than the $n$ edges in this orbit. Then this $\nu$-orbit can occur in only one of the following three forms:

\begin{enumerate}
\item For each edge $\overline{h_i}$, neither $\overline{\rho(h_i)}$ nor $\overline{\rho(\iota(h_i))}$ belongs to the given $\nu$-orbit.

\item For each edge $\overline{h_i}$, at least one of $\overline{\rho(h_i)}$ and $\overline{\rho(\iota(h_i))}$ belongs to the given $\nu$-orbit, and each $\overline{h_i}$ joins two distinct vertices.

\item For each edge $\overline{h_i}$, at least one of $\overline{\rho(h_i)}$ and $\overline{\rho(\iota(h_i))}$ belongs to the given $\nu$-orbit, and each $\overline{h_i}$ is a loop, that is, its two endpoints coincide.
\end{enumerate}
Therefore, from a graphical point of view, every $\nu$-orbit has one of the forms illustrated in Figure~3.
\end{lemma}

\begin{proof}
We first assume that for each edge $\overline{h_i}$, at least one of 
$\overline{\rho(h_i)}$ and $\overline{\rho(\iota(h_i))}$ belongs to the given $\nu$-orbit; otherwise, we are in Case~(1).

Suppose moreover that every edge in the orbit joins two distinct vertices $\vv_1$ and $\vv_2$. Then either
\[
val(\vv_1)=n
\qquad\text{or}\qquad
val(\vv_2)=n,
\]
and hence Case~(2) holds.

Now suppose that some edge in this $\langle \nu \rangle$-orbit is a loop; that is, there exists a half-edge $h$ such that
\[
\rho(h)=\iota(h).
\]
Without loss of generality, we assume that $h=h_1$. Since all $h_i$ lie in the same $\langle \nu \rangle$-orbit, we may assume without loss of generality that
\[
h_2=\rho^{d(\vv)}(h_1),
\qquad \text{where } \vv=s(h_1).
\]
By condition (SI) in the definition of a biserial FBG, we have
\[
\iota(h_2)
=\iota\rho^{d(\vv)}(h_1)
=\rho^{d(\vv)}(\iota(h_1))
=\rho^{d(\vv)+1}(h_1)
=\rho(h_2).
\]
Thus $\rho(h_2)=\iota(h_2)$. Repeating the same argument inductively gives
\[
\rho(h_i)=\iota(h_i)
\qquad
\text{for all }1\le i\le n.
\]
Therefore each edge $\overline{h_i}$ is a loop, and hence we are in Case~(3).
\end{proof}

Now let $A$ be the biserial FBGA associated with a biserial FBG $(\Gamma,d)$. By Lemma~\ref{lem:nu-orbit}, each $\nu$-orbit has one of the forms described there, and accordingly each irreducible tilting mutation can be described by its ribbon graph $\Gamma$ as follows.

Since $A$ is self-injective and its Nakayama automorphism $\nu_A$ is induced by $\nu$ of $(\Gamma,d)$, for each $\nu_A$-orbit $I$ of indecomposable projective $A$-modules, the endomorphism algebra
\[
B:=\mathrm{End}_{\mathcal{K}^b(\mathrm{proj}\,A)}(\mu^-_{P(I)}(A))
\]
is again a biserial FBGA associated with the ribbon graph
\[
\Gamma'=(\Gamma\setminus I)\cup I',
\]
where $I$ denotes the blue edges in the ribbon graph $\Gamma$ (always those on the right) corresponding to the indecomposable projective modules in $I$, and $I'$ is obtained from $\Gamma$ by performing one of the local moves illustrated in Figure~\ref{fig:1-KM-fms}. We note that the (fractional) multiplicity
$m=\frac{d}{\operatorname{val}}$
at each vertex is preserved.

\begin{figure}
	\centering
		\begin{center}
\begin{tikzpicture}[x=0.75pt,y=0.75pt,yscale=-1,xscale=1]
%uncomment if require: \path (0,235); %set diagram left start at 0, and has height of 235

%Shape: Circle [id:dp7748846988676599] 
\draw  [fill={rgb, 255:red, 0; green, 0; blue, 0 }  ,fill opacity=1 ] (21.05,57.25) .. controls (21.05,54.49) and (23.29,52.25) .. (26.05,52.25) .. controls (28.81,52.25) and (31.05,54.49) .. (31.05,57.25) .. controls (31.05,60.01) and (28.81,62.25) .. (26.05,62.25) .. controls (23.29,62.25) and (21.05,60.01) .. (21.05,57.25) -- cycle ;
%Curve Lines [id:da30932062120644077] 
\draw [line width=1.5]    (26.05,57.25) .. controls (1.55,28.25) and (-17.45,14.25) .. (176.05,57.25) ;
%Curve Lines [id:da5952065060106606] 
\draw [line width=1.5]    (26.05,57.25) .. controls (223.55,13.25) and (206.55,31.25) .. (176.05,57.25) ;
%Shape: Circle [id:dp6727692032861552] 
\draw  [fill={rgb, 255:red, 0; green, 0; blue, 0 }  ,fill opacity=1 ] (258.59,169.2) .. controls (258.58,171.96) and (256.34,174.19) .. (253.58,174.19) .. controls (250.82,174.18) and (248.58,171.94) .. (248.59,169.18) .. controls (248.6,166.41) and (250.84,164.18) .. (253.6,164.19) .. controls (256.36,164.19) and (258.6,166.44) .. (258.59,169.2) -- cycle ;
%Curve Lines [id:da5149546755405692] 
\draw [line width=1.5]    (253.59,169.19) .. controls (278.02,198.25) and (296.98,212.29) .. (103.59,168.82) ;
%Curve Lines [id:da6560152633559009] 
\draw [line width=1.5]    (253.59,169.19) .. controls (55.98,212.7) and (73.03,194.75) .. (103.59,168.82) ;
%Curve Lines [id:da1593483656913195] 
\draw [color={rgb, 255:red, 0; green, 0; blue, 255 }  ,draw opacity=1 ][line width=1.5]    (176.05,57.25) .. controls (171.05,23.25) and (148.05,29.25) .. (103.59,168.82) ;
%Curve Lines [id:da12886011351661075] 
\draw [color={rgb, 255:red, 0; green, 0; blue, 255 }  ,draw opacity=1 ][line width=1.5]    (103.59,168.82) .. controls (120.05,213.25) and (130.05,181.25) .. (176.05,57.25) ;
%Shape: Arc [id:dp979658826946201] 
\draw  [draw opacity=0][dash pattern={on 0.84pt off 2.51pt}] (162.1,51.72) .. controls (163.47,48.27) and (166.08,45.45) .. (169.38,43.81) -- (176.05,57.25) -- cycle ; \draw  [dash pattern={on 0.84pt off 2.51pt}] (162.1,51.72) .. controls (163.47,48.27) and (166.08,45.45) .. (169.38,43.81) ;  
%Shape: Arc [id:dp8173257926906683] 
\draw  [draw opacity=0] (173.29,42.5) .. controls (174.18,42.34) and (175.11,42.25) .. (176.05,42.25) .. controls (179.5,42.25) and (182.68,43.42) .. (185.22,45.38) -- (176.05,57.25) -- cycle ; \draw   (173.29,42.5) .. controls (174.18,42.34) and (175.11,42.25) .. (176.05,42.25) .. controls (179.5,42.25) and (182.68,43.42) .. (185.22,45.38) ;  
%Shape: Arc [id:dp2804036214627248] 
\draw  [draw opacity=0][dash pattern={on 0.84pt off 2.51pt}] (189.18,50) .. controls (190.37,52.15) and (191.05,54.62) .. (191.05,57.25) .. controls (191.05,65.53) and (184.33,72.25) .. (176.05,72.25) .. controls (175.18,72.25) and (174.34,72.18) .. (173.51,72.04) -- (176.05,57.25) -- cycle ; \draw  [dash pattern={on 0.84pt off 2.51pt}] (189.18,50) .. controls (190.37,52.15) and (191.05,54.62) .. (191.05,57.25) .. controls (191.05,65.53) and (184.33,72.25) .. (176.05,72.25) .. controls (175.18,72.25) and (174.34,72.18) .. (173.51,72.04) ;  
%Shape: Arc [id:dp6812097760042648] 
\draw  [draw opacity=0] (169.06,70.53) .. controls (164.34,68.03) and (161.1,63.09) .. (161.05,57.39) -- (176.05,57.25) -- cycle ; \draw   (169.06,70.53) .. controls (164.34,68.03) and (161.1,63.09) .. (161.05,57.39) ;  
%Shape: Circle [id:dp2807558750355861] 
\draw  [fill={rgb, 255:red, 0; green, 0; blue, 0 }  ,fill opacity=1 ] (171.05,57.25) .. controls (171.05,54.49) and (173.29,52.25) .. (176.05,52.25) .. controls (178.81,52.25) and (181.05,54.49) .. (181.05,57.25) .. controls (181.05,60.01) and (178.81,62.25) .. (176.05,62.25) .. controls (173.29,62.25) and (171.05,60.01) .. (171.05,57.25) -- cycle ;
\draw   (183.74,42.13) -- (185.89,45.12) -- (182.29,45.9) ;
\draw   (159.55,60.68) -- (161.07,57.33) -- (163.54,60.07) ;
%Shape: Arc [id:dp8200320230828264] 
\draw  [draw opacity=0][dash pattern={on 0.84pt off 2.51pt}] (91.04,176.62) .. controls (89.59,174.32) and (88.76,171.6) .. (88.76,168.68) .. controls (88.76,160.39) and (95.48,153.68) .. (103.76,153.68) .. controls (104.62,153.68) and (105.47,153.75) .. (106.29,153.89) -- (103.76,168.68) -- cycle ; \draw  [dash pattern={on 0.84pt off 2.51pt}] (91.04,176.62) .. controls (89.59,174.32) and (88.76,171.6) .. (88.76,168.68) .. controls (88.76,160.39) and (95.48,153.68) .. (103.76,153.68) .. controls (104.62,153.68) and (105.47,153.75) .. (106.29,153.89) ;  
%Shape: Arc [id:dp16854895416339333] 
\draw  [draw opacity=0] (110.32,155.19) .. controls (115.32,157.62) and (118.76,162.75) .. (118.76,168.68) .. controls (118.76,168.81) and (118.76,168.94) .. (118.76,169.07) -- (103.76,168.68) -- cycle ; \draw   (110.32,155.19) .. controls (115.32,157.62) and (118.76,162.75) .. (118.76,168.68) .. controls (118.76,168.81) and (118.76,168.94) .. (118.76,169.07) ;  
%Shape: Arc [id:dp3031867404174613] 
\draw  [draw opacity=0][dash pattern={on 0.84pt off 2.51pt}] (117.46,174.81) .. controls (116.32,177.33) and (114.51,179.49) .. (112.25,181.05) -- (103.76,168.68) -- cycle ; \draw  [dash pattern={on 0.84pt off 2.51pt}] (117.46,174.81) .. controls (116.32,177.33) and (114.51,179.49) .. (112.25,181.05) ;  
%Shape: Arc [id:dp032393507167176816] 
\draw  [draw opacity=0] (108.51,182.91) .. controls (107.01,183.41) and (105.42,183.68) .. (103.76,183.68) .. controls (100.29,183.68) and (97.1,182.5) .. (94.55,180.52) -- (103.76,168.68) -- cycle ; \draw   (108.51,182.91) .. controls (107.01,183.41) and (105.42,183.68) .. (103.76,183.68) .. controls (100.29,183.68) and (97.1,182.5) .. (94.55,180.52) ;  
\draw   (120.21,166.62) -- (118.15,169.68) -- (116.17,166.57) ;
\draw   (95.57,183.67) -- (93.99,180.33) -- (97.68,180.22) ;
%Shape: Arc [id:dp008212830453646403] 
\draw  [draw opacity=0][dash pattern={on 0.84pt off 2.51pt}] (41.03,56.61) .. controls (41.04,56.82) and (41.05,57.04) .. (41.05,57.25) .. controls (41.05,65.53) and (34.33,72.25) .. (26.05,72.25) .. controls (17.76,72.25) and (11.05,65.53) .. (11.05,57.25) .. controls (11.05,53.41) and (12.49,49.91) .. (14.86,47.26) -- (26.05,57.25) -- cycle ; \draw  [dash pattern={on 0.84pt off 2.51pt}] (41.03,56.61) .. controls (41.04,56.82) and (41.05,57.04) .. (41.05,57.25) .. controls (41.05,65.53) and (34.33,72.25) .. (26.05,72.25) .. controls (17.76,72.25) and (11.05,65.53) .. (11.05,57.25) .. controls (11.05,53.41) and (12.49,49.91) .. (14.86,47.26) ;  
%Shape: Arc [id:dp46315963068092625] 
\draw  [draw opacity=0][dash pattern={on 0.84pt off 2.51pt}] (18.78,44.12) .. controls (20.94,42.93) and (23.41,42.25) .. (26.05,42.25) .. controls (32.33,42.25) and (37.71,46.11) .. (39.94,51.59) -- (26.05,57.25) -- cycle ; \draw  [dash pattern={on 0.84pt off 2.51pt}] (18.78,44.12) .. controls (20.94,42.93) and (23.41,42.25) .. (26.05,42.25) .. controls (32.33,42.25) and (37.71,46.11) .. (39.94,51.59) ;  
%Shape: Arc [id:dp11722155441360593] 
\draw  [draw opacity=0][dash pattern={on 0.84pt off 2.51pt}] (238.74,167.08) .. controls (239.76,159.8) and (246.02,154.19) .. (253.59,154.19) .. controls (261.87,154.19) and (268.59,160.9) .. (268.59,169.19) .. controls (268.59,172.91) and (267.23,176.31) .. (264.99,178.93) -- (253.59,169.19) -- cycle ; \draw  [dash pattern={on 0.84pt off 2.51pt}] (238.74,167.08) .. controls (239.76,159.8) and (246.02,154.19) .. (253.59,154.19) .. controls (261.87,154.19) and (268.59,160.9) .. (268.59,169.19) .. controls (268.59,172.91) and (267.23,176.31) .. (264.99,178.93) ;  
%Shape: Arc [id:dp836381718448677] 
\draw  [draw opacity=0][dash pattern={on 0.84pt off 2.51pt}] (259.72,182.88) .. controls (257.84,183.72) and (255.77,184.19) .. (253.59,184.19) .. controls (247.06,184.19) and (241.5,180.01) .. (239.44,174.19) -- (253.59,169.19) -- cycle ; \draw  [dash pattern={on 0.84pt off 2.51pt}] (259.72,182.88) .. controls (257.84,183.72) and (255.77,184.19) .. (253.59,184.19) .. controls (247.06,184.19) and (241.5,180.01) .. (239.44,174.19) ;  
%Shape: Circle [id:dp32632152500141] 
\draw  [fill={rgb, 255:red, 0; green, 0; blue, 0 }  ,fill opacity=1 ] (108.76,168.69) .. controls (108.76,171.45) and (106.51,173.69) .. (103.75,173.68) .. controls (100.99,173.67) and (98.76,171.43) .. (98.76,168.67) .. controls (98.77,165.91) and (101.01,163.67) .. (103.77,163.68) .. controls (106.54,163.69) and (108.77,165.93) .. (108.76,168.69) -- cycle ;
%Curve Lines [id:da05911205311005707] 
\draw [line width=1.5]    (345.05,57.25) .. controls (320.55,28.25) and (301.55,14.25) .. (495.05,57.25) ;
%Curve Lines [id:da6599646649521269] 
\draw [line width=1.5]    (345.05,57.25) .. controls (542.55,13.25) and (525.55,31.25) .. (495.05,57.25) ;
%Curve Lines [id:da28725270954329285] 
\draw [line width=1.5]    (572.59,169.19) .. controls (597.02,198.25) and (615.98,212.29) .. (422.59,168.82) ;
%Curve Lines [id:da4275981956499484] 
\draw [line width=1.5]    (572.59,169.19) .. controls (374.98,212.7) and (392.03,194.75) .. (422.59,168.82) ;
%Curve Lines [id:da2175930828276178] 
\draw [color={rgb, 255:red, 0; green, 0; blue, 255 }  ,draw opacity=1 ][line width=1.5]    (345.05,57.25) .. controls (336,13.33) and (401.33,48) .. (572.59,169.19) ;
%Curve Lines [id:da7080963773312152] 
\draw [color={rgb, 255:red, 0; green, 0; blue, 255 }  ,draw opacity=1 ][line width=1.5]    (572.59,169.19) .. controls (579.33,225.33) and (430.67,130) .. (345.05,57.25) ;
%Shape: Arc [id:dp50478589710841] 
\draw  [draw opacity=0][dash pattern={on 0.84pt off 2.51pt}] (481.1,51.72) .. controls (483.3,46.18) and (488.72,42.25) .. (495.05,42.25) .. controls (498.69,42.25) and (502.03,43.55) .. (504.62,45.7) -- (495.05,57.25) -- cycle ; \draw  [dash pattern={on 0.84pt off 2.51pt}] (481.1,51.72) .. controls (483.3,46.18) and (488.72,42.25) .. (495.05,42.25) .. controls (498.69,42.25) and (502.03,43.55) .. (504.62,45.7) ;  
%Shape: Arc [id:dp813475476700859] 
\draw  [draw opacity=0] (337.7,44.17) .. controls (339.5,43.16) and (341.52,42.51) .. (343.67,42.31) -- (345.05,57.25) -- cycle ; \draw   (337.7,44.17) .. controls (339.5,43.16) and (341.52,42.51) .. (343.67,42.31) ;  
%Shape: Arc [id:dp39951851556676177] 
\draw  [draw opacity=0][dash pattern={on 0.84pt off 2.51pt}] (508.18,50) .. controls (509.37,52.15) and (510.05,54.62) .. (510.05,57.25) .. controls (510.05,65.53) and (503.33,72.25) .. (495.05,72.25) .. controls (486.79,72.25) and (480.09,65.57) .. (480.05,57.32) -- (495.05,57.25) -- cycle ; \draw  [dash pattern={on 0.84pt off 2.51pt}] (508.18,50) .. controls (509.37,52.15) and (510.05,54.62) .. (510.05,57.25) .. controls (510.05,65.53) and (503.33,72.25) .. (495.05,72.25) .. controls (486.79,72.25) and (480.09,65.57) .. (480.05,57.32) ;  
%Shape: Arc [id:dp5994292716738758] 
\draw  [draw opacity=0] (360.04,56.86) .. controls (360.05,56.99) and (360.05,57.12) .. (360.05,57.25) .. controls (360.05,59.97) and (359.32,62.53) .. (358.05,64.73) -- (345.05,57.25) -- cycle ; \draw   (360.04,56.86) .. controls (360.05,56.99) and (360.05,57.12) .. (360.05,57.25) .. controls (360.05,59.97) and (359.32,62.53) .. (358.05,64.73) ;  
%Shape: Circle [id:dp036569114043248696] 
\draw  [fill={rgb, 255:red, 0; green, 0; blue, 0 }  ,fill opacity=1 ] (490.05,57.25) .. controls (490.05,54.49) and (492.29,52.25) .. (495.05,52.25) .. controls (497.81,52.25) and (500.05,54.49) .. (500.05,57.25) .. controls (500.05,60.01) and (497.81,62.25) .. (495.05,62.25) .. controls (492.29,62.25) and (490.05,60.01) .. (490.05,57.25) -- cycle ;
\draw   (360.74,63.27) -- (357.58,65.15) -- (357.11,61.5) ;
\draw   (341.29,40.49) -- (344.45,42.39) -- (341.45,44.53) ;
%Shape: Arc [id:dp8993070378237973] 
\draw  [draw opacity=0][dash pattern={on 0.84pt off 2.51pt}] (410.04,176.62) .. controls (408.59,174.32) and (407.76,171.6) .. (407.76,168.68) .. controls (407.76,160.39) and (414.48,153.68) .. (422.76,153.68) .. controls (430.87,153.68) and (437.47,160.11) .. (437.75,168.14) -- (422.76,168.68) -- cycle ; \draw  [dash pattern={on 0.84pt off 2.51pt}] (410.04,176.62) .. controls (408.59,174.32) and (407.76,171.6) .. (407.76,168.68) .. controls (407.76,160.39) and (414.48,153.68) .. (422.76,153.68) .. controls (430.87,153.68) and (437.47,160.11) .. (437.75,168.14) ;  
%Shape: Arc [id:dp6080396560150785] 
\draw  [draw opacity=0] (557.6,169.8) .. controls (557.59,169.6) and (557.59,169.39) .. (557.59,169.19) .. controls (557.59,166.55) and (558.27,164.08) .. (559.46,161.93) -- (572.59,169.19) -- cycle ; \draw   (557.6,169.8) .. controls (557.59,169.6) and (557.59,169.39) .. (557.59,169.19) .. controls (557.59,166.55) and (558.27,164.08) .. (559.46,161.93) ;  
%Shape: Arc [id:dp1957779301123901] 
\draw  [draw opacity=0][dash pattern={on 0.84pt off 2.51pt}] (436.46,174.81) .. controls (434.11,180.04) and (428.86,183.68) .. (422.76,183.68) .. controls (418.97,183.68) and (415.51,182.27) .. (412.86,179.95) -- (422.76,168.68) -- cycle ; \draw  [dash pattern={on 0.84pt off 2.51pt}] (436.46,174.81) .. controls (434.11,180.04) and (428.86,183.68) .. (422.76,183.68) .. controls (418.97,183.68) and (415.51,182.27) .. (412.86,179.95) ;  
%Shape: Arc [id:dp845207386608607] 
\draw  [draw opacity=0] (581.11,181.53) .. controls (578.69,183.21) and (575.75,184.19) .. (572.59,184.19) .. controls (572.24,184.19) and (571.89,184.18) .. (571.55,184.15) -- (572.59,169.19) -- cycle ; \draw   (581.11,181.53) .. controls (578.69,183.21) and (575.75,184.19) .. (572.59,184.19) .. controls (572.24,184.19) and (571.89,184.18) .. (571.55,184.15) ;  
\draw   (556.31,164.01) -- (559.35,161.91) -- (560.06,165.53) ;
\draw   (574.08,186.34) -- (570.97,184.36) -- (574.02,182.29) ;
%Shape: Arc [id:dp5279646635487363] 
\draw  [draw opacity=0][dash pattern={on 0.84pt off 2.51pt}] (354.33,69.03) .. controls (351.78,71.05) and (348.55,72.25) .. (345.05,72.25) .. controls (336.76,72.25) and (330.05,65.53) .. (330.05,57.25) .. controls (330.05,53.36) and (331.53,49.81) .. (333.96,47.15) -- (345.05,57.25) -- cycle ; \draw  [dash pattern={on 0.84pt off 2.51pt}] (354.33,69.03) .. controls (351.78,71.05) and (348.55,72.25) .. (345.05,72.25) .. controls (336.76,72.25) and (330.05,65.53) .. (330.05,57.25) .. controls (330.05,53.36) and (331.53,49.81) .. (333.96,47.15) ;  
%Shape: Arc [id:dp7936139351257276] 
\draw  [draw opacity=0][dash pattern={on 0.84pt off 2.51pt}] (351.37,43.64) .. controls (354.79,45.24) and (357.51,48.08) .. (358.94,51.59) -- (345.05,57.25) -- cycle ; \draw  [dash pattern={on 0.84pt off 2.51pt}] (351.37,43.64) .. controls (354.79,45.24) and (357.51,48.08) .. (358.94,51.59) ;  
%Shape: Arc [id:dp6768377876285474] 
\draw  [draw opacity=0][dash pattern={on 0.84pt off 2.51pt}] (563.72,157.09) .. controls (566.2,155.27) and (569.27,154.19) .. (572.59,154.19) .. controls (580.87,154.19) and (587.59,160.9) .. (587.59,169.19) .. controls (587.59,172.46) and (586.54,175.48) .. (584.77,177.95) -- (572.59,169.19) -- cycle ; \draw  [dash pattern={on 0.84pt off 2.51pt}] (563.72,157.09) .. controls (566.2,155.27) and (569.27,154.19) .. (572.59,154.19) .. controls (580.87,154.19) and (587.59,160.9) .. (587.59,169.19) .. controls (587.59,172.46) and (586.54,175.48) .. (584.77,177.95) ;  
%Shape: Arc [id:dp7906308975244649] 
\draw  [draw opacity=0][dash pattern={on 0.84pt off 2.51pt}] (565.97,182.65) .. controls (562.56,180.97) and (559.88,178.03) .. (558.54,174.45) -- (572.59,169.19) -- cycle ; \draw  [dash pattern={on 0.84pt off 2.51pt}] (565.97,182.65) .. controls (562.56,180.97) and (559.88,178.03) .. (558.54,174.45) ;  
%Shape: Circle [id:dp9398878327333056] 
\draw  [fill={rgb, 255:red, 0; green, 0; blue, 0 }  ,fill opacity=1 ] (427.76,168.69) .. controls (427.76,171.45) and (425.51,173.69) .. (422.75,173.68) .. controls (419.99,173.67) and (417.76,171.43) .. (417.76,168.67) .. controls (417.77,165.91) and (420.01,163.67) .. (422.77,163.68) .. controls (425.54,163.69) and (427.77,165.93) .. (427.76,168.69) -- cycle ;
%Shape: Circle [id:dp5587909097110528] 
\draw  [fill={rgb, 255:red, 0; green, 0; blue, 0 }  ,fill opacity=1 ] (340.05,57.25) .. controls (340.05,54.49) and (342.29,52.25) .. (345.05,52.25) .. controls (347.81,52.25) and (350.05,54.49) .. (350.05,57.25) .. controls (350.05,60.01) and (347.81,62.25) .. (345.05,62.25) .. controls (342.29,62.25) and (340.05,60.01) .. (340.05,57.25) -- cycle ;
%Shape: Circle [id:dp6426043587862107] 
\draw  [fill={rgb, 255:red, 0; green, 0; blue, 0 }  ,fill opacity=1 ] (577.59,169.2) .. controls (577.58,171.96) and (575.34,174.19) .. (572.58,174.19) .. controls (569.82,174.18) and (567.58,171.94) .. (567.59,169.18) .. controls (567.6,166.41) and (569.84,164.18) .. (572.6,164.19) .. controls (575.36,164.19) and (577.6,166.44) .. (577.59,169.2) -- cycle ;

% Text Node
\draw (276.33,94.33) node [anchor=north west][inner sep=0.75pt]  [font=\LARGE] [align=left] {$ \mathrel{\substack{
\overset{\mu_P^+}{\longrightarrow}\\[-0.3ex]
\underset{\mu_{P}^-}{\longleftarrow}
}} $};
% Text Node
%\draw (175.51,75.04) node [anchor=north west][inner sep=0.75pt]   [align=left] {$h_1$};
% Text Node
%\draw (122.18,59.04) node [anchor=north west][inner sep=0.75pt]   [align=left] {$h_2$};
% Text Node
%\draw (142.18,156.37) node [anchor=north west][inner sep=0.75pt]   [align=left] {$\iota(h_1)$};
% Text Node
%\draw (80,132.37) node [anchor=north west][inner sep=0.75pt]   [align=left] {$\iota(h_2)$};
% Text Node
\draw (195.33,49) node [anchor=north west][inner sep=0.75pt]   [align=left] {$\vv_3$};
% Text Node
\draw (71.33,161.67) node [anchor=north west][inner sep=0.75pt]   [align=left] {$\vv_4$};
% Text Node
\draw (10,73) node [anchor=north west][inner sep=0.75pt]   [align=left] {$\vv_1$};
% Text Node
\draw (244.67,136.33) node [anchor=north west][inner sep=0.75pt]   [align=left] {$\vv_2$};
% Text Node
\draw (336.33,72.33) node [anchor=north west][inner sep=0.75pt]   [align=left] {$\vv_1$};
% Text Node
\draw (410.33,141) node [anchor=north west][inner sep=0.75pt]   [align=left] {$\vv_4$};
% Text Node
\draw (509.67,59) node [anchor=north west][inner sep=0.75pt]   [align=left] {$\vv_3$};
% Text Node
\draw (577,145) node [anchor=north west][inner sep=0.75pt]   [align=left] {$\vv_2$};

\begin{scope}[xshift=35pt]
	\begin{scope}[yshift=160pt]
%Shape: Arc [id:dp19825269849568516] 
\draw  [draw opacity=0] (31.11,99.18) .. controls (33.35,93.71) and (38.73,89.86) .. (45,89.86) .. controls (47.66,89.86) and (50.16,90.55) .. (52.33,91.77) -- (45,104.86) -- cycle ; \draw   (31.11,99.18) .. controls (33.35,93.71) and (38.73,89.86) .. (45,89.86) .. controls (47.66,89.86) and (50.16,90.55) .. (52.33,91.77) ;  
\draw   (59.26,114.19) -- (55.85,115.59) -- (55.98,111.9) ;
%Curve Lines [id:da10058334342291175] 
\draw [line width=1.5]    (45,105) .. controls (45,60) and (195,149.5) .. (195,105) ;
%Curve Lines [id:da27659345925155066] 
\draw [line width=1.5]    (45,105) .. controls (44.5,150) and (195,60) .. (195,105) ;
%Shape: Circle [id:dp35994590238836466] 
\draw  [color={rgb, 255:red, 0; green, 0; blue, 255 }  ,draw opacity=1 ][line width=1.5]  (5,65) .. controls (5,42.91) and (22.91,25) .. (45,25) .. controls (67.09,25) and (85,42.91) .. (85,65) .. controls (85,87.09) and (67.09,105) .. (45,105) .. controls (22.91,105) and (5,87.09) .. (5,65) -- cycle ;
%Shape: Circle [id:dp10865762141246571] 
\draw  [fill={rgb, 255:red, 0; green, 0; blue, 0 }  ,fill opacity=1 ] (40,105) .. controls (40,102.24) and (42.24,100) .. (45,100) .. controls (47.76,100) and (50,102.24) .. (50,105) .. controls (50,107.76) and (47.76,110) .. (45,110) .. controls (42.24,110) and (40,107.76) .. (40,105) -- cycle ;
%Shape: Circle [id:dp562607268575402] 
\draw  [fill={rgb, 255:red, 0; green, 0; blue, 0 }  ,fill opacity=1 ] (190,105) .. controls (190,102.24) and (192.24,100) .. (195,100) .. controls (197.76,100) and (200,102.24) .. (200,105) .. controls (200,107.76) and (197.76,110) .. (195,110) .. controls (192.24,110) and (190,107.76) .. (190,105) -- cycle ;
%Shape: Circle [id:dp8459331155326753] 
\draw [color={rgb,255:red,0;green,0;blue,255}, draw opacity=1, line width=1.5]
(315,25)
.. controls (340,100) and (410,95)
.. (465,105);
\draw [color={rgb,255:red,0;green,0;blue,255}, draw opacity=1, line width=1.5]
(315,25)
.. controls (350,-5) and (550,95)
.. (465,105);
\draw  [fill={rgb, 255:red, 0; green, 0; blue, 0 }  ,fill opacity=1 ] (40,25) .. controls (40,22.24) and (42.24,20) .. (45,20) .. controls (47.76,20) and (50,22.24) .. (50,25) .. controls (50,27.76) and (47.76,30) .. (45,30) .. controls (42.24,30) and (40,27.76) .. (40,25) -- cycle ;
%Curve Lines [id:da26964258702895405] 
\draw [line width=1.5]    (315,105) .. controls (315,60) and (465,149.5) .. (465,105) ;
%Curve Lines [id:da3040343775850778] 
\draw [line width=1.5]    (315,105) .. controls (314.5,150) and (465,60) .. (465,105) ;
%Shape: Circle [id:dp7847464979701646] 
%\draw  [color={rgb, 255:red, 0; green, 0; blue, 255 }  ,draw opacity=1 ][line width=1.5]  (425,65) .. controls (425,42.91) and (442.91,25) .. (465,25) .. controls (487.09,25) and (505,42.91) .. (505,65) .. controls (505,87.09) and (487.09,105) .. (465,105) .. controls (442.91,105) and (425,87.09) .. (425,65) -- cycle ;
%Shape: Circle [id:dp21914765106767375] 
\draw  [fill={rgb, 255:red, 0; green, 0; blue, 0 }  ,fill opacity=1 ] (310,105) .. controls (310,102.24) and (312.24,100) .. (315,100) .. controls (317.76,100) and (320,102.24) .. (320,105) .. controls (320,107.76) and (317.76,110) .. (315,110) .. controls (312.24,110) and (310,107.76) .. (310,105) -- cycle ;
%Shape: Circle [id:dp07562564482828749] 
\draw  [fill={rgb, 255:red, 0; green, 0; blue, 0 }  ,fill opacity=1 ] (460,105) .. controls (460,102.24) and (462.24,100) .. (465,100) .. controls (467.76,100) and (470,102.24) .. (470,105) .. controls (470,107.76) and (467.76,110) .. (465,110) .. controls (462.24,110) and (460,107.76) .. (460,105) -- cycle ;
%Shape: Circle [id:dp4700023322734219] 
%\draw  [fill={rgb, 255:red, 0; green, 0; blue, 0 }  ,fill opacity=1 ] (460,25) .. controls (460,22.24) and (462.24,20) .. (465,20) .. controls (467.76,20) and (470,22.24) .. (470,25) .. controls (470,27.76) and (467.76,30) .. (465,30) .. controls (462.24,30) and (460,27.76) .. (460,25) -- cycle ;
%Shape: Arc [id:dp4982976252170226] 
\draw  [draw opacity=0][dash pattern={on 0.84pt off 2.51pt}] (56.23,94.91) .. controls (58.07,96.99) and (59.35,99.59) .. (59.81,102.47) -- (45,104.86) -- cycle ; \draw  [dash pattern={on 0.84pt off 2.51pt}] (56.23,94.91) .. controls (58.07,96.99) and (59.35,99.59) .. (59.81,102.47) ;  
\draw  [fill={rgb, 255:red, 0; green, 0; blue, 0 }  ,fill opacity=1 ]
(310,25) .. controls (310,22.24) and (312.24,20) .. (315,20)
.. controls (317.76,20) and (320,22.24) .. (320,25)
.. controls (320,27.76) and (317.76,30) .. (315,30)
.. controls (312.24,30) and (310,27.76) .. (310,25) -- cycle ;
%Shape: Arc [id:dp6662454200820731] 
\draw  [draw opacity=0] (59.97,104) .. controls (59.99,104.33) and (60,104.67) .. (60,105) .. controls (60,108.8) and (58.59,112.27) .. (56.26,114.91) -- (45,105) -- cycle ; \draw   (59.97,104) .. controls (59.99,104.33) and (60,104.67) .. (60,105) .. controls (60,108.8) and (58.59,112.27) .. (56.26,114.91) ;  
%Shape: Arc [id:dp5888097613306145] 
\draw  [draw opacity=0][dash pattern={on 0.84pt off 2.51pt}] (51.71,118.42) .. controls (49.69,119.43) and (47.41,120) .. (45,120) .. controls (36.72,120) and (30,113.28) .. (30,105) .. controls (30,104.24) and (30.06,103.49) .. (30.17,102.76) -- (45,105) -- cycle ; \draw  [dash pattern={on 0.84pt off 2.51pt}] (51.71,118.42) .. controls (49.69,119.43) and (47.41,120) .. (45,120) .. controls (36.72,120) and (30,113.28) .. (30,105) .. controls (30,104.24) and (30.06,103.49) .. (30.17,102.76) ;  
\draw   (50.32,88.53) -- (52.09,91.77) -- (48.42,92.05) ;
%Shape: Arc [id:dp026563124742941424] 
\draw  [draw opacity=0][dash pattern={on 0.84pt off 2.51pt}] (30.19,22.59) .. controls (31.35,15.45) and (37.54,10) .. (45,10) .. controls (53.28,10) and (60,16.72) .. (60,25) .. controls (60,33.28) and (53.28,40) .. (45,40) .. controls (36.72,40) and (30,33.28) .. (30,25) .. controls (30,24.24) and (30.06,23.49) .. (30.17,22.76) -- (45,25) -- cycle ; \draw  [dash pattern={on 0.84pt off 2.51pt}] (30.19,22.59) .. controls (31.35,15.45) and (37.54,10) .. (45,10) .. controls (53.28,10) and (60,16.72) .. (60,25) .. controls (60,33.28) and (53.28,40) .. (45,40) .. controls (36.72,40) and (30,33.28) .. (30,25) .. controls (30,24.24) and (30.06,23.49) .. (30.17,22.76) ;  
%Shape: Circle [id:dp8298855859210115] 
\draw  [dash pattern={on 0.84pt off 2.51pt}][line width=0.75]  (180,105) .. controls (180,96.72) and (186.72,90) .. (195,90) .. controls (203.28,90) and (210,96.72) .. (210,105) .. controls (210,113.28) and (203.28,120) .. (195,120) .. controls (186.72,120) and (180,113.28) .. (180,105) -- cycle ;
%Shape: Circle [id:dp3992844782998235] 
\draw  [dash pattern={on 0.84pt off 2.51pt}][line width=0.75]  (300,105) .. controls (300,96.72) and (306.72,90) .. (315,90) .. controls (323.28,90) and (330,96.72) .. (330,105) .. controls (330,113.28) and (323.28,120) .. (315,120) .. controls (306.72,120) and (300,113.28) .. (300,105) -- cycle ;
%Shape: Circle [id:dp23853417071075445] 
\draw [dash pattern={on 0.84pt off 2.51pt}][line width=0.75]
(300,25)
.. controls (300,16.72) and (306.72,10) .. (315,10)
.. controls (323.28,10) and (330,16.72) .. (330,25)
.. controls (330,33.28) and (323.28,40) .. (315,40)
.. controls (306.72,40) and (300,33.28) .. (300,25)
-- cycle ;
%\draw  [dash pattern={on 0.84pt off 2.51pt}][line width=0.75]  (450,25) .. controls (450,16.72) and (456.72,10) .. (465,10) .. controls (473.28,10) and (480,16.72) .. (480,25) .. controls (480,33.28) and (473.28,40) .. (465,40) .. controls (456.72,40) and (450,33.28) .. (450,25) -- cycle ;
%Shape: Arc [id:dp9315278893266354] 
\draw  [draw opacity=0] (459.01,92.09) .. controls (460.73,91.41) and (462.61,91.03) .. (464.57,91.03) .. controls (470.86,91.03) and (476.25,94.9) .. (478.48,100.39) -- (464.57,106.03) -- cycle ; \draw   (459.01,92.09) .. controls (460.73,91.41) and (462.61,91.03) .. (464.57,91.03) .. controls (470.86,91.03) and (476.25,94.9) .. (478.48,100.39) ;  
%Shape: Arc [id:dp2107638145200541] 
\draw  [draw opacity=0][dash pattern={on 0.84pt off 2.51pt}] (479.44,104.01) .. controls (479.53,104.67) and (479.57,105.34) .. (479.57,106.03) .. controls (479.57,114.31) and (472.86,121.03) .. (464.57,121.03) .. controls (460.97,121.03) and (457.67,119.76) .. (455.08,117.65) -- (464.57,106.03) -- cycle ; \draw  [dash pattern={on 0.84pt off 2.51pt}] (479.44,104.01) .. controls (479.53,104.67) and (479.57,105.34) .. (479.57,106.03) .. controls (479.57,114.31) and (472.86,121.03) .. (464.57,121.03) .. controls (460.97,121.03) and (457.67,119.76) .. (455.08,117.65) ;  
%Shape: Arc [id:dp8225567383332992] 
\draw  [draw opacity=0] (452.72,115.22) .. controls (450.75,112.68) and (449.57,109.49) .. (449.57,106.03) .. controls (449.57,105.5) and (449.6,104.98) .. (449.65,104.46) -- (464.57,106.03) -- cycle ; \draw   (452.72,115.22) .. controls (450.75,112.68) and (449.57,109.49) .. (449.57,106.03) .. controls (449.57,105.5) and (449.6,104.98) .. (449.65,104.46) ;  
%Shape: Arc [id:dp38184517931837947] 
\draw  [draw opacity=0][dash pattern={on 0.84pt off 2.51pt}] (450.45,100.96) .. controls (451.58,97.82) and (453.72,95.16) .. (456.49,93.39) -- (464.57,106.03) -- cycle ; \draw  [dash pattern={on 0.84pt off 2.51pt}] (450.45,100.96) .. controls (451.58,97.82) and (453.72,95.16) .. (456.49,93.39) ;  
\draw   (479.59,97.37) -- (478.38,100.85) -- (475.71,98.31) ;
\draw   (447.86,107.48) -- (449.84,104.37) -- (451.86,107.45) ;

% Text Node
\draw (229.75,47.5) node [anchor=north west][inner sep=0.75pt]  [font=\LARGE] [align=left] {$ \mathrel{\substack{
\overset{\mu_P^+}{\longrightarrow}\\[-0.3ex]
\underset{\mu_{P}^-}{\longleftarrow}
}} $};
% Text Node
\draw (36,124) node [anchor=north west][inner sep=0.75pt]   [align=left] {$\vv_3$};
% Text Node
\draw (63.75,8) node [anchor=north west][inner sep=0.75pt]   [align=left] {$\vv_2$};
% Text Node
\draw (188.75,124) node [anchor=north west][inner sep=0.75pt]   [align=left] {$\vv_1$};
% Text Node
%\draw (5,96) node [anchor=north west][inner sep=0.75pt]   [align=left] {$h_1$};
% Text Node
%\draw (85,80) node [anchor=north west][inner sep=0.75pt]   [align=left] {$h_2$};
% Text Node
\draw (307,124) node [anchor=north west][inner sep=0.75pt]   [align=left] {$\vv_3$};
% Text Node
\draw (287,6) node [anchor=north west][inner sep=0.75pt]   [align=left] {$\vv_2$};
% Text Node
\draw (456,124) node [anchor=north west][inner sep=0.75pt]   [align=left] {$\vv_1$};
\end{scope}
\end{scope}

	\begin{scope}[yshift=260pt]
%uncomment if require: \path (0,235); %set diagram left start at 0, and has height of 235

%Curve Lines [id:da26964258702895405] 
\draw [line width=1.5]    (66.45,63.84) .. controls (66.45,18.84) and (216.45,18.59) .. (216.45,63.84) ;
%Curve Lines [id:da3040343775850778] 
\draw [line width=1.5]    (66.45,63.84) .. controls (66.95,109.59) and (217.95,110.09) .. (216.45,63.84) ;
%Shape: Arc [id:dp8225567383332992] 
\draw  [draw opacity=0] (53.16,70.82) .. controls (52.07,68.74) and (51.45,66.36) .. (51.45,63.84) .. controls (51.45,62.04) and (51.77,60.31) .. (52.35,58.7) -- (66.45,63.84) -- cycle ; \draw   (53.16,70.82) .. controls (52.07,68.74) and (51.45,66.36) .. (51.45,63.84) .. controls (51.45,62.04) and (51.77,60.31) .. (52.35,58.7) ;  
%Shape: Arc [id:dp38184517931837947] 
\draw  [draw opacity=0][dash pattern={on 0.84pt off 2.51pt}] (72.86,50.28) .. controls (75.01,51.3) and (76.88,52.81) .. (78.32,54.67) -- (66.45,63.84) -- cycle ; \draw  [dash pattern={on 0.84pt off 2.51pt}] (72.86,50.28) .. controls (75.01,51.3) and (76.88,52.81) .. (78.32,54.67) ;  
\draw   (49.61,60.96) -- (52.34,58.48) -- (53.47,61.99) ;
%Curve Lines [id:da5634367753821239] 
\draw [color={rgb, 255:red, 0; green, 0; blue, 255 }  ,draw opacity=1 ][line width=1.5]    (66.45,61.84) .. controls (122.38,31.42) and (122.66,97.7) .. (66.45,65.84) ;
%Curve Lines [id:da8872491626724717] 
\draw [color={rgb, 255:red, 0; green, 0; blue, 255 }  ,draw opacity=1 ][line width=1.5]    (66.45,61.84) .. controls (16.23,35.13) and (16.8,96.27) .. (66.45,65.84) ;
%Shape: Circle [id:dp73045980503443] 
\draw  [fill={rgb, 255:red, 0; green, 0; blue, 0 }  ,fill opacity=1 ] (61.45,63.84) .. controls (61.45,61.08) and (63.69,58.84) .. (66.45,58.84) .. controls (69.21,58.84) and (71.45,61.08) .. (71.45,63.84) .. controls (71.45,66.61) and (69.21,68.84) .. (66.45,68.84) .. controls (63.69,68.84) and (61.45,66.61) .. (61.45,63.84) -- cycle ;
%Shape: Circle [id:dp5731178521934357] 
\draw  [fill={rgb, 255:red, 0; green, 0; blue, 0 }  ,fill opacity=1 ] (211.45,63.84) .. controls (211.45,61.08) and (213.69,58.84) .. (216.45,58.84) .. controls (219.21,58.84) and (221.45,61.08) .. (221.45,63.84) .. controls (221.45,66.61) and (219.21,68.84) .. (216.45,68.84) .. controls (213.69,68.84) and (211.45,66.61) .. (211.45,63.84) -- cycle ;
%Shape: Arc [id:dp16423722955129327] 
\draw  [draw opacity=0] (54.8,54.39) .. controls (57.55,51.01) and (61.75,48.84) .. (66.45,48.84) .. controls (67.57,48.84) and (68.67,48.97) .. (69.73,49.2) -- (66.45,63.84) -- cycle ; \draw   (54.8,54.39) .. controls (57.55,51.01) and (61.75,48.84) .. (66.45,48.84) .. controls (67.57,48.84) and (68.67,48.97) .. (69.73,49.2) ;  
%Shape: Arc [id:dp6722095931134748] 
\draw  [draw opacity=0] (80.04,57.5) .. controls (80.94,59.43) and (81.45,61.58) .. (81.45,63.84) .. controls (81.45,65.87) and (81.04,67.8) .. (80.32,69.57) -- (66.45,63.84) -- cycle ; \draw   (80.04,57.5) .. controls (80.94,59.43) and (81.45,61.58) .. (81.45,63.84) .. controls (81.45,65.87) and (81.04,67.8) .. (80.32,69.57) ;  
%Shape: Arc [id:dp9364535099735034] 
\draw  [draw opacity=0][dash pattern={on 0.84pt off 2.51pt}] (70.32,78.34) .. controls (69.09,78.67) and (67.79,78.84) .. (66.45,78.84) .. controls (61.82,78.84) and (57.68,76.75) .. (54.93,73.45) -- (66.45,63.84) -- cycle ; \draw  [dash pattern={on 0.84pt off 2.51pt}] (70.32,78.34) .. controls (69.09,78.67) and (67.79,78.84) .. (66.45,78.84) .. controls (61.82,78.84) and (57.68,76.75) .. (54.93,73.45) ;  
%Shape: Arc [id:dp93925628825885] 
\draw  [draw opacity=0] (78.15,73.23) .. controls (77.01,74.64) and (75.63,75.85) .. (74.06,76.77) -- (66.45,63.84) -- cycle ; \draw   (78.15,73.23) .. controls (77.01,74.64) and (75.63,75.85) .. (74.06,76.77) ;  
\draw   (68.05,46.89) -- (70.42,49.72) -- (66.87,50.71) ;
\draw   (83.51,67.26) -- (80.52,69.42) -- (79.78,65.8) ;
\draw   (77.27,76.96) -- (73.59,76.73) -- (75.32,73.47) ;
%Shape: Circle [id:dp7415223615105135] 
\draw  [dash pattern={on 0.84pt off 2.51pt}] (201.45,63.84) .. controls (201.45,55.56) and (208.16,48.84) .. (216.45,48.84) .. controls (224.73,48.84) and (231.45,55.56) .. (231.45,63.84) .. controls (231.45,72.13) and (224.73,78.84) .. (216.45,78.84) .. controls (208.16,78.84) and (201.45,72.13) .. (201.45,63.84) -- cycle ;
%Curve Lines [id:da9510757695694854] 
\draw [line width=1.5]    (366.45,63.84) .. controls (366.45,18.84) and (516.45,18.59) .. (516.45,63.84) ;
%Curve Lines [id:da21620394803610266] 
\draw [line width=1.5]    (366.45,63.84) .. controls (366.95,109.59) and (517.95,110.09) .. (516.45,63.84) ;
%Shape: Arc [id:dp6492206314936131] 
\draw  [draw opacity=0] (503.16,70.82) .. controls (502.07,68.74) and (501.45,66.36) .. (501.45,63.84) .. controls (501.45,62.04) and (501.77,60.31) .. (502.35,58.7) -- (516.45,63.84) -- cycle ; \draw   (503.16,70.82) .. controls (502.07,68.74) and (501.45,66.36) .. (501.45,63.84) .. controls (501.45,62.04) and (501.77,60.31) .. (502.35,58.7) ;  
%Shape: Arc [id:dp6146136579289707] 
\draw  [draw opacity=0][dash pattern={on 0.84pt off 2.51pt}] (528.19,73.17) .. controls (525.44,76.63) and (521.2,78.84) .. (516.45,78.84) .. controls (515.41,78.84) and (514.4,78.74) .. (513.43,78.54) -- (516.45,63.84) -- cycle ; \draw  [dash pattern={on 0.84pt off 2.51pt}] (528.19,73.17) .. controls (525.44,76.63) and (521.2,78.84) .. (516.45,78.84) .. controls (515.41,78.84) and (514.4,78.74) .. (513.43,78.54) ;  
\draw   (500.28,60.29) -- (503.01,57.81) -- (504.14,61.32) ;
%Curve Lines [id:da7390364694006315] 
\draw [color={rgb, 255:red, 0; green, 0; blue, 255 }  ,draw opacity=1 ][line width=1.5]    (516.78,61.51) .. controls (572.71,31.08) and (572.99,97.37) .. (516.78,65.51) ;
%Curve Lines [id:da6165900518532821] 
\draw [color={rgb, 255:red, 0; green, 0; blue, 255 }  ,draw opacity=1 ][line width=1.5]    (516.78,61.51) .. controls (466.57,34.8) and (467.14,95.94) .. (516.78,65.51) ;
%Shape: Circle [id:dp8911649905109431] 
\draw  [fill={rgb, 255:red, 0; green, 0; blue, 0 }  ,fill opacity=1 ] (361.45,63.84) .. controls (361.45,61.08) and (363.69,58.84) .. (366.45,58.84) .. controls (369.21,58.84) and (371.45,61.08) .. (371.45,63.84) .. controls (371.45,66.61) and (369.21,68.84) .. (366.45,68.84) .. controls (363.69,68.84) and (361.45,66.61) .. (361.45,63.84) -- cycle ;
%Shape: Circle [id:dp8557307286731404] 
\draw  [fill={rgb, 255:red, 0; green, 0; blue, 0 }  ,fill opacity=1 ] (511.45,63.84) .. controls (511.45,61.08) and (513.69,58.84) .. (516.45,58.84) .. controls (519.21,58.84) and (521.45,61.08) .. (521.45,63.84) .. controls (521.45,66.61) and (519.21,68.84) .. (516.45,68.84) .. controls (513.69,68.84) and (511.45,66.61) .. (511.45,63.84) -- cycle ;
%Shape: Arc [id:dp6066944172526132] 
\draw  [draw opacity=0] (509.55,77.17) .. controls (507.62,76.17) and (505.94,74.76) .. (504.61,73.06) -- (516.45,63.84) -- cycle ; \draw   (509.55,77.17) .. controls (507.62,76.17) and (505.94,74.76) .. (504.61,73.06) ;  
%Shape: Arc [id:dp1894259975513315] 
\draw  [draw opacity=0] (530.04,57.5) .. controls (530.94,59.43) and (531.45,61.58) .. (531.45,63.84) .. controls (531.45,65.87) and (531.04,67.8) .. (530.32,69.57) -- (516.45,63.84) -- cycle ; \draw   (530.04,57.5) .. controls (530.94,59.43) and (531.45,61.58) .. (531.45,63.84) .. controls (531.45,65.87) and (531.04,67.8) .. (530.32,69.57) ;  
%Shape: Arc [id:dp1699419620844025] 
\draw  [draw opacity=0][dash pattern={on 0.84pt off 2.51pt}] (504.62,54.62) .. controls (505.89,52.99) and (507.49,51.64) .. (509.31,50.65) -- (516.45,63.84) -- cycle ; \draw  [dash pattern={on 0.84pt off 2.51pt}] (504.62,54.62) .. controls (505.89,52.99) and (507.49,51.64) .. (509.31,50.65) ;  
%Shape: Arc [id:dp38821440591232825] 
\draw  [draw opacity=0] (513.16,49.2) .. controls (514.22,48.97) and (515.32,48.84) .. (516.45,48.84) .. controls (521.35,48.84) and (525.7,51.19) .. (528.43,54.82) -- (516.45,63.84) -- cycle ; \draw   (513.16,49.2) .. controls (514.22,48.97) and (515.32,48.84) .. (516.45,48.84) .. controls (521.35,48.84) and (525.7,51.19) .. (528.43,54.82) ;  
\draw   (505.15,76.48) -- (504.44,72.86) -- (508.03,73.7) ;
\draw   (533.51,67.59) -- (530.52,69.75) -- (529.78,66.14) ;
\draw   (527.86,51.6) -- (528.8,55.16) -- (525.16,54.55) ;
%Shape: Circle [id:dp5128746506297459] 
\draw  [dash pattern={on 0.84pt off 2.51pt}] (351.45,63.84) .. controls (351.45,55.56) and (358.16,48.84) .. (366.45,48.84) .. controls (374.73,48.84) and (381.45,55.56) .. (381.45,63.84) .. controls (381.45,72.13) and (374.73,78.84) .. (366.45,78.84) .. controls (358.16,78.84) and (351.45,72.13) .. (351.45,63.84) -- cycle ;

% Text Node
\draw (277,40) node [anchor=north west][inner sep=0.75pt]  [font=\LARGE] [align=left] {$ \mathrel{\substack{
\overset{\mu_P^+}{\longrightarrow}\\[-0.3ex]
\underset{\mu_{P}^-}{\longleftarrow}
}} $};
% Text Node
%\draw (112,55) node [anchor=north west][inner sep=0.75pt]   [align=left] {$\bar{h_1}$};
% Text Node
%\draw (6,58) node [anchor=north west][inner sep=0.75pt]   [align=left] {$\bar{h_2}$};
% Text Node
\draw (53,83) node [anchor=north west][inner sep=0.75pt]   [align=left] {$\vv_2$};
% Text Node
\draw (213,83) node [anchor=north west][inner sep=0.75pt]   [align=left] {$\vv_1$};
% Text Node
\draw (354.5,83) node [anchor=north west][inner sep=0.75pt]   [align=left] {$\vv_2$};
% Text Node
\draw (512.5,83) node [anchor=north west][inner sep=0.75pt]   [align=left] {$\vv_1$};
\end{scope}

\end{tikzpicture}

\caption{Kauer moves on biserial FBGs with clockwise cyclic orderings at all vertices.}
		\label{fig:1-KM-fms}	
	\end{center}
	\end{figure}
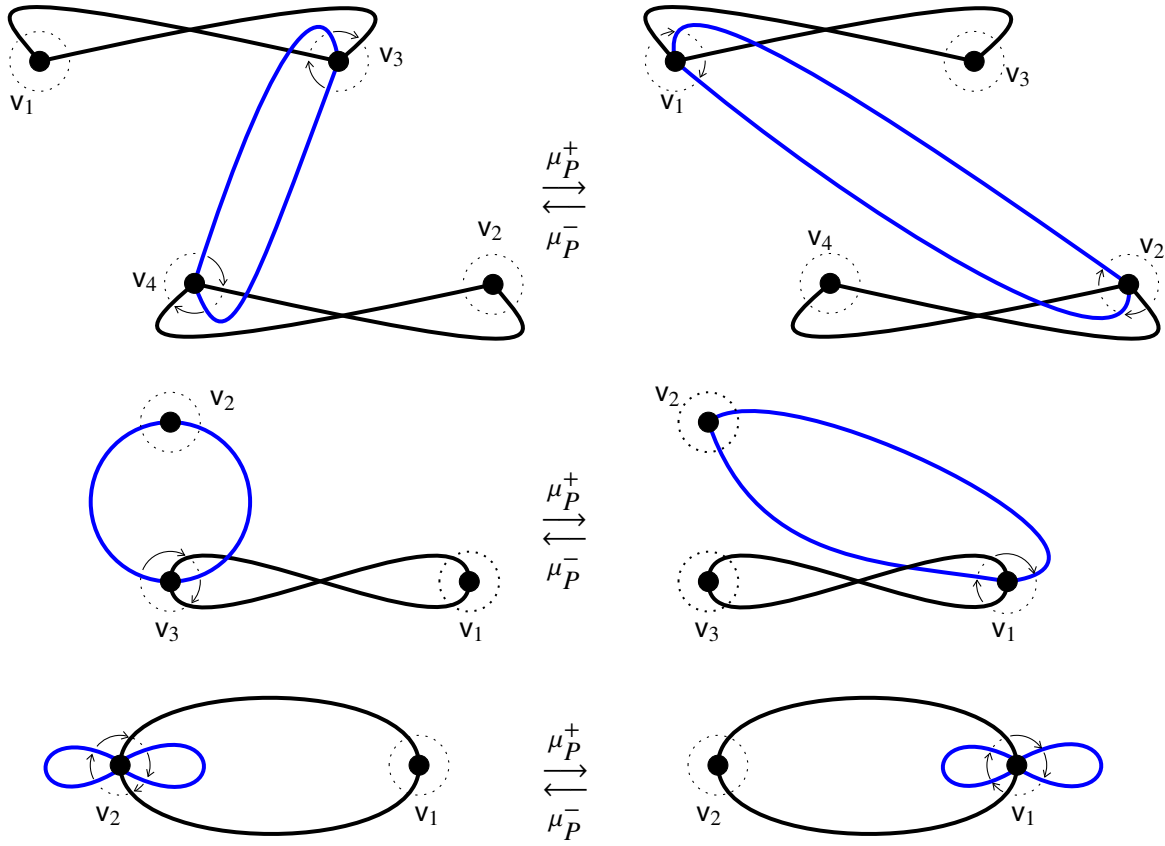 

The proof of these irreducible tilting mutations is analogous to that in the BGA case and consists of lengthy but straightforward verifications. The only extra difficulty lies in the dimension comparison, where additional cases must be treated, since one cannot in general traverse a full cycle around each vertex. For this reason, we defer these computations to the appendix.

We conclude above discussions with the following example.

\begin{example}
	Consider the ribbon graphs shown below, where the left (resp.\ right) graph is denoted by $\Gamma_1$ (resp.\ $\Gamma_2$). Each vertex is oriented clockwise.
	\begin{center}
\tikzset{every picture/.style={line width=0.75pt}} %set default line width to 0.75pt        

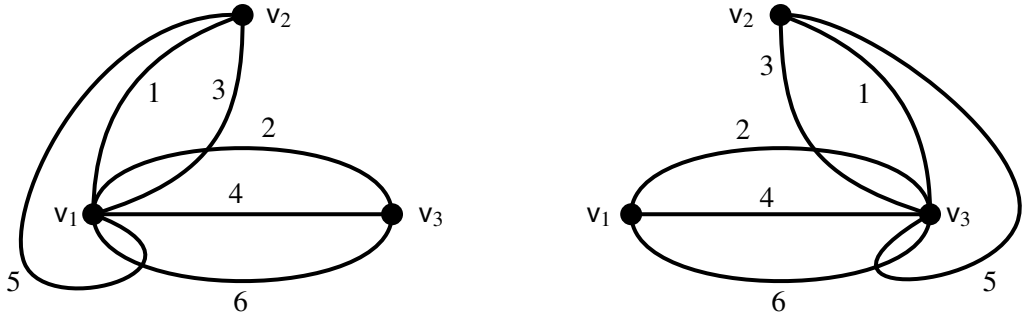
\begin{tikzpicture}[x=0.75pt,y=0.75pt,yscale=-1,xscale=1]
%uncomment if require: \path (0,235); %set diagram left start at 0, and has height of 235

%Shape: Circle [id:dp614547405669825] 
\draw  [fill={rgb, 255:red, 0; green, 0; blue, 0 }  ,fill opacity=1 ] (50,155) .. controls (50,152.24) and (52.24,150) .. (55,150) .. controls (57.76,150) and (60,152.24) .. (60,155) .. controls (60,157.76) and (57.76,160) .. (55,160) .. controls (52.24,160) and (50,157.76) .. (50,155) -- cycle ;
%Shape: Circle [id:dp6556135124424562] 
\draw  [fill={rgb, 255:red, 0; green, 0; blue, 0 }  ,fill opacity=1 ] (200,155) .. controls (200,152.24) and (202.24,150) .. (205,150) .. controls (207.76,150) and (210,152.24) .. (210,155) .. controls (210,157.76) and (207.76,160) .. (205,160) .. controls (202.24,160) and (200,157.76) .. (200,155) -- cycle ;
%Straight Lines [id:da1645085954759551] 
\draw [line width=1.5]    (55,155) -- (205,155) ;
%Curve Lines [id:da918065404640215] 
\draw [line width=1.5]    (55,155) .. controls (54.5,111) and (204.5,110.5) .. (205,155) ;
%Curve Lines [id:da4202388045517407] 
\draw [line width=1.5]    (55,155) .. controls (55.5,200) and (205,199.5) .. (205,155) ;
%Shape: Circle [id:dp21267779868081882] 
\draw  [fill={rgb, 255:red, 0; green, 0; blue, 0 }  ,fill opacity=1 ] (125,55) .. controls (125,52.24) and (127.24,50) .. (130,50) .. controls (132.76,50) and (135,52.24) .. (135,55) .. controls (135,57.76) and (132.76,60) .. (130,60) .. controls (127.24,60) and (125,57.76) .. (125,55) -- cycle ;
%Curve Lines [id:da4664347016009811] 
\draw [line width=1.5]    (55,155) .. controls (55.5,100) and (82.5,71.5) .. (130,55) ;
%Curve Lines [id:da19263303202223914] 
\draw [line width=1.5]    (55,155) .. controls (115.5,137) and (130.5,109) .. (130,55) ;
%Curve Lines [id:da9546587307922114] 
\draw [line width=1.5]    (55,155) .. controls (125.5,181.5) and (32.5,209) .. (20.5,179.5) .. controls (8.5,150) and (57.5,51) .. (130,55) ;
%Shape: Circle [id:dp7952016335846566] 
\draw  [fill={rgb, 255:red, 0; green, 0; blue, 0 }  ,fill opacity=1 ] (320,155) .. controls (320,152.24) and (322.24,150) .. (325,150) .. controls (327.76,150) and (330,152.24) .. (330,155) .. controls (330,157.76) and (327.76,160) .. (325,160) .. controls (322.24,160) and (320,157.76) .. (320,155) -- cycle ;
%Shape: Circle [id:dp29473079999107377] 
\draw  [fill={rgb, 255:red, 0; green, 0; blue, 0 }  ,fill opacity=1 ] (470,155) .. controls (470,152.24) and (472.24,150) .. (475,150) .. controls (477.76,150) and (480,152.24) .. (480,155) .. controls (480,157.76) and (477.76,160) .. (475,160) .. controls (472.24,160) and (470,157.76) .. (470,155) -- cycle ;
%Straight Lines [id:da9898298371191689] 
\draw [line width=1.5]    (325,155) -- (475,155) ;
%Curve Lines [id:da7073166384451035] 
\draw [line width=1.5]    (325,155) .. controls (324.5,111) and (474.5,110.5) .. (475,155) ;
%Curve Lines [id:da41939892022189906] 
\draw [line width=1.5]    (325,155) .. controls (325.5,200) and (475,199.5) .. (475,155) ;
%Shape: Circle [id:dp6821210109459024] 
\draw  [fill={rgb, 255:red, 0; green, 0; blue, 0 }  ,fill opacity=1 ] (395,55) .. controls (395,52.24) and (397.24,50) .. (400,50) .. controls (402.76,50) and (405,52.24) .. (405,55) .. controls (405,57.76) and (402.76,60) .. (400,60) .. controls (397.24,60) and (395,57.76) .. (395,55) -- cycle ;
%Curve Lines [id:da7257423097205837] 
\draw [line width=1.5]    (475,155) .. controls (475.5,100) and (444.5,70.75) .. (400,55) ;
%Curve Lines [id:da3023850145915912] 
\draw [line width=1.5]    (475,155) .. controls (411,135.75) and (400.5,109) .. (400,55) ;
%Curve Lines [id:da6604020987627623] 
\draw [line width=1.5]    (400,55) .. controls (440.78,52.67) and (533,123.25) .. (518.5,159.75) .. controls (504,196.25) and (399.5,199.25) .. (475,155) ;

% Text Node
\draw (80.5,87.25) node [anchor=north west][inner sep=0.75pt]   [align=left] {$1$};
% Text Node
\draw (138,105) node [anchor=north west][inner sep=0.75pt]   [align=left] {$2$};
% Text Node
\draw (113,85.75) node [anchor=north west][inner sep=0.75pt]   [align=left] {$3$};
% Text Node
\draw (121.5,138.25) node [anchor=north west][inner sep=0.75pt]   [align=left] {$4$};
% Text Node
\draw (10,182.25) node [anchor=north west][inner sep=0.75pt]   [align=left] {$5$};
% Text Node
\draw (124,192.25) node [anchor=north west][inner sep=0.75pt]   [align=left] {$6$};
% Text Node
\draw (34,150) node [anchor=north west][inner sep=0.75pt]   [align=left] {$\vv_1$};
% Text Node
\draw (140.5,50) node [anchor=north west][inner sep=0.75pt]   [align=left] {$\vv_2$};
% Text Node
\draw (217,150) node [anchor=north west][inner sep=0.75pt]   [align=left] {$\vv_3$};
% Text Node
\draw (436.75,87.75) node [anchor=north west][inner sep=0.75pt]   [align=left] {$1$};
% Text Node
\draw (376,105.75) node [anchor=north west][inner sep=0.75pt]   [align=left] {$2$};
% Text Node
\draw (388,74) node [anchor=north west][inner sep=0.75pt]   [align=left] {$3$};
% Text Node
\draw (388,140.25) node [anchor=north west][inner sep=0.75pt]   [align=left] {$4$};
% Text Node
\draw (500,182.25) node [anchor=north west][inner sep=0.75pt]   [align=left] {$5$};
% Text Node
\draw (394,192.25) node [anchor=north west][inner sep=0.75pt]   [align=left] {$6$};
% Text Node
\draw (301.25,150) node [anchor=north west][inner sep=0.75pt]   [align=left] {$\vv_1$};
% Text Node
\draw (372.75,50) node [anchor=north west][inner sep=0.75pt]   [align=left] {$\vv_2$};
% Text Node
\draw (481.75,150) node [anchor=north west][inner sep=0.75pt]   [align=left] {$\vv_3$};

\end{tikzpicture}
	\end{center}
Define the multiplicity on each vertex by 
\[
m(\vv_1)=\tfrac{2}{3}, \quad m(\vv_2)=m(\vv_3)=\tfrac{1}{3}.
\]
This induces degree functions on $\Gamma_1$ and $\Gamma_2$, making them two biserial FBGAs.

	Moreover, the biserial FBGA $A_1 = kQ_1/I_1$ associated with $\Gamma_1$ is given by the quiver $Q_1 = Q_{\Gamma_1}$:
		$$
		
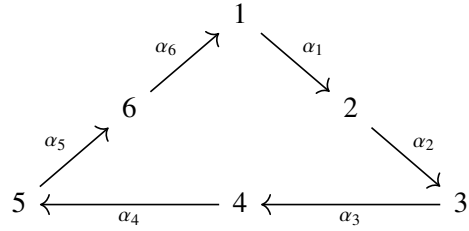
\begin{tikzcd}
			&                         & 1 \arrow[rd, "\alpha_1"] &                         &                         \\
			& 6 \arrow[ru, "\alpha_6"] &                         & 2 \arrow[rd, "\alpha_{2}"] &                         \\
			5 \arrow[ru, "\alpha_5"] &                         & 4 \arrow[ll, "\alpha_4"] &                         & 3 \arrow[ll, "\alpha_3"]
		\end{tikzcd}$$
		Then $I_1$ is generated by
		\begin{itemize}
			\item $\alpha_1\alpha_2\alpha_3\alpha_4\alpha_5$, 
$\alpha_2\alpha_3\alpha_4\alpha_5\alpha_6$, 
$\alpha_3\alpha_4\alpha_5\alpha_6\alpha_1$,  
$\alpha_4\alpha_5\alpha_6\alpha_1\alpha_2$, 
$\alpha_5\alpha_6\alpha_1\alpha_2\alpha_3$, 
$\alpha_6\alpha_1\alpha_2\alpha_3\alpha_4$.
		\end{itemize}
		Thus $A_1$ is a self-injective Nakayama algebra.

		The biserial FBGA $A_2=kQ_2/I_2$ associated with $\Gamma_2$ is given by the quiver $Q_2 = Q_{\Gamma_2}$:
		$$
	
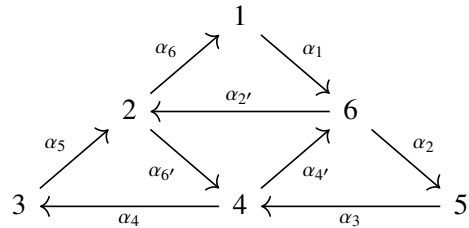
\begin{tikzcd}
		&                                              & 1 \arrow[rd, "\alpha_{1}"]                      &                                              &                         \\
		& 2 \arrow[ru, "\alpha_6"] \arrow[rd, "\alpha_{6'}"'] &                                              & 6 \arrow[rd, "\alpha_{2}"] \arrow[ll, "\alpha_{2'}"'] &                         \\
		3 \arrow[ru, "\alpha_5"] &                                              & 4 \arrow[ll, "\alpha_4"] \arrow[ru, "\alpha_{4'}"'] &                                              & 5 \arrow[ll, "\alpha_3"]
	\end{tikzcd}$$
		Then $I_2$ is generated by
		\begin{itemize}
			\item $\alpha_6\alpha_1-\alpha_{6'}\alpha_{4'}$, 
      $\alpha_2\alpha_{3}-\alpha_{2'}\alpha_{6'}$, 
      $\alpha_4\alpha_{5}-\alpha_{4'}\alpha_{2'}$.

\item $\alpha_{1}\alpha_{2}\alpha_{3}$, 
      $\alpha_{2}\alpha_{3}\alpha_{4}$, 
      $\alpha_{3}\alpha_{4}\alpha_{5}$,  
      $\alpha_{4}\alpha_{5}\alpha_{6}$, 
      $\alpha_{5}\alpha_{6}\alpha_{1}$, 
      $\alpha_{6}\alpha_{1}\alpha_{2}$, 
      $\alpha_{4'}\alpha_{2'}\alpha_{6'}$, 
      $\alpha_{6'}\alpha_{4'}\alpha_{2'}$, 
      $\alpha_{2'}\alpha_{6'}\alpha_{4'}$.

\item $\alpha_{5}\alpha_{6'}$, 
      $\alpha_{6'}\alpha_{4}$, 
      $\alpha_{3}\alpha_{4'}$, 
      $\alpha_{4'}\alpha_{2}$, 
      $\alpha_{1}\alpha_{2'}$, 
      $\alpha_{2'}\alpha_{6}$.
		\end{itemize}
		
		These two algebras are derived equivalent, since they are related by a Kauer move on the $\nu_A$-orbit corresponding to the edges $\{1,3,5\}$. Such a derived equivalence is also demonstrated in \cite[Example~3.5]{Miz}.
\end{example}

\subsection{Realization of Kauer moves on $\Gamma / \langle \nu \rangle$}\label{subsec:move-on-red}
\

By the definition of the graph of orbits (see Subsection \ref{subsec:reduced-form}), each (orbifold) edge in 
$\Gamma/\langle \nu\rangle$ corresponds to a $\nu$-orbit of edges in the original graph $\Gamma$. It also naturally corresponds to a $\nu_A$-orbit of indecomposable projective modules of the associated biserial FBGA $A$, where $\nu_A$ is the Nakayama automorphism of $A$. Therefore, the $\nu$-stable mutations considered above naturally correspond to moves along each (orbifold) edge of $\Gamma/\langle \nu\rangle$. More precisely,
\begin{enumerate}
\item if the move does not involve orbifold edges, then it is the same as in the BGA case, as illustrated in Figure~\ref{fig:Kauer-moves}.

\item If the move involves orbifold edges, it likewise corresponds to sliding an (orbifold) edge along the orientation, as illustrated in Figure~\ref{fig:Kauer-moves-orbifold}. This is because each orbifold edge corresponds, in effect, to two glued half-edges; equivalently, a vertex $v$ incident with an orbifold edge lifts in the double cover $\widehat{\Gamma/\langle \nu \rangle}$ to two vertices $v'$ and $v''$. Thus, sliding from $v$ along the orbifold edge passes through the orbifold and returns to $v$, which in $\widehat{\Gamma/\langle \nu \rangle}$ corresponds to sliding from $v'$ to $v''$.
\end{enumerate}

\begin{figure}[ht]
		\centering
		\begin{center}

\tikzset{every picture/.style={line width=0.75pt}} %set default line width to 0.75pt        

\begin{tikzpicture}[x=0.75pt,y=0.75pt,yscale=-1,xscale=1]
%uncomment if require: \path (0,235); %set diagram left start at 0, and has height of 235

%Straight Lines [id:da8545627489039611] 
\draw [line width=1.5]    (100,111) -- (200,111) ;
%Shape: Circle [id:dp26055871143917075] 
\draw  [fill={rgb, 255:red, 0; green, 0; blue, 0 }  ,fill opacity=1 ] (195,111) .. controls (195,108.24) and (197.24,106) .. (200,106) .. controls (202.76,106) and (205,108.24) .. (205,111) .. controls (205,113.76) and (202.76,116) .. (200,116) .. controls (197.24,116) and (195,113.76) .. (195,111) -- cycle ;
%Straight Lines [id:da10749896890831767] 
\draw [color={rgb, 255:red, 0; green, 0; blue, 255 }  ,draw opacity=1 ][line width=1.5]    (100,111) -- (100,41) ;
%Shape: Circle [id:dp684987033191639] 
\draw  [fill={rgb, 255:red, 0; green, 0; blue, 0 }  ,fill opacity=1 ] (95,111) .. controls (95,108.24) and (97.24,106) .. (100,106) .. controls (102.76,106) and (105,108.24) .. (105,111) .. controls (105,113.76) and (102.76,116) .. (100,116) .. controls (97.24,116) and (95,113.76) .. (95,111) -- cycle ;
%Straight Lines [id:da82873827900161] 
\draw [line width=1.5]    (350,111) -- (450,111) ;
%Shape: Circle [id:dp2544617409501262] 
\draw  [fill={rgb, 255:red, 0; green, 0; blue, 0 }  ,fill opacity=1 ] (345,111) .. controls (345,108.24) and (347.24,106) .. (350,106) .. controls (352.76,106) and (355,108.24) .. (355,111) .. controls (355,113.76) and (352.76,116) .. (350,116) .. controls (347.24,116) and (345,113.76) .. (345,111) -- cycle ;
%Shape: Arc [id:dp39725639423197623] 
\draw  [draw opacity=0][dash pattern={on 0.84pt off 2.51pt}] (114.34,115.41) .. controls (112.46,121.54) and (106.75,126) .. (100,126) .. controls (91.72,126) and (85,119.28) .. (85,111) .. controls (85,103.05) and (91.18,96.55) .. (99,96.03) -- (100,111) -- cycle ; \draw  [dash pattern={on 0.84pt off 2.51pt}] (114.34,115.41) .. controls (112.46,121.54) and (106.75,126) .. (100,126) .. controls (91.72,126) and (85,119.28) .. (85,111) .. controls (85,103.05) and (91.18,96.55) .. (99,96.03) ;  
%Shape: Arc [id:dp4638717400522323] 
\draw  [draw opacity=0] (103.21,96.34) .. controls (108.82,97.57) and (113.26,101.93) .. (114.59,107.5) -- (100,111) -- cycle ; \draw   (103.21,96.34) .. controls (108.82,97.57) and (113.26,101.93) .. (114.59,107.5) ;  
\draw   (115.88,104.38) -- (114.76,107.9) -- (112.01,105.43) ;
%Shape: Circle [id:dp6967319191602261] 
\draw  [dash pattern={on 0.84pt off 2.51pt}] (185,111) .. controls (185,102.72) and (191.72,96) .. (200,96) .. controls (208.28,96) and (215,102.72) .. (215,111) .. controls (215,119.28) and (208.28,126) .. (200,126) .. controls (191.72,126) and (185,119.28) .. (185,111) -- cycle ;
%Shape: Circle [id:dp1801702762332369] 
\draw  [dash pattern={on 0.84pt off 2.51pt}] (335,111) .. controls (335,102.72) and (341.72,96) .. (350,96) .. controls (358.28,96) and (365,102.72) .. (365,111) .. controls (365,119.28) and (358.28,126) .. (350,126) .. controls (341.72,126) and (335,119.28) .. (335,111) -- cycle ;
%Curve Lines [id:da9532218780482482] 
\draw [color={rgb, 255:red, 0; green, 0; blue, 255 }  ,draw opacity=1 ][line width=1.5]    (349.6,41) .. controls (350,89) and (407.2,76.6) .. (450,111) ;
%Shape: Circle [id:dp7148480296477373] 
\draw  [fill={rgb, 255:red, 0; green, 0; blue, 0 }  ,fill opacity=1 ] (445,111) .. controls (445,108.24) and (447.24,106) .. (450,106) .. controls (452.76,106) and (455,108.24) .. (455,111) .. controls (455,113.76) and (452.76,116) .. (450,116) .. controls (447.24,116) and (445,113.76) .. (445,111) -- cycle ;
%Shape: Arc [id:dp41486476999926913] 
\draw  [draw opacity=0][dash pattern={on 0.84pt off 2.51pt}] (438.53,101.33) .. controls (441.28,98.07) and (445.4,96) .. (450,96) .. controls (458.28,96) and (465,102.72) .. (465,111) .. controls (465,119.28) and (458.28,126) .. (450,126) .. controls (441.72,126) and (435,119.28) .. (435,111) .. controls (435,111) and (435,111) .. (435,111) -- (450,111) -- cycle ; \draw  [dash pattern={on 0.84pt off 2.51pt}] (438.53,101.33) .. controls (441.28,98.07) and (445.4,96) .. (450,96) .. controls (458.28,96) and (465,102.72) .. (465,111) .. controls (465,119.28) and (458.28,126) .. (450,126) .. controls (441.72,126) and (435,119.28) .. (435,111) .. controls (435,111) and (435,111) .. (435,111) ;  
%Shape: Arc [id:dp9251235147556214] 
\draw  [draw opacity=0] (435.18,108.69) .. controls (435.4,107.23) and (435.84,105.83) .. (436.46,104.54) -- (450,111) -- cycle ; \draw   (435.18,108.69) .. controls (435.4,107.23) and (435.84,105.83) .. (436.46,104.54) ;  
\draw   (433.86,105.65) -- (437.07,103.83) -- (437.41,107.5) ;

% Text Node
\draw (92,33) node [anchor=north west][inner sep=0.75pt] [font=\Large]  [align=left] {$\times$};
% Text Node
\draw (342,33) node [anchor=north west][inner sep=0.75pt]  [font=\Large] [align=left] {$\times$};
% Text Node
\draw (244,58) node [anchor=north west][inner sep=0.75pt]  [font=\LARGE] [align=left] {$ \mathrel{\substack{
\overset{\mu_P^+}{\longrightarrow}\\[-0.3ex]
\underset{\mu_{P}^-}{\longleftarrow}
}} $};

\end{tikzpicture}

			\medskip

\tikzset{every picture/.style={line width=0.75pt}} %set default line width to 0.75pt        

\begin{tikzpicture}[x=0.75pt,y=0.75pt,yscale=-1,xscale=1]
%uncomment if require: \path (0,235); %set diagram left start at 0, and has height of 235

%Straight Lines [id:da8545627489039611] 
\draw [color={rgb, 255:red, 0; green, 0; blue, 255 }  ,draw opacity=1 ][line width=1.5]    (100,111) -- (200,111) ;
%Shape: Circle [id:dp26055871143917075] 
\draw  [fill={rgb, 255:red, 0; green, 0; blue, 0 }  ,fill opacity=1 ] (195,111) .. controls (195,108.24) and (197.24,106) .. (200,106) .. controls (202.76,106) and (205,108.24) .. (205,111) .. controls (205,113.76) and (202.76,116) .. (200,116) .. controls (197.24,116) and (195,113.76) .. (195,111) -- cycle ;
%Straight Lines [id:da10749896890831767] 
\draw [color={rgb, 255:red, 0; green, 0; blue, 0 }  ,draw opacity=1 ][line width=1.5]    (200,111) -- (200,41) ;
%Shape: Circle [id:dp684987033191639] 
\draw  [fill={rgb, 255:red, 0; green, 0; blue, 0 }  ,fill opacity=1 ] (95,111) .. controls (95,108.24) and (97.24,106) .. (100,106) .. controls (102.76,106) and (105,108.24) .. (105,111) .. controls (105,113.76) and (102.76,116) .. (100,116) .. controls (97.24,116) and (95,113.76) .. (95,111) -- cycle ;
%Shape: Circle [id:dp2544617409501262] 
\draw  [fill={rgb, 255:red, 0; green, 0; blue, 0 }  ,fill opacity=1 ] (345,111) .. controls (345,108.24) and (347.24,106) .. (350,106) .. controls (352.76,106) and (355,108.24) .. (355,111) .. controls (355,113.76) and (352.76,116) .. (350,116) .. controls (347.24,116) and (345,113.76) .. (345,111) -- cycle ;
%Shape: Arc [id:dp39725639423197623] 
\draw  [draw opacity=0][dash pattern={on 0.84pt off 2.51pt}] (99.4,125.99) .. controls (91.39,125.67) and (85,119.08) .. (85,111) .. controls (85,102.72) and (91.72,96) .. (100,96) .. controls (107.8,96) and (114.2,101.95) .. (114.93,109.55) -- (100,111) -- cycle ; \draw  [dash pattern={on 0.84pt off 2.51pt}] (99.4,125.99) .. controls (91.39,125.67) and (85,119.08) .. (85,111) .. controls (85,102.72) and (91.72,96) .. (100,96) .. controls (107.8,96) and (114.2,101.95) .. (114.93,109.55) ;  
%Shape: Arc [id:dp4638717400522323] 
\draw  [draw opacity=0] (114.92,112.52) .. controls (114.22,119.5) and (108.74,125.06) .. (101.8,125.89) -- (100,111) -- cycle ; \draw   (114.92,112.52) .. controls (114.22,119.5) and (108.74,125.06) .. (101.8,125.89) ;  
\draw   (105.77,126.81) -- (102.13,126.22) -- (104.16,123.14) ;
%Shape: Circle [id:dp1801702762332369] 
\draw  [dash pattern={on 0.84pt off 2.51pt}] (335,111) .. controls (335,102.72) and (341.72,96) .. (350,96) .. controls (358.28,96) and (365,102.72) .. (365,111) .. controls (365,119.28) and (358.28,126) .. (350,126) .. controls (341.72,126) and (335,119.28) .. (335,111) -- cycle ;
%Shape: Circle [id:dp7148480296477373] 
\draw  [fill={rgb, 255:red, 0; green, 0; blue, 0 }  ,fill opacity=1 ] (445,111) .. controls (445,108.24) and (447.24,106) .. (450,106) .. controls (452.76,106) and (455,108.24) .. (455,111) .. controls (455,113.76) and (452.76,116) .. (450,116) .. controls (447.24,116) and (445,113.76) .. (445,111) -- cycle ;
%Shape: Arc [id:dp41486476999926913] 
\draw  [draw opacity=0][dash pattern={on 0.84pt off 2.51pt}] (459.66,99.52) .. controls (462.92,102.27) and (465,106.39) .. (465,111) .. controls (465,119.28) and (458.28,126) .. (450,126) .. controls (441.72,126) and (435,119.28) .. (435,111) .. controls (435,103.67) and (440.26,97.57) .. (447.21,96.26) -- (450,111) -- cycle ; \draw  [dash pattern={on 0.84pt off 2.51pt}] (459.66,99.52) .. controls (462.92,102.27) and (465,106.39) .. (465,111) .. controls (465,119.28) and (458.28,126) .. (450,126) .. controls (441.72,126) and (435,119.28) .. (435,111) .. controls (435,103.67) and (440.26,97.57) .. (447.21,96.26) ;  
%Shape: Arc [id:dp9251235147556214] 
\draw  [draw opacity=0] (451.8,96.11) .. controls (454.05,96.38) and (456.14,97.14) .. (457.97,98.29) -- (450,111) -- cycle ; \draw   (451.8,96.11) .. controls (454.05,96.38) and (456.14,97.14) .. (457.97,98.29) ;  
\draw   (456.56,95.48) -- (458.04,98.85) -- (454.36,98.82) ;
%Straight Lines [id:da47263035339706927] 
\draw [color={rgb, 255:red, 0; green, 0; blue, 0 }  ,draw opacity=1 ][line width=1.5]    (450,111) -- (450,41) ;
%Straight Lines [id:da8704012366727096] 
\draw [color={rgb, 255:red, 0; green, 0; blue, 0 }  ,draw opacity=1 ][line width=1.5]    (100,181) -- (100,111) ;
%Straight Lines [id:da7107760640710492] 
\draw [color={rgb, 255:red, 0; green, 0; blue, 0 }  ,draw opacity=1 ][line width=1.5]    (350,181) -- (350,111) ;
%Shape: Circle [id:dp2257285636334383] 
\draw  [dash pattern={on 0.84pt off 2.51pt}] (85,181) .. controls (85,172.72) and (91.72,166) .. (100,166) .. controls (108.28,166) and (115,172.72) .. (115,181) .. controls (115,189.28) and (108.28,196) .. (100,196) .. controls (91.72,196) and (85,189.28) .. (85,181) -- cycle ;
%Shape: Circle [id:dp7168381680231402] 
\draw  [fill={rgb, 255:red, 0; green, 0; blue, 0 }  ,fill opacity=1 ] (95,181) .. controls (95,178.24) and (97.24,176) .. (100,176) .. controls (102.76,176) and (105,178.24) .. (105,181) .. controls (105,183.76) and (102.76,186) .. (100,186) .. controls (97.24,186) and (95,183.76) .. (95,181) -- cycle ;
%Shape: Arc [id:dp8624357849233557] 
\draw  [draw opacity=0][dash pattern={on 0.84pt off 2.51pt}] (200.69,96.02) .. controls (208.65,96.38) and (215,102.95) .. (215,111) .. controls (215,119.28) and (208.28,126) .. (200,126) .. controls (191.96,126) and (185.4,119.68) .. (185.02,111.74) -- (200,111) -- cycle ; \draw  [dash pattern={on 0.84pt off 2.51pt}] (200.69,96.02) .. controls (208.65,96.38) and (215,102.95) .. (215,111) .. controls (215,119.28) and (208.28,126) .. (200,126) .. controls (191.96,126) and (185.4,119.68) .. (185.02,111.74) ;  
%Shape: Arc [id:dp6280893643148724] 
\draw  [draw opacity=0] (185.05,109.76) .. controls (185.63,102.67) and (191.14,96.98) .. (198.14,96.11) -- (200,111) -- cycle ; \draw   (185.05,109.76) .. controls (185.63,102.67) and (191.14,96.98) .. (198.14,96.11) ;  
\draw   (194.77,94.63) -- (198.25,95.83) -- (195.72,98.51) ;
%Curve Lines [id:da627070108258077] 
\draw [color={rgb, 255:red, 0; green, 0; blue, 255 }  ,draw opacity=1 ][line width=1.5]    (350,181) .. controls (388.2,123.8) and (400.2,24.2) .. (443.4,24.6) .. controls (486.6,25) and (486.38,66.77) .. (450,111) ;
%Shape: Arc [id:dp7621180013656907] 
\draw  [draw opacity=0][dash pattern={on 0.84pt off 2.51pt}] (358.74,168.81) .. controls (362.53,171.53) and (365,175.98) .. (365,181) .. controls (365,189.28) and (358.28,196) .. (350,196) .. controls (341.72,196) and (335,189.28) .. (335,181) .. controls (335,173.67) and (340.26,167.57) .. (347.21,166.26) -- (350,181) -- cycle ; \draw  [dash pattern={on 0.84pt off 2.51pt}] (358.74,168.81) .. controls (362.53,171.53) and (365,175.98) .. (365,181) .. controls (365,189.28) and (358.28,196) .. (350,196) .. controls (341.72,196) and (335,189.28) .. (335,181) .. controls (335,173.67) and (340.26,167.57) .. (347.21,166.26) ;  
%Shape: Circle [id:dp7748691013133424] 
\draw  [fill={rgb, 255:red, 0; green, 0; blue, 0 }  ,fill opacity=1 ] (345,181) .. controls (345,178.24) and (347.24,176) .. (350,176) .. controls (352.76,176) and (355,178.24) .. (355,181) .. controls (355,183.76) and (352.76,186) .. (350,186) .. controls (347.24,186) and (345,183.76) .. (345,181) -- cycle ;
%Shape: Arc [id:dp4517203629233135] 
\draw  [draw opacity=0] (349.46,166.01) .. controls (349.64,166) and (349.82,166) .. (350,166) .. controls (352.2,166) and (354.3,166.48) .. (356.18,167.33) -- (350,181) -- cycle ; \draw   (349.46,166.01) .. controls (349.64,166) and (349.82,166) .. (350,166) .. controls (352.2,166) and (354.3,166.48) .. (356.18,167.33) ;  
\draw   (354.15,164.4) -- (356.41,167.32) -- (352.82,168.17) ;

% Text Node
\draw (192,34) node [anchor=north west][inner sep=0.75pt]  [font=\Large] [align=left] {$\times$};
% Text Node
\draw (442,34) node [anchor=north west][inner sep=0.75pt]  [font=\Large] [align=left] {$\times$};
% Text Node
\draw (254,78) node [anchor=north west][inner sep=0.75pt]  [font=\LARGE] [align=left] {$ \mathrel{\substack{
\overset{\mu_P^+}{\longrightarrow}\\[-0.3ex]
\underset{\mu_{P}^-}{\longleftarrow}
}} $};

\end{tikzpicture}

		\medskip

\tikzset{every picture/.style={line width=0.75pt}} %set default line width to 0.75pt        

\tikzset{every picture/.style={line width=0.75pt}} %set default line width to 0.75pt        

\begin{tikzpicture}[x=0.75pt,y=0.75pt,yscale=-1,xscale=1]
%uncomment if require: \path (0,235); %set diagram left start at 0, and has height of 235

%Straight Lines [id:da8545627489039611] 
\draw [color={rgb, 255:red, 0; green, 0; blue, 255 }  ,draw opacity=1 ][line width=1.5]    (100,111) -- (200,111) ;
%Shape: Circle [id:dp26055871143917075] 
\draw  [fill={rgb, 255:red, 0; green, 0; blue, 0 }  ,fill opacity=1 ] (195,111) .. controls (195,108.24) and (197.24,106) .. (200,106) .. controls (202.76,106) and (205,108.24) .. (205,111) .. controls (205,113.76) and (202.76,116) .. (200,116) .. controls (197.24,116) and (195,113.76) .. (195,111) -- cycle ;
%Straight Lines [id:da10749896890831767] 
\draw [color={rgb, 255:red, 0; green, 0; blue, 0 }  ,draw opacity=1 ][line width=1.5]    (200,111) -- (200,41) ;
%Shape: Circle [id:dp684987033191639] 
\draw  [fill={rgb, 255:red, 0; green, 0; blue, 0 }  ,fill opacity=1 ] (95,111) .. controls (95,108.24) and (97.24,106) .. (100,106) .. controls (102.76,106) and (105,108.24) .. (105,111) .. controls (105,113.76) and (102.76,116) .. (100,116) .. controls (97.24,116) and (95,113.76) .. (95,111) -- cycle ;
%Shape: Circle [id:dp2544617409501262] 
\draw  [fill={rgb, 255:red, 0; green, 0; blue, 0 }  ,fill opacity=1 ] (345,111) .. controls (345,108.24) and (347.24,106) .. (350,106) .. controls (352.76,106) and (355,108.24) .. (355,111) .. controls (355,113.76) and (352.76,116) .. (350,116) .. controls (347.24,116) and (345,113.76) .. (345,111) -- cycle ;
%Shape: Arc [id:dp4638717400522323] 
\draw  [draw opacity=0] (114.92,112.52) .. controls (114.16,120.09) and (107.77,126) .. (100,126) .. controls (91.72,126) and (85,119.28) .. (85,111) .. controls (85,102.72) and (91.72,96) .. (100,96) .. controls (107.31,96) and (113.39,101.22) .. (114.73,108.14) -- (100,111) -- cycle ; \draw   (114.92,112.52) .. controls (114.16,120.09) and (107.77,126) .. (100,126) .. controls (91.72,126) and (85,119.28) .. (85,111) .. controls (85,102.72) and (91.72,96) .. (100,96) .. controls (107.31,96) and (113.39,101.22) .. (114.73,108.14) ;  
\draw   (115.82,105.14) -- (114.8,108.68) -- (111.99,106.29) ;
%Shape: Circle [id:dp7148480296477373] 
\draw  [fill={rgb, 255:red, 0; green, 0; blue, 0 }  ,fill opacity=1 ] (445,111) .. controls (445,108.24) and (447.24,106) .. (450,106) .. controls (452.76,106) and (455,108.24) .. (455,111) .. controls (455,113.76) and (452.76,116) .. (450,116) .. controls (447.24,116) and (445,113.76) .. (445,111) -- cycle ;
%Shape: Arc [id:dp41486476999926913] 
\draw  [draw opacity=0][dash pattern={on 0.84pt off 2.51pt}] (459.66,99.52) .. controls (462.92,102.27) and (465,106.39) .. (465,111) .. controls (465,119.28) and (458.28,126) .. (450,126) .. controls (441.72,126) and (435,119.28) .. (435,111) .. controls (435,103.67) and (440.26,97.57) .. (447.21,96.26) -- (450,111) -- cycle ; \draw  [dash pattern={on 0.84pt off 2.51pt}] (459.66,99.52) .. controls (462.92,102.27) and (465,106.39) .. (465,111) .. controls (465,119.28) and (458.28,126) .. (450,126) .. controls (441.72,126) and (435,119.28) .. (435,111) .. controls (435,103.67) and (440.26,97.57) .. (447.21,96.26) ;  
%Shape: Arc [id:dp9251235147556214] 
\draw  [draw opacity=0] (451.8,96.11) .. controls (454.05,96.38) and (456.14,97.14) .. (457.97,98.29) -- (450,111) -- cycle ; \draw   (451.8,96.11) .. controls (454.05,96.38) and (456.14,97.14) .. (457.97,98.29) ;  
\draw   (456.56,95.48) -- (458.04,98.85) -- (454.36,98.82) ;
%Straight Lines [id:da47263035339706927] 
\draw [color={rgb, 255:red, 0; green, 0; blue, 0 }  ,draw opacity=1 ][line width=1.5]    (450,111) -- (450,41) ;
%Shape: Arc [id:dp8624357849233557] 
\draw  [draw opacity=0][dash pattern={on 0.84pt off 2.51pt}] (200.69,96.02) .. controls (208.65,96.38) and (215,102.95) .. (215,111) .. controls (215,119.28) and (208.28,126) .. (200,126) .. controls (191.96,126) and (185.4,119.68) .. (185.02,111.74) -- (200,111) -- cycle ; \draw  [dash pattern={on 0.84pt off 2.51pt}] (200.69,96.02) .. controls (208.65,96.38) and (215,102.95) .. (215,111) .. controls (215,119.28) and (208.28,126) .. (200,126) .. controls (191.96,126) and (185.4,119.68) .. (185.02,111.74) ;  
%Shape: Arc [id:dp6280893643148724] 
\draw  [draw opacity=0] (185.05,109.76) .. controls (185.63,102.67) and (191.14,96.98) .. (198.14,96.11) -- (200,111) -- cycle ; \draw   (185.05,109.76) .. controls (185.63,102.67) and (191.14,96.98) .. (198.14,96.11) ;  
\draw   (194.77,94.63) -- (198.25,95.83) -- (195.72,98.51) ;
%Curve Lines [id:da627070108258077] 
\draw [color={rgb, 255:red, 0; green, 0; blue, 255 }  ,draw opacity=1 ][line width=1.5]    (350,111) .. controls (388.2,53.8) and (400.2,24.2) .. (443.4,24.6) .. controls (486.6,25) and (486.38,66.77) .. (450,111) ;
%Shape: Arc [id:dp8680202812886213] 
\draw  [draw opacity=0] (359.37,99.29) .. controls (362.8,102.04) and (365,106.26) .. (365,111) .. controls (365,119.28) and (358.28,126) .. (350,126) .. controls (341.72,126) and (335,119.28) .. (335,111) .. controls (335,102.72) and (341.72,96) .. (350,96) .. controls (352.15,96) and (354.19,96.45) .. (356.04,97.27) -- (350,111) -- cycle ; \draw   (359.37,99.29) .. controls (362.8,102.04) and (365,106.26) .. (365,111) .. controls (365,119.28) and (358.28,126) .. (350,126) .. controls (341.72,126) and (335,119.28) .. (335,111) .. controls (335,102.72) and (341.72,96) .. (350,96) .. controls (352.15,96) and (354.19,96.45) .. (356.04,97.27) ;  
\draw   (355.16,94.68) -- (356.64,98.05) -- (352.95,98.02) ;

% Text Node
\draw (192,34) node [anchor=north west][inner sep=0.75pt]  [font=\Large] [align=left] {$\times$};
% Text Node
\draw (442,34) node [anchor=north west][inner sep=0.75pt]  [font=\Large] [align=left] {$\times$};
% Text Node
\draw (254,58) node [anchor=north west][inner sep=0.75pt]  [font=\LARGE] [align=left] {$ \mathrel{\substack{
\overset{\mu_P^+}{\longrightarrow}\\[-0.3ex]
\underset{\mu_{P}^-}{\longleftarrow}
}} $};

\end{tikzpicture}

			\medskip

\tikzset{every picture/.style={line width=0.75pt}} %set default line width to 0.75pt        

\begin{tikzpicture}[x=0.75pt,y=0.75pt,yscale=-1,xscale=1]
%uncomment if require: \path (0,235); %set diagram left start at 0, and has height of 235

%Straight Lines [id:da8545627489039611] 
\draw [color={rgb, 255:red, 0; green, 0; blue, 255 }  ,draw opacity=1 ][line width=1.5]    (100,111) -- (200,111) ;
%Shape: Circle [id:dp26055871143917075] 
\draw  [fill={rgb, 255:red, 0; green, 0; blue, 0 }  ,fill opacity=1 ] (195,111) .. controls (195,108.24) and (197.24,106) .. (200,106) .. controls (202.76,106) and (205,108.24) .. (205,111) .. controls (205,113.76) and (202.76,116) .. (200,116) .. controls (197.24,116) and (195,113.76) .. (195,111) -- cycle ;
%Straight Lines [id:da10749896890831767] 
\draw [color={rgb, 255:red, 0; green, 0; blue, 0 }  ,draw opacity=1 ][line width=1.5]    (200,111) -- (200,41) ;
%Shape: Circle [id:dp7148480296477373] 
\draw  [fill={rgb, 255:red, 0; green, 0; blue, 0 }  ,fill opacity=1 ] (445,111) .. controls (445,108.24) and (447.24,106) .. (450,106) .. controls (452.76,106) and (455,108.24) .. (455,111) .. controls (455,113.76) and (452.76,116) .. (450,116) .. controls (447.24,116) and (445,113.76) .. (445,111) -- cycle ;
%Shape: Arc [id:dp41486476999926913] 
\draw  [draw opacity=0][dash pattern={on 0.84pt off 2.51pt}] (459.66,99.52) .. controls (462.92,102.27) and (465,106.39) .. (465,111) .. controls (465,119.28) and (458.28,126) .. (450,126) .. controls (441.72,126) and (435,119.28) .. (435,111) .. controls (435,103.67) and (440.26,97.57) .. (447.21,96.26) -- (450,111) -- cycle ; \draw  [dash pattern={on 0.84pt off 2.51pt}] (459.66,99.52) .. controls (462.92,102.27) and (465,106.39) .. (465,111) .. controls (465,119.28) and (458.28,126) .. (450,126) .. controls (441.72,126) and (435,119.28) .. (435,111) .. controls (435,103.67) and (440.26,97.57) .. (447.21,96.26) ;  
%Shape: Arc [id:dp9251235147556214] 
\draw  [draw opacity=0] (451.8,96.11) .. controls (454.05,96.38) and (456.14,97.14) .. (457.97,98.29) -- (450,111) -- cycle ; \draw   (451.8,96.11) .. controls (454.05,96.38) and (456.14,97.14) .. (457.97,98.29) ;  
\draw   (456.56,95.48) -- (458.04,98.85) -- (454.36,98.82) ;
%Straight Lines [id:da47263035339706927] 
\draw [color={rgb, 255:red, 0; green, 0; blue, 0 }  ,draw opacity=1 ][line width=1.5]    (450,111) -- (450,41) ;
%Shape: Arc [id:dp8624357849233557] 
\draw  [draw opacity=0][dash pattern={on 0.84pt off 2.51pt}] (200.69,96.02) .. controls (208.65,96.38) and (215,102.95) .. (215,111) .. controls (215,119.28) and (208.28,126) .. (200,126) .. controls (191.96,126) and (185.4,119.68) .. (185.02,111.74) -- (200,111) -- cycle ; \draw  [dash pattern={on 0.84pt off 2.51pt}] (200.69,96.02) .. controls (208.65,96.38) and (215,102.95) .. (215,111) .. controls (215,119.28) and (208.28,126) .. (200,126) .. controls (191.96,126) and (185.4,119.68) .. (185.02,111.74) ;  
%Shape: Arc [id:dp6280893643148724] 
\draw  [draw opacity=0] (185.05,109.76) .. controls (185.63,102.67) and (191.14,96.98) .. (198.14,96.11) -- (200,111) -- cycle ; \draw   (185.05,109.76) .. controls (185.63,102.67) and (191.14,96.98) .. (198.14,96.11) ;  
\draw   (194.77,94.63) -- (198.25,95.83) -- (195.72,98.51) ;
%Curve Lines [id:da627070108258077] 
\draw [color={rgb, 255:red, 0; green, 0; blue, 255 }  ,draw opacity=1 ][line width=1.5]    (350,111) .. controls (388.2,53.8) and (400.2,24.2) .. (443.4,24.6) .. controls (486.6,25) and (486.38,66.77) .. (450,111) ;

\draw (192,34) node [anchor=north west][inner sep=0.75pt]  [font=\Large] [align=left] {$\times$};
% Text Node
\draw (442,34) node [anchor=north west][inner sep=0.75pt]  [font=\Large] [align=left] {$\times$};
% Text Node
\draw (255,58) node [anchor=north west][inner sep=0.75pt]  [font=\LARGE] [align=left] {$ \mathrel{\substack{
\overset{\mu_P^+}{\longrightarrow}\\[-0.3ex]
\underset{\mu_{P}^-}{\longleftarrow}
}} $};
% Text Node
\draw (92,105) node [anchor=north west][inner sep=0.75pt]  [font=\Large] [align=left] {$\times$};
% Text Node
\draw (343,105) node [anchor=north west][inner sep=0.75pt]  [font=\Large] [align=left] {$\times$};

\end{tikzpicture}

\medskip

\tikzset{every picture/.style={line width=0.75pt}} %set default line width to 0.75pt        

\begin{tikzpicture}[x=0.75pt,y=0.75pt,yscale=-1,xscale=1]
%uncomment if require: \path (0,235); %set diagram left start at 0, and has height of 235

%Straight Lines [id:da10749896890831767] 
\draw [color={rgb, 255:red, 0; green, 0; blue, 0 }  ,draw opacity=1 ][line width=1.5]    (200,111) -- (200,41) ;
%Shape: Arc [id:dp41486476999926913] 
\draw  [draw opacity=0][dash pattern={on 0.84pt off 2.51pt}] (465,110.66) .. controls (465,110.77) and (465,110.89) .. (465,111) .. controls (465,119.28) and (458.28,126) .. (450,126) .. controls (441.72,126) and (435,119.28) .. (435,111) .. controls (435,103.67) and (440.26,97.57) .. (447.21,96.26) -- (450,111) -- cycle ; \draw  [dash pattern={on 0.84pt off 2.51pt}] (465,110.66) .. controls (465,110.77) and (465,110.89) .. (465,111) .. controls (465,119.28) and (458.28,126) .. (450,126) .. controls (441.72,126) and (435,119.28) .. (435,111) .. controls (435,103.67) and (440.26,97.57) .. (447.21,96.26) ;  
%Shape: Arc [id:dp9251235147556214] 
\draw  [draw opacity=0] (451.8,96.11) .. controls (454.05,96.38) and (456.14,97.14) .. (457.97,98.29) -- (450,111) -- cycle ; \draw   (451.8,96.11) .. controls (454.05,96.38) and (456.14,97.14) .. (457.97,98.29) ;  
\draw   (456.56,95.48) -- (458.04,98.85) -- (454.36,98.82) ;
%Straight Lines [id:da47263035339706927] 
\draw [color={rgb, 255:red, 0; green, 0; blue, 0 }  ,draw opacity=1 ][line width=1.5]    (450,111) -- (450,41) ;
%Shape: Arc [id:dp8624357849233557] 
\draw  [draw opacity=0][dash pattern={on 0.84pt off 2.51pt}] (200.69,96.02) .. controls (208.65,96.38) and (215,102.95) .. (215,111) .. controls (215,119.28) and (208.28,126) .. (200,126) .. controls (193.74,126) and (188.38,122.17) .. (186.13,116.72) -- (200,111) -- cycle ; \draw  [dash pattern={on 0.84pt off 2.51pt}] (200.69,96.02) .. controls (208.65,96.38) and (215,102.95) .. (215,111) .. controls (215,119.28) and (208.28,126) .. (200,126) .. controls (193.74,126) and (188.38,122.17) .. (186.13,116.72) ;  
%Shape: Arc [id:dp6280893643148724] 
\draw  [draw opacity=0] (193.65,97.41) .. controls (195.04,96.76) and (196.55,96.31) .. (198.14,96.11) -- (200,111) -- cycle ; \draw   (193.65,97.41) .. controls (195.04,96.76) and (196.55,96.31) .. (198.14,96.11) ;  
\draw   (194.77,94.63) -- (198.25,95.83) -- (195.72,98.51) ;
%Curve Lines [id:da4207984294791134] 
\draw [color={rgb, 255:red, 0; green, 0; blue, 255 }  ,draw opacity=1 ][line width=1.5]    (196.8,109.2) .. controls (192.47,86.53) and (137.33,22.33) .. (108,51) .. controls (78.67,79.67) and (180.07,130.93) .. (200,111) ;
%Shape: Circle [id:dp896839477586505] 
\draw  [fill={rgb, 255:red, 0; green, 0; blue, 0 }  ,fill opacity=1 ] (195,111) .. controls (195,108.24) and (197.24,106) .. (200,106) .. controls (202.76,106) and (205,108.24) .. (205,111) .. controls (205,113.76) and (202.76,116) .. (200,116) .. controls (197.24,116) and (195,113.76) .. (195,111) -- cycle ;
\draw   (187.09,99.41) -- (190.76,99.03) -- (189.59,102.53) ;
%Shape: Arc [id:dp3472457549274752] 
\draw  [draw opacity=0] (185.28,113.9) .. controls (185.1,112.96) and (185,111.99) .. (185,111) .. controls (185,106.36) and (187.11,102.22) .. (190.41,99.46) -- (200,111) -- cycle ; \draw   (185.28,113.9) .. controls (185.1,112.96) and (185,111.99) .. (185,111) .. controls (185,106.36) and (187.11,102.22) .. (190.41,99.46) ;  
%Curve Lines [id:da3179275704468616] 
\draw [color={rgb, 255:red, 0; green, 0; blue, 255 }  ,draw opacity=1 ][line width=1.5]    (451.8,107.8) .. controls (482,84.8) and (500.6,48.2) .. (471.4,33.4) .. controls (442.2,18.6) and (400.6,64.6) .. (390.2,45) .. controls (379.8,25.4) and (424.2,-0.6) .. (484.2,10.6) .. controls (544.2,21.8) and (517.4,66.6) .. (506.6,79) .. controls (495.8,91.4) and (462.52,114.51) .. (450,111) ;
%Shape: Circle [id:dp46217081543424055] 
\draw  [fill={rgb, 255:red, 0; green, 0; blue, 0 }  ,fill opacity=1 ] (445,111) .. controls (445,108.24) and (447.24,106) .. (450,106) .. controls (452.76,106) and (455,108.24) .. (455,111) .. controls (455,113.76) and (452.76,116) .. (450,116) .. controls (447.24,116) and (445,113.76) .. (445,111) -- cycle ;
%Shape: Arc [id:dp9418678546613215] 
\draw  [draw opacity=0] (461.47,101.34) .. controls (462.7,102.79) and (463.65,104.47) .. (464.25,106.31) -- (450,111) -- cycle ; \draw   (461.47,101.34) .. controls (462.7,102.79) and (463.65,104.47) .. (464.25,106.31) ;  
\draw   (464.79,102.59) -- (464.16,106.23) -- (461.1,104.15) ;

\draw (192,34) node [anchor=north west][inner sep=0.75pt]  [font=\Large] [align=left] {$\times$};
% Text Node
\draw (442,34) node [anchor=north west][inner sep=0.75pt]  [font=\Large] [align=left] {$\times$};
% Text Node
\draw (278,58) node [anchor=north west][inner sep=0.75pt]  [font=\LARGE] [align=left] {$ \mathrel{\substack{
\overset{\mu_P^+}{\longrightarrow}\\[-0.3ex]
\underset{\mu_{P}^-}{\longleftarrow}
}} $};

\end{tikzpicture}

		\end{center}
					\caption{Kauer moves on  orbifold ribbon graphs, assuming clockwise cyclic orderings at all vertices.}
		\label{fig:Kauer-moves-orbifold}	
	\end{figure}
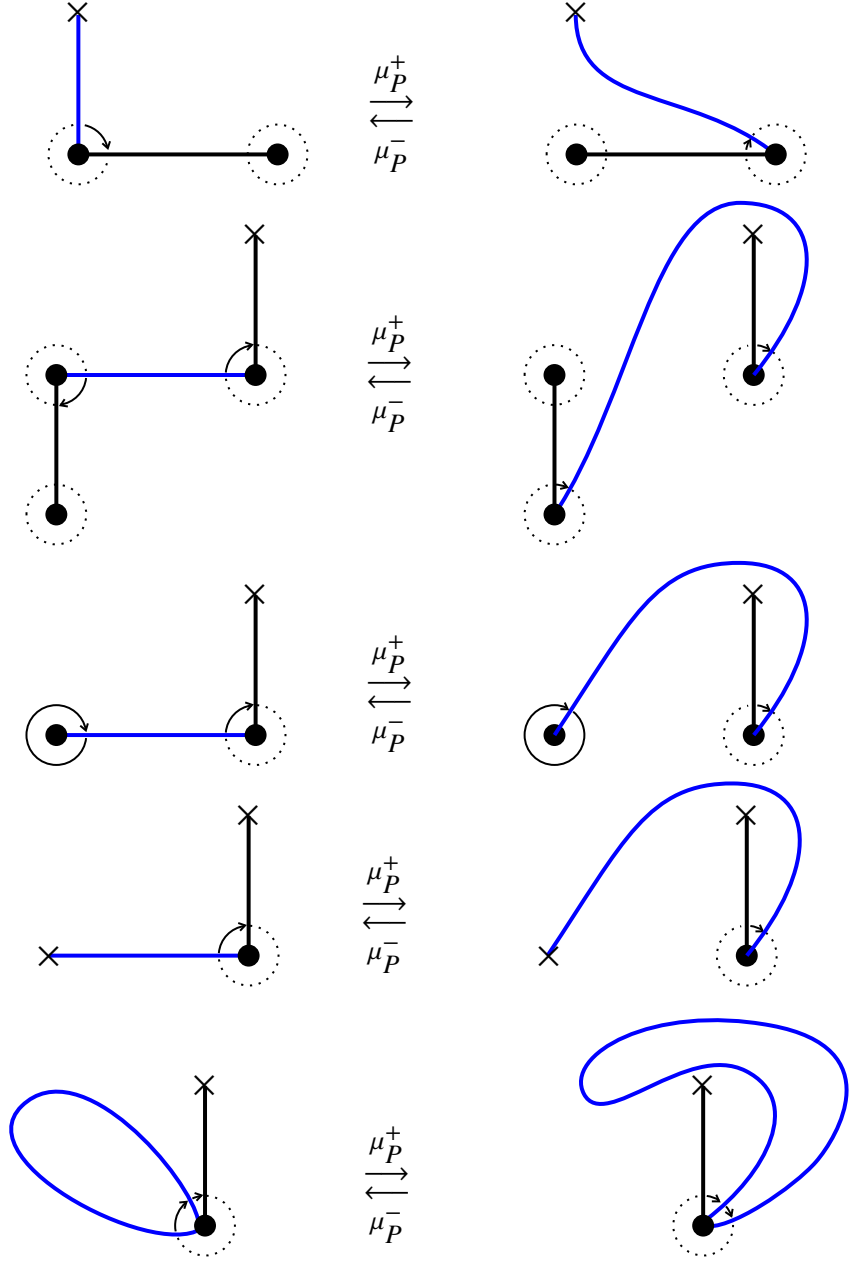	

	We note that in general, as illustrated in \cite[Example~1.16]{So1}, a generalized Kauer move (that is, a tilting mutation which need not be irreducible) cannot, in general, be realized simply as a succession of ordinary Kauer moves. However, in the tilting-connected case, one can obtain the following consequence by iterating irreducible mutations.

	\begin{corollary}\label{cor:gen-kauer-move}
		Let $A$ be a tilting-connected biserial FBGA with associated biserial FBG $(\Gamma,d)$. Let $H'$ be a set of half-edges of $\Gamma$ that is stable under both $\iota$ and $\nu$ of $(\Gamma,d)$. Denote by $(\mu_{H'}^-(\Gamma),d)$ the biserial FBG obtained by applying local moves (from left to right) that simultaneously shift all cyclically consecutive subsets of $H'$ (which is also corresponding to local moves on the orbifold ribbon graph $\Gamma/\langle\nu\rangle$), as illustrated in Figure~\ref{fig:gen-Kauer-move}. Let $I$ be the set of indecomposable projective $A$-modules corresponding to the edges induced by $H'$. Then there exists a $k$-algebra isomorphism
$$
\mathrm{End}_{\mathcal{K}^b(\mathrm{proj}\,A)}(\mu^-_{P(I)}(A))
\;\cong\;
kQ_{\mu_{H'}^-(\Gamma)} \big/ I_{\mu_{H'}^-(\Gamma)} .
$$	\end{corollary}

\begin{proof}
	Since $A$ is tilting-connected, the tilting mutation induced by the Okuyama--Rickard complex associated with $I$ can be obtained by a sequence of mutations at orbits of simple modules that are stable under the Nakayama automorphis of $A$. Therefore, by the discussion in this section, the statement of this proposition follows naturally.
\end{proof}

\begin{remark}
	The above corollary suggests that tilting mutations for biserial FBGAs should follow the same pattern as for (skew-)BGAs in \cite{So2}. These mutations correspond to generalized Kauer moves on $\Gamma / \langle \nu \rangle$, as illustrated in Figure~\ref{fig:gen-Kauer-move}. This correspondence can be verified by arguments analogous to those in the appendix, though the verification involves rather technical computations.

There are also alternative ways to verify this result. For example, \cite{So1} studies generalized Kauer moves for gentle algebras via the geometric models introduced in \cite{OPS}, and transfers these results to BGAs using known results on trivial extensions \cite{Ric2,Sch}.
\end{remark}

	\section{Two-term tilting complexes of biserial FBGAs}\label{sec:2-term-tilt}

\subsection{Walks on biserial FBGs}\label{subsec:zigzag-walk}
	\

	We first define walks, the sign condition, and the non-crossing condition on biserial FBGs. These are slight modifications of the definitions in \cite{AAC}, in order to make them compatible with the Nakayama automorphism on the FBG.

	\begin{definition}[Walk]
Let $(\Gamma,d)$ be a biserial FBG with set of half-edges $H$ and involution $\iota:H\to H$.  
A {\it half-walk} is a non-empty sequence
\[
w=(h_1,\dots,h_\ell)
\]
of half-edges such that
\[
s(h_{i+1})=s(\iota(h_i))
\quad\text{for all }1\le i\le \ell-1.
\]
Define
\[
\iota(w):=(\iota(h_\ell),\dots,\iota(h_1)).
\]
A {\it walk} on $\Gamma$ is the $\iota$-orbit
\[
W=\{w,\iota(w)\}.
\]
For $w=(h_1,\dots,h_\ell)$, the endpoints of $W$ are
\[
s(w):=s(h_1),
\qquad
s(\iota(w)):=s(\iota(h_\ell)).
\]
\end{definition}

\begin{definition}[Signed walk and sign condition] \label{def:signed-walk}
Let $W=\{w,\iota(w)\}$ be a walk with $w=(h_1,\dots,h_\ell)$.  
A {\it signature} on $W$ is a map
\[
\varepsilon_W:\{h_1,\dots,h_\ell\}\to\{+,-\}
\]
such that
\[
\varepsilon_W(h_i)\neq\varepsilon_W(h_{i+1})
\quad\text{for all }1\le i\le \ell-1.
\]
Moreover, for all $1\le i\le \ell-1$, there exists a positive number $k_i\leq o(s(\iota(h_i)))$ (here $o$ means the number of $\langle\nu\rangle$-orbits around $s(\iota(h_i))$, see Subsection \ref{subsec:numbers-of-bFBGA}), such that
\[
\rho^{k_i}(\iota(h_i))=h_{i+1}\ \text{ if \;} \varepsilon_W(h_i)=-,\qquad
\rho^{k_i}(h_{i+1})=\iota(h_{i})\ \text{ if \;} \varepsilon_W(h_i)=+
\]
That is, for each pair $(h_i,h_{i+1})$, there exists a $\nu$-fan in $(\Gamma,d)$ containing both of them.
A walk equipped with a signature is called a {\it signed walk}.

A signed walk $W$ satisfies the {\it sign condition} if
\[
\varepsilon_W(h_1)=\varepsilon_W(h_\ell)
\quad\text{whenever } s(h_1)=s(\iota(h_\ell)).
\]

For two signed walks $W=\{w,\iota(w)\}$ and $W'=\{w',\iota(w')\}$ with
\[
w=(h_1,\dots,h_m),\qquad
w'=(h'_1,\dots,h'_n),
\]
we say that $W$ and $W'$ satisfy the {\it sign condition} if the signs agree at every common endpoint, i.e.
\[
\varepsilon_W(h_1)=\varepsilon_{W'}(h'_1)\ \text{ if } s(h_1)=s(h'_1),\qquad
\varepsilon_W(h_1)=\varepsilon_{W'}(h'_n)\ \text{ if } s(h_1)=s(\iota(h'_n)),
\]
\[
\varepsilon_W(h_m)=\varepsilon_{W'}(h'_1)\ \text{ if } s(\iota(h_m))=s(h'_1),\qquad
\varepsilon_W(h_m)=\varepsilon_{W'}(h'_n)\ \text{ if } s(\iota(h_m))=s(\iota(h'_n)).
\]
\end{definition}

\begin{definition}[Non-crossing condition]
Let $W$ and $W'$ be signed walks.

A {\it subwalk} of $W$ is a walk given by a continuous subsequence of some representative half-walk of $W$.  
A {\it maximal common subwalk} $Z$ of $W$ and $W'$ is a common subwalk which is maximal with respect to inclusion.

We say that $W$ and $W'$ are {\it non-crossing at some $Z$} if
\begin{itemize}
\item[(NC1)] the signatures coincide on $Z$;
\item[(NC2)] the cyclic orderings around the endpoints of $Z$ induced by $W$ and $W'$ are compatible; that is, if $\{h_1,Z,h_2\}$ (resp. $\{h_1',Z,h_2'\}$) is the subwalk of $W$ (resp. $W'$), then the cyclic orderings around $s(Z)$ and $s(\iota(Z))$ admit the following local configurations.
$$% https://tikzcd.yichuanshen.de/#N4Igdg9gJgpgziAXAbVABwnAlgFyxMJZARgBpiBdUkANwEMAbAVxiRDgAoAtAShAF9S6TLnyEUAJnJVajFmwFCQGbHgJEAzNOr1mrRCEXDVYogBZtsvW04AdW-hx1uPPoOOj1KAAylvM3XkDI2URNXFkXwkAuX1Dd1CTL2QAVj8Y62CElU8ItOidWIV+GRgoAHN4IlAAMwAnCABbJF8QHAgkMhAACxg6KDZIMFYE+qbO6nakKR6+gYMh1moGLGG2KDo4XoHCzJAuELHmxBmpxC1Z-sGCEaUjpAs2jsRW3quFm5BdoJB7R2dugB9YhuO4NY5pJ4tahvebgT7fOJ-CBODhA4gAcj4yzoACMYAwAAphUwGOpYcrdHCHcFIC5nABsMLm1zWiLYQIkNPG50mzwA7Mz3vC2VYfuiMV8QAw8QTiUlxCByZTqSV+EA
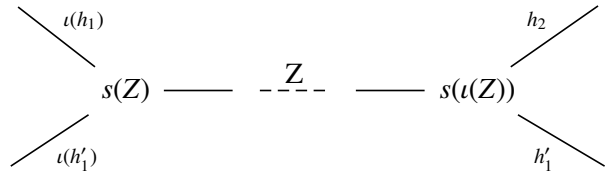
\begin{tikzcd}
{} \arrow[rd, "\iota(h_1)", no head]   &                         &                                    &                       &                                                                     & {} \\
                                       & s(Z) \arrow[r, no head] & {} \arrow[r, "{\scalebox{1}{Z}}", no head, dashed] & {} \arrow[r, no head] & s(\iota(Z)) \arrow[ru, "h_2", no head] \arrow[rd, "h_1'"', no head] &    \\
{} \arrow[ru, "\iota(h_1')"', no head] &                         &                                    &                       &                                                                     & {}
\end{tikzcd}$$
\end{itemize}

Let $\vv$ be a vertex belonging to both $W$ and $W'$. For $\vv=s(h_i)$ in $W$, define its neighbourhood in $W$ as
\[
\{\iota(h_{i-1}),\,h_i\}.
\]
We call $\vv$ an {\it intersecting vertex} of $W$ and $W'$ if the neighbourhoods of $\vv$ in $W$ and in $W'$ consist of four pairwise distinct half-edges.

We say that $W$ and $W'$ are {\it non-crossing at $\vv$} if the cyclic ordering of these four half-edges around $\vv$, together with their signs (for example, assuming $\vv$ is an intersecting vertex with respect to the neighbourhoods $\{a,b\}$ in $W$ and $\{c,d\}$ in $W'$), is one of the two alternating configurations described below.
$$% https://tikzcd.yichuanshen.de/#N4Igdg9gJgpgziAXAbVABwnAlgFyxMJZARgBpiBdUkANwEMAbAVxiVpAF9T1Nd9CUAFnJVajFmxqduIDNjwEiABlJLR9Zq0QhpPefyIAmVevFadXPX0Upjh05ra7ZvBQOQr71DRO3O51u4AzCbeZk6WLvo2yACsoWKOfpEBbkTxXom+FjKpBighmT7mnKIwUADm8ESgAGYAThAAtkjGIDgQSCogABYwdFBskGCsYUkgdAB6ANTODc1IIe2diN19A0MEo1nmAEaTALQg1Ax0uzAMAAqu+SD1WBU9OHONLavUHUjCvf2D2sPbYpsADGM2OIFO5yuNxsdweTxeC0Q8WWXWo6z+4C24KB2igh3BkIu12iAjhj2ekXmbwAbB8VmQfht-tixtkprMqa8kIzPogAOzo36bEY48LafZHLlI3krAAcQuZWNFbPMoKOJzOxJhZPuFMRbwAnPSeYrMQCxeN8bNNVCSYE2HqERwKBwgA
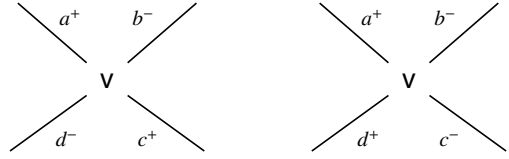
\begin{tikzcd}
{} \arrow[rd, "a^+", no head]  &                               & {} \arrow[ld, "b^-"', no head] & {} \arrow[rd, "a^+", no head]  &                                                          & {} \\
                               & \vv \arrow[rd, "c^+"', no head] &                                &                                & \vv \arrow[ru, "b^-", no head] \arrow[rd, "c^-"', no head] &    \\
{} \arrow[ru, "d^-"', no head] &                               & {}                             & {} \arrow[ru, "d^+"', no head] &                                                          & {}
\end{tikzcd}$$

Finally, $W$ and $W'$ are {\it non-crossing} if they are non-crossing at every maximal common subwalk and every intersecting vertex.
\end{definition}

\begin{definition}\textnormal{(cf. \cite[Definition 4.5]{AAC})}
	A set $\mathbb{W}$ of signed walks is {\it admissible} if for any pair of (not necessarily distinct) walks $W$ and $W'$ in $\mathbb{W}$, they satisfy the sign condition and are non-crossing. 
An admissible set is called {\it complete} if any admissible set containing $\mathbb{W}$ coincides with $\mathbb{W}$. 
Denote by $\mathrm{AW}(\Gamma,d)$ (resp.\ $\mathrm{CW}(\Gamma,d)$) the set of all admissible (resp.\ complete) sets of signed walks of the biserial FBG $(\Gamma,d)$.
\end{definition}

As shown by \cite[Theorem 4.6]{AAC} and \cite[Corollary 7.31]{AY}, when $(\Gamma,d)$ is a BG, then there exist bijections between the following sets:
$$\mathrm{AW}\,(\Gamma,d)\rightarrow 2\text{-}\mathrm{presilt}\,A_\Gamma,$$
$$\mathrm{CW}\,(\Gamma,d)\rightarrow 2\text{-}\mathrm{tilt}\,A_\Gamma.$$

	Now let $\varphi$ be an automorphism of $\Gamma$, which also induces an automorphism of $A$. For each (signed) walk $$W=\{w=(h_1,\cdots,h_\ell), \iota(w)=(\iota(h_\ell),\cdots,\iota(h_1))\},$$ define $$\varphi(W):=\{\varphi(w)=(\varphi(h_1),\cdots,\varphi(h_\ell)), \varphi(\iota(w))=(\varphi(\iota(h_\ell)),\cdots,\varphi(\iota(h_1)))\}.$$
	Moreover, define 
	$$\langle W\rangle :=\{\varphi^i(W)\mid i\in\mathbb{Z}\}.$$
	Then a signed walk set $\mathbb{W}$ of $(\Gamma,d)$ is {\it $\varphi$-stable} if $\varphi(\mathbb{W})=\mathbb{W}$. Denote the set of $\varphi$-stable admissible signed walk sets of $A$ by $\mathrm{AW}^\varphi\,(\Gamma,d)$. Then we have the following proposition.

	\begin{lemma}\label{lem:bijection-on-zzwalk}
	Let $A$ be a biserial FBGA with associated biserial FBG $(\Gamma,d)$ and Nakayama automorphism $\nu$ on $(\Gamma,d)$. If $(\Gamma,d)$ is admissible, then there exists a bijection
	$$\psi\colon \mathrm{AW}^\nu\, (\Gamma,d) \rightarrow \mathrm{AW}\,(\Gamma_{\red},d).$$
	If $(\Gamma,d)$ is not admissible, then there exists a bijection
	$$\psi\colon \mathrm{AW}^\nu\, (\Gamma,d) \rightarrow \mathrm{AW}^{\phi}\,(\Gamma_{\red},d),$$ where $\phi$ is the natural involution on $(\Gamma_{\red},d)$.
\end{lemma}

\begin{proof}
	By the definition of the reduced form, if $(\Gamma,d)$ is admissible, then there exists a natural bijection
	$$f\colon \{\langle\nu\rangle\text{-orbits of } H(\Gamma)\}\longrightarrow H(\Gamma_{red}),$$
	and if $(\Gamma,d)$ is not admissible, then there exists a natural bijection
	$$f\colon \{\langle\nu\rangle\text{-orbits of } H(\Gamma)\}\longrightarrow \{\langle \phi\rangle\text{-orbits of } H(\Gamma_{red})\}.$$

	Define $\psi$ to be the map induced by the following bijection between the set of $\nu$-stable signed walks in $(\Gamma,d)$ and the set of ($\phi$-stable, when $(\Gamma,d)$ is non-admissible) signed walks in $(\Gamma_{\red},d)$, given by
$$
\langle \{w=(h_1,\dots,h_\ell),\, w^{-1}\}\rangle
\ \longmapsto\
\langle\{w'=(h_1',\dots,h_\ell'),\, (w')^{-1}\}\rangle,
$$
where $h_1'\in f([h_1])$. For each $1\le i\le \ell-1$, if
\[
h_{i+1}=\rho^k(\iota(h_i))
\]
for some integer $k$, then we define
\[
h_{i+1}'=\rho^k(\iota(h_i')).
\] This map $\psi$ is well-defined and induces the required bijection. Indeed, $\psi$ preserves admissibility: for each walk, any pair of consecutive half-edges $\{\iota(h_1),h_2\}$ at a vertex $\vv$ lies, by definition, in the same $\nu$-fan, and under $\psi$ each such $\nu$-fan in $\Gamma$ coincides locally with its image in $\Gamma_{\mathrm{red}}$. Hence if two walks satisfy the sign condition and the non-crossing condition, then so do their images under $\psi$.
Conversely, any local crossing configuration is preserved (and lifts) under $\psi$; thus if the images failed to be admissible, then the original walks would fail to be admissible as well.
\end{proof}

\begin{remark}
	We note that the above verification can also be understood via the surface model in~\cite{BC}. Indeed, for a biserial FBGA $A$ associated with $(\Gamma,d)$, the quotient $A/\operatorname{soc}(A)$ is a string algebra, and hence can be interpreted using the surface model in~\cite{BC}, given by the ribbon surface of $\Gamma$ with the label at each vertex $\vv$ having length $d(\vv)$.
From this perspective, since, by \cite[Theorem~11]{EJR},
$2\text{-}\mathrm{silt}\,A$ and
$2\text{-}\mathrm{silt}\,A/\operatorname{soc}(A)$ coincide,
the correspondence between walks above is essentially induced by the
correspondence between curves on the associated surfaces. Indeed, the surface
associated with $\Gamma$ may, in a suitable sense (for example, in the
admissible case), be viewed as a branched cover of the surface associated with
$\Gamma_{\mathrm{red}}$.
\end{remark}

\subsection{Admissible cases}
\

We now compare two-term tilting complexes of an admissible biserial FBGA $A$ with those of its reduced form $A_{\red}$.

\begin{proposition}\label{prop:adm-presilting-corres}
Let $A$ be an admissible biserial FBGA with associated biserial FBG $(\Gamma,d)$ and Nakayama automorphism $\nu$ of $(\Gamma,d)$. Denote by $\nu_A$ the Nakayama automorphism of $A$. Then there is a bijection among the following sets:
\begin{enumerate}
    \item the set $2\text{-}\mathrm{presilt}^{\nu_A} A$ of isomorphism classes of basic $\nu_A$-stable two-term presilting complexes of $A$;
    \item the set $\mathrm{AW}^{\nu}(\Gamma,d)$ of $\nu$-stable admissible sets of singed walks on the biserial FBG $(\Gamma,d)$;
    \item the set $2\text{-}\mathrm{presilt}\,A_{\red}$ of isomorphism classes of basic two-term presilting complexes of $A_{\red}$;
    \item the set $\mathrm{AW}\,(\Gamma_{\red},d)$ of admissible sets of singed walks on the BG $(\Gamma_{\red},d)$.
\end{enumerate}
\end{proposition}

\begin{proof}
By Lemma~\ref{lem:bijection-on-zzwalk}, the sets in~(2) and~(4) coincide. By~\cite[Theorem~4.6]{AAC}, the sets in~(3) and~(4) coincide. Hence, it suffices to show that the sets in~(1) and~(2) coincide. As in the proof of~\cite[Lemma~4.4]{AAC}, we obtain that~(1) is contained in~(2). Conversely, by the correspondence given in~\cite[Lemma~4.3]{AAC}, each admissible set of signed walks on the biserial FBG $(\Gamma,d)$ induces a string complex $T$ of $A$. It is naturally $\nu_A$-stable since by Proposition \ref{prop:bfBGA-basic} and Theorem \ref{prop:sel-inj-silting}, $\nu_A$ can be induced by $\nu$. Such a complex $T$ is presilting, since it satisfies the non-crossing conditions.
\end{proof}

We are now ready to prove the main result of this subsection.

\begin{theorem}\label{thm:adm-case}
Let $A$ be a biserial FBGA with associated biserial FBG $(\Gamma,d)$ that is admissible,, and let $A_{\red}$ be the BGA associated with the reduced form $(\Gamma_{\red},d)$. Then there is a poset isomorphism between
\begin{enumerate}[(a)]
  \item the set $2\text{-}\mathrm{tilt}\,A$ of isomorphism classes of basic two-term tilting complexes of $A$, and
  \item the set $2\text{-}\mathrm{tilt}\,A_{\red}$ of isomorphism classes of basic two-term tilting complexes of $A_{\red}$.
\end{enumerate}
Moreover, the following statements are equivalent:
\begin{enumerate}
  \item $A$ is tilting-discrete;
  \item $2\text{-}\mathrm{tilt}\,A$ is finite;
  \item $A_{\red}$ is tilting-discrete;
  \item $2\text{-}\mathrm{tilt}\,A_{\red}$ is finite;
  \item $\Gamma_{\red}$ contains at most one odd cycle and no even cycles.
\end{enumerate}
\end{theorem}

\begin{proof}
By Proposition~\ref{prop:adm-skew=BGA}, $A_{\mathrm{red}}$ can be regarded as a skew group algebra of $A$. Hence the bijection between $(2\text{-})\mathrm{tilt}\,A$ and $(2\text{-})\mathrm{tilt}\,A_{\mathrm{red}}$ follows from~\cite[Theorem~1.1]{KKKMM}. More concretely, the bijection between (a) and (b) is naturally induced by Proposition~\ref{prop:adm-presilting-corres}, since it preserves indecomposable direct summands. Moreover, this correspondence is an isomorphism of posets.
 Indeed, let $T$ and $T'$ be tilting complexes. If
\[
\mathrm{Hom}_{\mathcal{K}^b(\mathrm{proj}\,A)}(T,T'[i])=0
\quad \text{for all } i>0,
\]
then the same vanishing holds for any pair of indecomposable direct summands of $T$ and $T'$ (with the second one shifted).
Let $T_{\red}$ and $T'_{\red}$ be the tilting complexes over $A_{\red}$ corresponding to $T$ and $T'$ under the above bijection. Suppose that $T_{\red}\not\ge T'_{\red}$. Then there exist indecomposable direct summands $P_{\red}$ of $T_{\red}$ and $P'_{\red}$ of $T'_{\red}$ and some positive integer $i$ such that
\[
\mathrm{Hom}_{\mathcal{K}^b(\mathrm{proj}\,A_{\red})}(P_{\red},P'_{\red}[i])\neq 0.
\]
By the correspondence in Proposition~\ref{prop:adm-presilting-corres}, there exist indecomposable complexes $P$ and $P'$ corresponding to $P_{\red}$ and $P'_{\red}$, which are direct summands of $T$ and $T'$ respectively, such that
\[
\mathrm{Hom}_{\mathcal{K}^b(\mathrm{proj}\,A)}(P,P'[i])\neq 0.
\]
This contradicts the assumption that $T\ge T'$.

We now prove the second equivalences. Conditions~(3), (4), and~(5) are equivalent by~\cite[Theorem~6.7]{AAC}, and conditions~(2) and~(4) are equivalent by the first bijection of this theorem. It is clear that~(1) implies~(2). Now assume that~(2) holds. Let $T$ be a tilting complex lying in the connected component of the tilting mutation quiver $Q_{\mathrm{tilt}}(A)$ containing $A$. By Subsection \ref{subsec:KM-FMS-BGA}, the algebra $
B := \mathrm{End}_{\mathcal{K}^b(\mathrm{proj}\,A)}(T)$
is again a biserial FBGA associated with a biserial FBG $(\Gamma',d)$. As in the proof of~\cite[Lemma~6.12]{AAC}, the reduced graph $\Gamma'_{\red}$ contains at most one odd cycle and no even cycles. Consequently, the set $2\text{-}\mathrm{tilt}\,B$ is finite. By~\cite[Theorem~1.2]{AM}, it follows that $A$ is tilting-discrete.
\end{proof}

\subsection{Non-admissible cases}
\

We now consider the non-admissible cases. 

\begin{proposition}\label{prop:non-adm-presilting-corres}
Let $A$ be a non-admissible biserial FBGA with associated biserial FBG $(\Gamma,d)$ and Nakayama automorphism $\nu$ of $(\Gamma,d)$. Denote by $\nu_A$ the Nakayama automorphism of $A$ and by $\phi$ the natural involution on the reduced form $A_\red$. Then there is a bijection among the following sets:
\begin{enumerate}
    \item the set $2\text{-}\mathrm{presilt}^{\nu_A} A$ of isomorphism classes of basic $\nu_A$-stable two-term presilting complexes of $A$;
    \item the set $\mathrm{AW}^{\nu}(\Gamma,d)$ of $\nu$-stable admissible sets of singed walks on the biserial FBG $(\Gamma,d)$;
    \item the set $2\text{-}\mathrm{presilt}^{\phi}\,A_{\red}$ of isomorphism classes of basic $\phi$-stable two-term presilting complexes of $A_{\red}$;
    \item the set $\mathrm{AW}^{\phi}\,(\Gamma_{\red},d)$ of $\phi$-stable admissible sets of singed walks on the BG $(\Gamma_{\red},d)$.
\end{enumerate}
\end{proposition}

\begin{proof}
	This follows from the same verification as in the proof of Proposition~\ref{prop:adm-presilting-corres}.
\end{proof}

Then we have the following result.

\begin{theorem}\label{thm:non-adm-case}
Let $A$ be a biserial FBGA with associated biserial FBG $(\Gamma,d)$ that is non-admissible, and let $A_{\red}$ be the BGA associated with the reduced form $(\Gamma_{\red},d)$. Then there is a poset isomorphism between
\begin{enumerate}[(a)]
  \item the set $2\text{-}\mathrm{tilt}\,A$ of isomorphism classes of basic two-term tilting complexes of $A$, and
  \item the set $2\text{-}\mathrm{tilt}^{\phi}\,A_{\red}$ of isomorphism classes of basic $\phi$-stable two-term tilting complexes of $A_{\red}$.
\end{enumerate}
Moreover, the following statements are equivalent:
\begin{enumerate}
  \item $A$ is tilting-discrete;
  \item $2\text{-}\mathrm{tilt}\,A$ is finite;
  \item $A_{\red}$ is tilting-discrete;
  \item $2\text{-}\mathrm{tilt}\,A_{\red}$ is finite;
  \item $2\text{-}\mathrm{tilt}^{\phi}\,A_{\red}$ is finite;
  \item $\Gamma_{\red}$ contains at most one odd cycle and no even cycles;
  \item The underlying graph of $\Gamma_{\red}$ is a tree.
\end{enumerate}
\end{theorem}

\begin{proof}
	The first bijection between (a) and (b) is naturally induced by Proposition~\ref{prop:non-adm-presilting-corres}, since this bijection preserves indecomposable direct summands. Such a correspondence is an isomorphism between posets by the same verification in Theorem \ref{thm:adm-case}.

We now prove the second equivalences. Conditions~(3), (4), and~(6) are equivalent by~\cite[Theorem~6.7]{AAC}, conditions~(2) and~(5) are equivalent by the first bijection of this theorem. It is clear that~(1) implies~(2), and (4) implies~(5).

We first show that Conditions~(6) and~(7) are equivalent. It is clear that Condition~(7) implies Condition~(6).
Conversely, assume that Condition~(6) holds. Suppose, for contradiction, that the underlying graph of $\Gamma_{\mathrm{red}}$ is not a tree. Since Condition~(6) excludes even cycles, $\Gamma_{\mathrm{red}}$ must then contain an odd cycle.
Using that $\Gamma_{\mathrm{red}}$ is a double cover of $\Gamma/\langle \nu\rangle$, such an odd cycle cannot arise as a cycle in $\Gamma/\langle \nu\rangle$, since any odd cycle in $\Gamma/\langle \nu\rangle$ would lift to two odd cycles in $\Gamma_{\mathrm{red}}$, contradicting Condition~(6). Nor can such an odd cycle be formed by concatenating two lifted paths corresponding to a path in $\Gamma/\langle \nu\rangle$, since any such construction would necessarily involve an even number of edges in $\Gamma_{\mathrm{red}}$, again a contradiction.
Therefore $\Gamma_{\mathrm{red}}$ contains no cycles, and hence its underlying graph is a tree.

We then show that~(5) implies~(7). Indeed, if~(7) does not hold, then, as in the construction of~\cite[Proposition~2.12]{AAC}, one can construct infinitely many basic ${\phi}$-stable two-term tilting complexes of $A_{\red}$. For example, suppose that $\Gamma_{\red}$ contains an even cycle as shown in Figure~\ref{fig:even-red-cycle}, where the edges labeled by $1$ and $k$ are induced by the orbifold edges in $\Gamma/\langle \nu \rangle$. Then the involution on this cycle of $A_{\red}$ is given by
$$
1 \longleftrightarrow 1,\quad
2 \longleftrightarrow 2k,\quad
3 \longleftrightarrow 2k-1,\quad \cdots,\quad
k \longleftrightarrow k.
$$

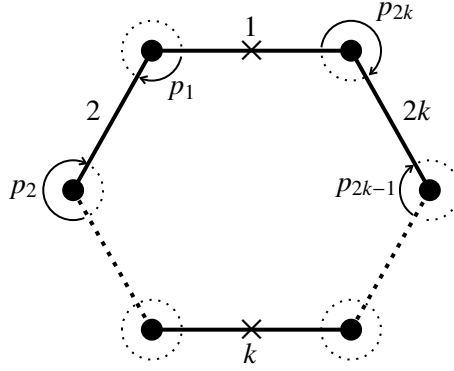
\begin{figure}
	\centering
		\begin{center}

\tikzset{every picture/.style={line width=0.75pt}} %set default line width to 0.75pt        

\begin{tikzpicture}[x=0.75pt,y=0.75pt,yscale=-1,xscale=1]
%uncomment if require: \path (0,235); %set diagram left start at 0, and has height of 235

%Straight Lines [id:da7130523259089323] 
\draw [line width=1.5]    (99.5,31.5) -- (199.5,31.5) ;
%Straight Lines [id:da20595582346696473] 
\draw [line width=1.5]    (99.5,31.5) -- (60,101.5) ;
%Straight Lines [id:da46933656794843803] 
\draw [line width=1.5]    (99.5,171.5) -- (199.5,171.5) ;
%Straight Lines [id:da32172828403671305] 
\draw [line width=1.5]  [dash pattern={on 1.69pt off 2.76pt}]  (60,101.5) -- (99.5,171.5) ;
%Straight Lines [id:da3106735683391819] 
\draw [line width=1.5]    (199.5,31.5) -- (239,101.5) ;
%Straight Lines [id:da7273283120911587] 
\draw [line width=1.5]  [dash pattern={on 1.69pt off 2.76pt}]  (199.5,171.5) -- (239,101.5) ;
%Shape: Circle [id:dp14102509584697875] 
\draw  [fill={rgb, 255:red, 0; green, 0; blue, 0 }  ,fill opacity=1 ] (94.5,31.5) .. controls (94.5,28.74) and (96.74,26.5) .. (99.5,26.5) .. controls (102.26,26.5) and (104.5,28.74) .. (104.5,31.5) .. controls (104.5,34.26) and (102.26,36.5) .. (99.5,36.5) .. controls (96.74,36.5) and (94.5,34.26) .. (94.5,31.5) -- cycle ;
%Shape: Circle [id:dp9197376204116503] 
\draw  [fill={rgb, 255:red, 0; green, 0; blue, 0 }  ,fill opacity=1 ] (55,101.5) .. controls (55,98.74) and (57.24,96.5) .. (60,96.5) .. controls (62.76,96.5) and (65,98.74) .. (65,101.5) .. controls (65,104.26) and (62.76,106.5) .. (60,106.5) .. controls (57.24,106.5) and (55,104.26) .. (55,101.5) -- cycle ;
%Shape: Circle [id:dp6317359860103804] 
\draw  [fill={rgb, 255:red, 0; green, 0; blue, 0 }  ,fill opacity=1 ] (94.5,171.5) .. controls (94.5,168.74) and (96.74,166.5) .. (99.5,166.5) .. controls (102.26,166.5) and (104.5,168.74) .. (104.5,171.5) .. controls (104.5,174.26) and (102.26,176.5) .. (99.5,176.5) .. controls (96.74,176.5) and (94.5,174.26) .. (94.5,171.5) -- cycle ;
%Shape: Circle [id:dp9428714432468424] 
\draw  [fill={rgb, 255:red, 0; green, 0; blue, 0 }  ,fill opacity=1 ] (194.5,171.5) .. controls (194.5,168.74) and (196.74,166.5) .. (199.5,166.5) .. controls (202.26,166.5) and (204.5,168.74) .. (204.5,171.5) .. controls (204.5,174.26) and (202.26,176.5) .. (199.5,176.5) .. controls (196.74,176.5) and (194.5,174.26) .. (194.5,171.5) -- cycle ;
%Shape: Circle [id:dp9469278463299717] 
\draw  [fill={rgb, 255:red, 0; green, 0; blue, 0 }  ,fill opacity=1 ] (234,101.5) .. controls (234,98.74) and (236.24,96.5) .. (239,96.5) .. controls (241.76,96.5) and (244,98.74) .. (244,101.5) .. controls (244,104.26) and (241.76,106.5) .. (239,106.5) .. controls (236.24,106.5) and (234,104.26) .. (234,101.5) -- cycle ;
%Shape: Circle [id:dp333916196838276] 
\draw  [fill={rgb, 255:red, 0; green, 0; blue, 0 }  ,fill opacity=1 ] (194.5,31.5) .. controls (194.5,28.74) and (196.74,26.5) .. (199.5,26.5) .. controls (202.26,26.5) and (204.5,28.74) .. (204.5,31.5) .. controls (204.5,34.26) and (202.26,36.5) .. (199.5,36.5) .. controls (196.74,36.5) and (194.5,34.26) .. (194.5,31.5) -- cycle ;
%Shape: Arc [id:dp13426700610350872] 
\draw  [draw opacity=0] (114.22,34.41) .. controls (112.86,41.3) and (106.79,46.5) .. (99.5,46.5) .. controls (97.44,46.5) and (95.48,46.09) .. (93.7,45.34) -- (99.5,31.5) -- cycle ; \draw   (114.22,34.41) .. controls (112.86,41.3) and (106.79,46.5) .. (99.5,46.5) .. controls (97.44,46.5) and (95.48,46.09) .. (93.7,45.34) ;  
%Shape: Arc [id:dp1591989563465761] 
\draw  [draw opacity=0][dash pattern={on 0.84pt off 2.51pt}] (90.7,43.65) .. controls (86.94,40.92) and (84.5,36.5) .. (84.5,31.5) .. controls (84.5,23.22) and (91.22,16.5) .. (99.5,16.5) .. controls (107.58,16.5) and (114.17,22.89) .. (114.49,30.89) -- (99.5,31.5) -- cycle ; \draw  [dash pattern={on 0.84pt off 2.51pt}] (90.7,43.65) .. controls (86.94,40.92) and (84.5,36.5) .. (84.5,31.5) .. controls (84.5,23.22) and (91.22,16.5) .. (99.5,16.5) .. controls (107.58,16.5) and (114.17,22.89) .. (114.49,30.89) ;  
%Shape: Circle [id:dp2929325135599836] 
\draw  [dash pattern={on 0.84pt off 2.51pt}] (84.5,171.5) .. controls (84.5,163.22) and (91.22,156.5) .. (99.5,156.5) .. controls (107.78,156.5) and (114.5,163.22) .. (114.5,171.5) .. controls (114.5,179.78) and (107.78,186.5) .. (99.5,186.5) .. controls (91.22,186.5) and (84.5,179.78) .. (84.5,171.5) -- cycle ;
%Shape: Circle [id:dp2610550552222246] 
\draw  [dash pattern={on 0.84pt off 2.51pt}] (184.5,171.5) .. controls (184.5,163.22) and (191.22,156.5) .. (199.5,156.5) .. controls (207.78,156.5) and (214.5,163.22) .. (214.5,171.5) .. controls (214.5,179.78) and (207.78,186.5) .. (199.5,186.5) .. controls (191.22,186.5) and (184.5,179.78) .. (184.5,171.5) -- cycle ;
%Shape: Arc [id:dp5711770861821905] 
\draw  [draw opacity=0] (65.61,115.41) .. controls (63.88,116.11) and (61.98,116.5) .. (60,116.5) .. controls (51.72,116.5) and (45,109.78) .. (45,101.5) .. controls (45,93.22) and (51.72,86.5) .. (60,86.5) .. controls (62,86.5) and (63.91,86.89) .. (65.65,87.6) -- (60,101.5) -- cycle ; \draw   (65.61,115.41) .. controls (63.88,116.11) and (61.98,116.5) .. (60,116.5) .. controls (51.72,116.5) and (45,109.78) .. (45,101.5) .. controls (45,93.22) and (51.72,86.5) .. (60,86.5) .. controls (62,86.5) and (63.91,86.89) .. (65.65,87.6) ;  
%Shape: Arc [id:dp6448112369677204] 
\draw  [draw opacity=0][dash pattern={on 0.84pt off 2.51pt}] (69.43,89.84) .. controls (72.83,92.59) and (75,96.79) .. (75,101.5) .. controls (75,107.42) and (71.58,112.53) .. (66.6,114.97) -- (60,101.5) -- cycle ; \draw  [dash pattern={on 0.84pt off 2.51pt}] (69.43,89.84) .. controls (72.83,92.59) and (75,96.79) .. (75,101.5) .. controls (75,107.42) and (71.58,112.53) .. (66.6,114.97) ;  
%Shape: Arc [id:dp5457681819837605] 
\draw  [draw opacity=0][dash pattern={on 0.84pt off 2.51pt}] (206.13,44.96) .. controls (204.13,45.95) and (201.88,46.5) .. (199.5,46.5) .. controls (191.39,46.5) and (184.78,40.06) .. (184.51,32.02) -- (199.5,31.5) -- cycle ; \draw  [dash pattern={on 0.84pt off 2.51pt}] (206.13,44.96) .. controls (204.13,45.95) and (201.88,46.5) .. (199.5,46.5) .. controls (191.39,46.5) and (184.78,40.06) .. (184.51,32.02) ;  
%Shape: Arc [id:dp6353994127147982] 
\draw  [draw opacity=0][dash pattern={on 0.84pt off 2.51pt}] (233.2,87.66) .. controls (234.98,86.91) and (236.94,86.5) .. (239,86.5) .. controls (247.28,86.5) and (254,93.22) .. (254,101.5) .. controls (254,109.78) and (247.28,116.5) .. (239,116.5) .. controls (237.64,116.5) and (236.32,116.32) .. (235.07,115.98) -- (239,101.5) -- cycle ; \draw  [dash pattern={on 0.84pt off 2.51pt}] (233.2,87.66) .. controls (234.98,86.91) and (236.94,86.5) .. (239,86.5) .. controls (247.28,86.5) and (254,93.22) .. (254,101.5) .. controls (254,109.78) and (247.28,116.5) .. (239,116.5) .. controls (237.64,116.5) and (236.32,116.32) .. (235.07,115.98) ;  
%Shape: Arc [id:dp6738895352845895] 
\draw  [draw opacity=0] (230.6,113.93) .. controls (226.62,111.24) and (224,106.67) .. (224,101.5) .. controls (224,96.4) and (226.54,91.9) .. (230.43,89.19) -- (239,101.5) -- cycle ; \draw   (230.6,113.93) .. controls (226.62,111.24) and (224,106.67) .. (224,101.5) .. controls (224,96.4) and (226.54,91.9) .. (230.43,89.19) ;  
%Shape: Arc [id:dp35269715447011896] 
\draw  [draw opacity=0] (184.6,29.79) .. controls (185.44,22.31) and (191.79,16.5) .. (199.5,16.5) .. controls (207.78,16.5) and (214.5,23.22) .. (214.5,31.5) .. controls (214.5,36.21) and (212.33,40.41) .. (208.93,43.16) -- (199.5,31.5) -- cycle ; \draw   (184.6,29.79) .. controls (185.44,22.31) and (191.79,16.5) .. (199.5,16.5) .. controls (207.78,16.5) and (214.5,23.22) .. (214.5,31.5) .. controls (214.5,36.21) and (212.33,40.41) .. (208.93,43.16) ;  
\draw   (64.48,84.93) -- (66.46,88.05) -- (62.81,88.57) ;
\draw   (96.13,48.09) -- (93.88,45.17) -- (97.47,44.32) ;
\draw   (212.66,42.46) -- (209.05,43.22) -- (209.85,39.62) ;
\draw   (226.61,89.63) -- (230.28,89.27) -- (229.1,92.76) ;

% Text Node
\draw (144,13) node [anchor=north west][inner sep=0.75pt]   [align=left] {$1$};
\draw (142,25) node [anchor=north west][inner sep=0.75pt]  [font=\Large] [align=left] {$\times$};
\draw (142,165) node [anchor=north west][inner sep=0.75pt]  [font=\Large] [align=left] {$\times$};
% Text Node
\draw (65,55) node [anchor=north west][inner sep=0.75pt]   [align=left] {$2$};
% Text Node
\draw (144,177) node [anchor=north west][inner sep=0.75pt]   [align=left] {$k$};
% Text Node
\draw (224,55) node [anchor=north west][inner sep=0.75pt]   [align=left] {$2k$};
% Text Node
\draw (106,45.5) node [anchor=north west][inner sep=0.75pt]   [align=left] {$p_1$};
% Text Node
\draw (26.5,95) node [anchor=north west][inner sep=0.75pt]   [align=left] {$p_2$};
% Text Node
\draw (190,94) node [anchor=north west][inner sep=0.75pt]   [align=left] {$p_{2k-1}$};
% Text Node
\draw (210.5,5.5) node [anchor=north west][inner sep=0.75pt]   [align=left] {$p_{2k}$};
\end{tikzpicture}
\caption{An even cycle in $\Gamma_\red$}
		\label{fig:even-red-cycle}	
	\end{center}
	\end{figure}

Define $T_{2l}$ to be the two-term complex given by
$$
\begin{tikzcd}
        & P_{2l} \arrow[ld, "p_{2l-1}"'] \arrow[rd, "p_{2l}"] 
        &          & \cdots 
        \arrow[ld, "p_{2l+1}"'] \arrow[rd, "p_{2l-4}"] 
        &          & P_{2l-2} \arrow[ld, "p_{2l-3}"'] \arrow[rd, "p_{2l-2}"] &          \\
P_{2l-1} &                                                     & P_{2l+1} 
&                                                       & P_{2l-3} 
&                                                         & P_{2l-1}
\end{tikzcd}
$$
where $l \in \mathbb{Z}/2k\mathbb{Z}$ and $P_i$ denotes the indecomposable projective $A_{\red}$-module corresponding to the edge $i$ in $\Gamma_{\red}$.  We then construct a tilting complex
\[
T = \left( \bigoplus_{i=1}^{k} T_{2i} \right)
\oplus \left( \bigoplus_{j=1}^{k} P_{2j-1} \right)
\oplus \left( \bigoplus_{i \in I_1} P_i \right)
\oplus \left( \bigoplus_{j \in I_2} P_j[1] \right),
\]
where $I_1$ consists of those edges in $\Gamma_{\red}$ that are not contained in any path $p_i$, and $I_2$ consists of those edges in $\Gamma_{\red}$ that are contained in a path of the form $p_i$. The tilting complex $T$ constructed above is naturally $\mathrm{\phi}$-stable, and this construction can be generalized to string complexes of arbitrary length. Consequently, in this case the set $2\text{-}\mathrm{tilt}^{\mathrm{\phi}}\,A_{\red}$ is infinite.

Now assume that~(2) holds. Let $T$ be a tilting complex lying in the connected component of the tilting mutation quiver $Q_{\mathrm{tilt}}(A)$ containing $A$. By Subsection \ref{subsec:KM-FMS-BGA}, the algebra $
B := \mathrm{End}_{\mathcal{K}^b(\mathrm{proj}\,A)}(T)$
is again a biserial FBGA associated with a biserial FBG $(\Gamma',d)$. As in the proof of~\cite[Lemma~6.12]{AAC}, the reduced graph $\Gamma'_{\red}$ is also a tree. Consequently, the set $2\text{-}\mathrm{tilt}\,B$ is finite. By~\cite[Theorem~1.2]{AM}, it follows that $A$ is tilting-discrete.
\end{proof}

\begin{remark}\label{rmk:non-adm}
	Assume that \(\operatorname{char} k \neq 2\). By Propositions~\ref{prop:nonadm-skew=skew-BGA} and~\ref{prop:red-skew-BGA}, both $A$ and $A_{\mathrm{red}}$ are related via their skew group algebras, which coincide with the skew-BGA corresponding to the skew-BG $(\Gamma/\langle \nu \rangle,d)$. Consequently, the bijection between $(2\text{-})\mathrm{tilt}\,A$ and $(2\text{-})\mathrm{tilt}^{\phi}\,A_{\mathrm{red}}$ also follows from~\cite[Theorem~1.1]{KKKMM}.
\end{remark}

\subsection{Applications}
\

Combining the results of Theorems~\ref{thm:adm-case} and~\ref{thm:non-adm-case}, we obtain the following corollary.

\begin{corollary}\label{cor:der-closed-tilt-dis}
	Let $A$ be a tilting-discrete biserial FBGA with associated biserial FBG $(\Gamma,d)$. Then any basic algebra derived equivalent to $A$ is also a (tilting-discrete) biserial FBGA.
\end{corollary}

\begin{proof}
	Denote by $A_{\red}$ the BGA associated with the reduced form $(\Gamma_{\red}, d)$. By Theorems~\ref{thm:adm-case} and~\ref{thm:non-adm-case}, the graph $\Gamma_{\red}$ contains at most one odd cycle and no even cycles. As shown in~\cite[Lemma~5.13]{CKL}, for a self-injective algebra $A$, tilting-discrete implies tilting-connected. Hence any tilting complex $T$ of $A$ can be obtained from $A$ by a finite sequence of irreducible tilting mutations. By Propositions~\ref{prop:1-fms-Kauer-move}--\ref{prop:3-fms-Kauer-move}, the algebra $\mathrm{End}_{\mathcal{K}^b(\mathrm{proj}\, A)}(T)$ is again a biserial FBGA associated with some biserial FBG $(\Gamma', d)$, where $\Gamma'_{\red}$ has the same numbers of odd and even cycles as $\Gamma_{\red}$. Consequently, $\mathrm{End}_{\mathcal{K}^b(\mathrm{proj}\, A)}(T)$ is also a tilting-discrete biserial FBGA.
\end{proof}

As noted in \cite{AD,AK}, it is interesting to find algebras that are tilting-discrete but not silting-discrete. As an application of our results, we obtain some examples of biserial FBGAs which are tilting-discrete but not silting-discrete.

\begin{example}\label{exa:fms-til-dis-not-sil-dis}
	Let $A$ be a biserial FBGA with associated biserial FBG $(\Gamma,d)$ such that $m(v)\ge 1$ for every vertex $v$ of $\Gamma$. Let $A_{\red}$ be the BGA associated with the reduced form $(\Gamma_{\red},d)$. Assume that $\Gamma_{\red}$ contains at most one odd cycle and no even cycles, whereas $\Gamma$ does not satisfy this condition. Then, by Proposition~\ref{prop:2-silt-fms} and Theorems~\ref{thm:adm-case} and~\ref{thm:non-adm-case}, the self-injective algebra $A$ is tilting-discrete but not silting-discrete.
\end{example}

\appendix
\section{Proof of Kauer moves for biserial FBGAs}

By Lemma~\ref{lem:nu-orbit}, irreducible tilting mutations of biserial FBGAs fall into three classes, whose verifications correspond to the three cases treated in the BGA case in~\cite{K}. Without loss of generality, we provide a complete verification only for the first case. We note that the main difference for the reader is that the dimension comparison requires treating more cases in the present setting.

\begin{proposition}\label{prop:1-fms-Kauer-move}
	Let $\Gamma$ be a ribbon graph, and let $H' := \{h_1,\ldots,h_n\}$ be a $\langle \nu \rangle$-orbit in $\Gamma$.
Set $\vv_1 := s(h_1)$ and $\vv_2 := s(\iota(h_1))$. Assume that $val(\vv_1) \neq n$, $val(\vv_2) \neq n$, and that for each $i \in \{1,\ldots,n\}$ we have
$$
\rho(h_i) \neq \iota(h_i)
\quad\text{and}\quad
\rho^{-1}(h_i) \neq \iota(h_i).
$$
Then the tilting mutation along the $\langle \nu_A \rangle$-orbit $\{\bar{h}_1,\ldots,\bar{h}_n\}$ is described by the local move illustrated in the first graph of Figure~\ref{fig:1-KM-fms}, where the corresponding vertices in the two
biserial FBGs have the same multiplicity.

More explicitly, let the ribbon graph on the right-hand side of the first graph of Figure~\ref{fig:1-KM-fms} be denoted by $(\Gamma,d)$, and denote the ribbon graph
on the left-hand side by $(\mu_{H'}^-(\Gamma),d)$. Let $I$ be the set of indecomposable projective $A$-modules corresponding to $\{\bar{h}_1,\ldots,\bar{h}_n\}$. Then there exists a $k$-algebra isomorphism
$$
\mathrm{End}_{\mathcal{K}^b(\mathrm{proj}\,A)}(\mu^-_{P(I)}(A))
\;\cong\;
kQ_{\mu_{H'}^-(\Gamma)} \big/ I_{\mu_{H'}^-(\Gamma)} .
$$
\end{proposition}

\begin{proof}
	Let $A$ (resp.\ $B$) be the biserial FBGA associated with the right (resp.\ left) biserial FBG in Figure~\ref{fig:1-KM-fms}. Denote by $P_i$ the indecomposable projective $A$-module corresponding to the edge $\bar{h_i}$. For an indecomposable projective module $P$ corresponding to an edge $\bar{h}$, we denote by $\rho_{h}^i(P)$
the indecomposable projective module corresponding to the edge $\overline{\rho^i(h)}$, and by $\rho_{\iota(h)}^i(P)$
the indecomposable projective module corresponding to the edge $\overline{\rho^i(\iota(h))}$.

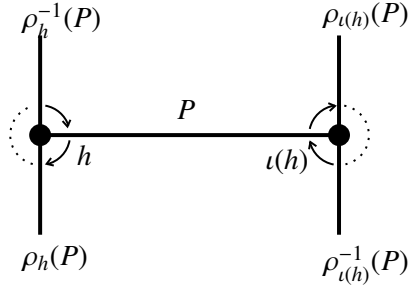
	\begin{center}

\tikzset{every picture/.style={line width=0.75pt}} %set default line width to 0.75pt        

\begin{tikzpicture}[x=0.75pt,y=0.75pt,yscale=-1,xscale=1]
%uncomment if require: \path (0,235); %set diagram left start at 0, and has height of 235

%Shape: Circle [id:dp9707066477236279] 
\draw  [fill={rgb, 255:red, 0; green, 0; blue, 0 }  ,fill opacity=1 ] (60,95) .. controls (60,92.24) and (62.24,90) .. (65,90) .. controls (67.76,90) and (70,92.24) .. (70,95) .. controls (70,97.76) and (67.76,100) .. (65,100) .. controls (62.24,100) and (60,97.76) .. (60,95) -- cycle ;
%Shape: Circle [id:dp3190307514009847] 
\draw  [fill={rgb, 255:red, 0; green, 0; blue, 0 }  ,fill opacity=1 ] (210,95) .. controls (210,92.24) and (212.24,90) .. (215,90) .. controls (217.76,90) and (220,92.24) .. (220,95) .. controls (220,97.76) and (217.76,100) .. (215,100) .. controls (212.24,100) and (210,97.76) .. (210,95) -- cycle ;
%Straight Lines [id:da13612657057932842] 
\draw [line width=1.5]    (65,95) -- (215,95) ;
%Straight Lines [id:da6665669522172712] 
\draw [line width=1.5]    (65,95) -- (65,145) ;
%Straight Lines [id:da44137964942136465] 
\draw [line width=1.5]    (215,95) -- (215,145) ;
%Straight Lines [id:da6931373066710853] 
\draw [line width=1.5]    (65,45) -- (65,95) ;
%Straight Lines [id:da7579622474804975] 
\draw [line width=1.5]    (215,45) -- (215,95) ;
%Shape: Arc [id:dp5102379576178959] 
\draw  [draw opacity=0][dash pattern={on 0.84pt off 2.51pt}] (61.29,109.54) .. controls (54.8,107.89) and (50,102.01) .. (50,95) .. controls (50,88.63) and (53.97,83.18) .. (59.58,81.01) -- (65,95) -- cycle ; \draw  [dash pattern={on 0.84pt off 2.51pt}] (61.29,109.54) .. controls (54.8,107.89) and (50,102.01) .. (50,95) .. controls (50,88.63) and (53.97,83.18) .. (59.58,81.01) ;  
%Shape: Arc [id:dp012657155951965704] 
\draw  [draw opacity=0] (67.61,80.23) .. controls (73.07,81.18) and (77.52,85.1) .. (79.23,90.26) -- (65,95) -- cycle ; \draw   (67.61,80.23) .. controls (73.07,81.18) and (77.52,85.1) .. (79.23,90.26) ;  
%Shape: Arc [id:dp33776080572955425] 
\draw  [draw opacity=0] (79.77,97.63) .. controls (78.73,103.53) and (74.23,108.24) .. (68.45,109.6) -- (65,95) -- cycle ; \draw   (79.77,97.63) .. controls (78.73,103.53) and (74.23,108.24) .. (68.45,109.6) ;  
%Shape: Arc [id:dp5098355367056471] 
\draw  [draw opacity=0] (212.3,109.76) .. controls (206.38,108.68) and (201.66,104.12) .. (200.36,98.28) -- (215,95) -- cycle ; \draw   (212.3,109.76) .. controls (206.38,108.68) and (201.66,104.12) .. (200.36,98.28) ;  
%Shape: Arc [id:dp35280706874096057] 
\draw  [draw opacity=0] (200.32,91.9) .. controls (201.61,85.78) and (206.62,81.03) .. (212.88,80.15) -- (215,95) -- cycle ; \draw   (200.32,91.9) .. controls (201.61,85.78) and (206.62,81.03) .. (212.88,80.15) ;  
%Shape: Arc [id:dp6771522483583994] 
\draw  [draw opacity=0][dash pattern={on 0.84pt off 2.51pt}] (217.61,80.23) .. controls (224.65,81.46) and (230,87.61) .. (230,95) .. controls (230,102.18) and (224.96,108.17) .. (218.23,109.65) -- (215,95) -- cycle ; \draw  [dash pattern={on 0.84pt off 2.51pt}] (217.61,80.23) .. controls (224.65,81.46) and (230,87.61) .. (230,95) .. controls (230,102.18) and (224.96,108.17) .. (218.23,109.65) ;  
\draw   (80.2,87.32) -- (78.71,90.7) -- (76.25,87.95) ;
\draw   (71.85,110.28) -- (68.26,109.44) -- (70.5,106.51) ;
\draw   (199.96,102.11) -- (200.63,98.48) -- (203.66,100.59) ;
\draw   (209.61,78.97) -- (213.01,80.39) -- (210.31,82.91) ;

% Text Node
\draw (82,98.6) node [anchor=north west][inner sep=0.75pt]   [align=left] {$h$};
% Text Node
\draw (177,100.6) node [anchor=north west][inner sep=0.75pt]   [align=left] {$\iota(h)$};
% Text Node
\draw (132,76) node [anchor=north west][inner sep=0.75pt]   [align=left] {$P$};
% Text Node
%\draw (29.6,87) node [anchor=north west][inner sep=0.75pt]   [align=left] {$w_1$};
% Text Node
%\draw (232.8,87) node [anchor=north west][inner sep=0.75pt]   [align=left] {$w_2$};
% Text Node
\draw (55,27) node [anchor=north west][inner sep=0.75pt]   [align=left] {$\rho^{-1}_{h}(P)$};
% Text Node
\draw (55,149) node [anchor=north west][inner sep=0.75pt]   [align=left] {$\rho_{h}(P)$};
% Text Node
\draw (206,149) node [anchor=north west][inner sep=0.75pt]   [align=left] {$\rho^{-1}_{\iota(h)}(P)$};
% Text Node
\draw (206,27) node [anchor=north west][inner sep=0.75pt]   [align=left] {$\rho_{\iota(h)}(P)$};

\end{tikzpicture}

	\end{center}
Denote by $T_i$ the complex
\[
 P_i \longrightarrow \rho^{-1}_{h_i}(P_i) \oplus \rho^{-1}_{\iota(h_i)}(P_i).
\]
Then the  Okuyama--Rickard complex associated with $\{\bar{h_1}, \ldots, \bar{h_n}\}$ is given by 
\[
T = \left( \bigoplus_{i=1}^{n} T_i \right) \oplus \left( \bigoplus_{j=1}^{m} Q_j \right),
\]
where each $Q_j$ is an indecomposable projective $A$-module not isomorphic to any $P_i$ ($1 \leq i \leq n$).

Denote by 
\[
B' = \mathrm{End}_{\mathcal{K}^b(\mathrm{proj}\,A)}(T).
\]
We will show that the quiver of $B$ coincides with the quiver $Q_{B'}$ of $B'$.

Consider an angle from $h'$ to $h''$ (with $\rho(h')=h''$) around a vertex $\ww$ as illustrated below, where the projective modules $Q'$ and $Q''$ are not isomorphic to any $P_i$.
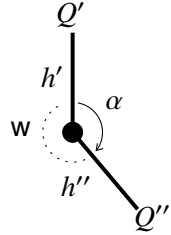
\begin{center}
\begin{tikzpicture}[x=0.75pt,y=0.75pt,yscale=-1,xscale=1]
% Shape: Circle 
\draw  [fill={rgb, 255:red, 0; green, 0; blue, 0 },fill opacity=1 ] (60,95) .. controls (60,92.24) and (62.24,90) .. (65,90) .. controls (67.76,90) and (70,92.24) .. (70,95) .. controls (70,97.76) and (67.76,100) .. (65,100) .. controls (62.24,100) and (60,97.76) .. (60,95) -- cycle ;
% Straight Lines 
\draw [line width=1.5] (65,95) -- (97.6,133.6);
\draw [line width=1.5] (65,45) -- (65,95);
% Dashed Arcs 
\draw [dash pattern={on 0.84pt off 2.51pt}] (71.21,108.66) .. controls (69.32,109.52) and (67.21,110) .. (65,110) .. controls (56.72,110) and (50,103.28) .. (50,95) .. controls (50,88.63) and (53.97,83.18) .. (59.58,81.01);
\draw (67.61,80.23) .. controls (74.65,81.46) and (80,87.61) .. (80,95) .. controls (80,98.53) and (78.78,101.78) .. (76.73,104.35);
\draw (80.39,103.16) -- (76.83,104.12) -- (77.42,100.48);
% Labels
\node at (39,93) {$\ww$};
\node at (64,36) {$Q'$};
\node at (105,138) {$Q''$};
\node at (86,80) {$\alpha$};
\node at (55,67) {$h'$};
\node at (67,122) {$h''$};
\end{tikzpicture}
\end{center}
Then there is an arrow in $Q_{B'}$ from the vertex corresponding to $Q'$ to the vertex corresponding to $Q''$, given by
$$
\begin{tikzcd}
Q'' \arrow[rr, "\alpha\cdot"] &  & Q'
\end{tikzcd}$$
This morphism does not factor through any object in $\mathrm{add}(T / (Q' \oplus Q''))$.

Now consider the angle from $\rho^{-1}(h_i)$ to $\rho(h_i)$, which passes through the half-edge $h_i$.
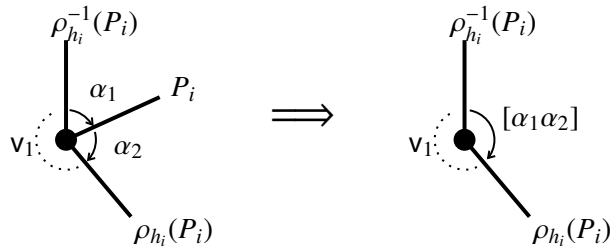
\begin{center}

\tikzset{every picture/.style={line width=0.75pt}} %set default line width to 0.75pt        

\begin{tikzpicture}[x=0.75pt,y=0.75pt,yscale=-1,xscale=1]
%uncomment if require: \path (0,235); %set diagram left start at 0, and has height of 235

%Shape: Circle [id:dp9707066477236279] 
\draw  [fill={rgb, 255:red, 0; green, 0; blue, 0 }  ,fill opacity=1 ] (59.4,94.6) .. controls (59.4,91.84) and (61.64,89.6) .. (64.4,89.6) .. controls (67.16,89.6) and (69.4,91.84) .. (69.4,94.6) .. controls (69.4,97.36) and (67.16,99.6) .. (64.4,99.6) .. controls (61.64,99.6) and (59.4,97.36) .. (59.4,94.6) -- cycle ;
%Straight Lines [id:da6665669522172712] 
\draw [line width=1.5]    (64.4,94.6) -- (97,133.2) ;
%Straight Lines [id:da6931373066710853] 
\draw [line width=1.5]    (64.4,44.6) -- (64.4,94.6) ;
%Shape: Arc [id:dp5102379576178959] 
\draw  [draw opacity=0][dash pattern={on 0.84pt off 2.51pt}] (70.61,108.26) .. controls (68.72,109.12) and (66.61,109.6) .. (64.4,109.6) .. controls (56.12,109.6) and (49.4,102.88) .. (49.4,94.6) .. controls (49.4,88.23) and (53.37,82.78) .. (58.98,80.61) -- (64.4,94.6) -- cycle ; \draw  [dash pattern={on 0.84pt off 2.51pt}] (70.61,108.26) .. controls (68.72,109.12) and (66.61,109.6) .. (64.4,109.6) .. controls (56.12,109.6) and (49.4,102.88) .. (49.4,94.6) .. controls (49.4,88.23) and (53.37,82.78) .. (58.98,80.61) ;  
%Shape: Arc [id:dp012657155951965704] 
\draw  [draw opacity=0] (66.17,79.7) .. controls (70.72,80.24) and (74.65,82.81) .. (77.01,86.48) -- (64.4,94.6) -- cycle ; \draw   (66.17,79.7) .. controls (70.72,80.24) and (74.65,82.81) .. (77.01,86.48) ;  
\draw   (79.04,103.76) -- (75.48,104.72) -- (76.07,101.08) ;
%Shape: Circle [id:dp7499140287470254] 
\draw  [fill={rgb, 255:red, 0; green, 0; blue, 0 }  ,fill opacity=1 ] (259.4,94.6) .. controls (259.4,91.84) and (261.64,89.6) .. (264.4,89.6) .. controls (267.16,89.6) and (269.4,91.84) .. (269.4,94.6) .. controls (269.4,97.36) and (267.16,99.6) .. (264.4,99.6) .. controls (261.64,99.6) and (259.4,97.36) .. (259.4,94.6) -- cycle ;
%Straight Lines [id:da11086111616299932] 
\draw [line width=1.5]    (264.4,94.6) -- (297,133.2) ;
%Straight Lines [id:da3408169693515175] 
\draw [line width=1.5]    (264.4,44.6) -- (264.4,94.6) ;
%Shape: Arc [id:dp288322688344143] 
\draw  [draw opacity=0][dash pattern={on 0.84pt off 2.51pt}] (270.61,108.26) .. controls (268.72,109.12) and (266.61,109.6) .. (264.4,109.6) .. controls (256.12,109.6) and (249.4,102.88) .. (249.4,94.6) .. controls (249.4,88.23) and (253.37,82.78) .. (258.98,80.61) -- (264.4,94.6) -- cycle ; \draw  [dash pattern={on 0.84pt off 2.51pt}] (270.61,108.26) .. controls (268.72,109.12) and (266.61,109.6) .. (264.4,109.6) .. controls (256.12,109.6) and (249.4,102.88) .. (249.4,94.6) .. controls (249.4,88.23) and (253.37,82.78) .. (258.98,80.61) ;  
%Shape: Arc [id:dp19825269849568516] 
\draw  [draw opacity=0] (267.01,79.83) .. controls (274.05,81.06) and (279.4,87.21) .. (279.4,94.6) .. controls (279.4,98.13) and (278.18,101.38) .. (276.13,103.95) -- (264.4,94.6) -- cycle ; \draw   (267.01,79.83) .. controls (274.05,81.06) and (279.4,87.21) .. (279.4,94.6) .. controls (279.4,98.13) and (278.18,101.38) .. (276.13,103.95) ;  
\draw   (279.79,102.76) -- (276.23,103.72) -- (276.82,100.08) ;
%Straight Lines [id:da7792790188484628] 
\draw [line width=1.5]    (64.4,94.6) -- (111.4,73.2) ;
%Shape: Arc [id:dp36244427817917924] 
\draw  [draw opacity=0] (78.87,90.65) .. controls (79.22,91.91) and (79.4,93.23) .. (79.4,94.6) .. controls (79.4,98.26) and (78.09,101.62) .. (75.91,104.22) -- (64.4,94.6) -- cycle ; \draw   (78.87,90.65) .. controls (79.22,91.91) and (79.4,93.23) .. (79.4,94.6) .. controls (79.4,98.26) and (78.09,101.62) .. (75.91,104.22) ;  
\draw   (77.4,83.37) -- (77.72,87.04) -- (74.24,85.82) ;

% Text Node
\draw (35,91.8) node [anchor=north west][inner sep=0.75pt]   [align=left] {$\vv_1$};
% Text Node
\draw (56.2,26) node [anchor=north west][inner sep=0.75pt]   [align=left] {$\rho^{-1}_{h_i}(P_i)$};
% Text Node
\draw (97.8,132.8) node [anchor=north west][inner sep=0.75pt]   [align=left] {$\rho_{h_i}(P_i)$};
% Text Node
\draw (74.8,65) node [anchor=north west][inner sep=0.75pt]   [align=left] {$\alpha_1$};
% Text Node
\draw (115.27,61.67) node [anchor=north west][inner sep=0.75pt]   [align=left] {$P_i$};
% Text Node
\draw (235,91.8) node [anchor=north west][inner sep=0.75pt]   [align=left] {$\vv_1$};
% Text Node
\draw (256.2,26) node [anchor=north west][inner sep=0.75pt]   [align=left] {$\rho^{-1}_{h_i}(P_i)$};
% Text Node
\draw (297.8,132.8) node [anchor=north west][inner sep=0.75pt]   [align=left] {$\rho_{h_i}(P_i)$};
% Text Node
\draw (281.6,76.4) node [anchor=north west][inner sep=0.75pt]   [align=left] {$[\alpha_1\alpha_2]$};
% Text Node
\draw (165,74.8) node [anchor=north west][inner sep=0.75pt]  [font=\LARGE] [align=left] {$\Longrightarrow$};
% Text Node
\draw (87.13,93.8) node [anchor=north west][inner sep=0.75pt]   [align=left] {$\alpha_2$};
\end{tikzpicture}
\end{center}
Then the following observations hold around the vertex $v_1$.
\begin{enumerate}
    \item There is no arrow in $Q_{B'}$ from the vertex corresponding to $\rho^{-1}_{h_i}(P_i)$ to the vertex corresponding to $T_i$, since the following chain map
	$$
    \begin{tikzcd}
   P_i \arrow[rr, "0"] \arrow[d, "{\begin{pmatrix}\alpha_1\cdot\\\text{*}\end{pmatrix}}"'] &  & 0 \arrow[d, "0"] \\
    \rho^{-1}_{h_i}(P_i) \oplus \rho^{-1}_{\iota(h_i)}(P_i) \arrow[rr, "{(1,0)}"]                                                &  & \rho^{-1}_{h_i}(P_i)                         
    \end{tikzcd}
    $$
    is not a chain map in general, where $*$ denotes an arrow around $\vv_2$ ending at $\bar{h_i}$.
    
    \item There is no arrow in $Q_{B'}$ from the vertex corresponding to $T_i$ to the vertex corresponding to $\rho_{h_i}(P_i)$, since the following morphism
	 $$
    \begin{tikzcd}
     0 \arrow[rr, "0"] \arrow[d, "0"'] &  & P_i \arrow[d, "{\begin{pmatrix}\alpha_1\cdot\\\text{*}\end{pmatrix}}"] \\  \rho_{h_i}(P_i)
    \arrow[rr, "\begin{pmatrix}\alpha_1\alpha_2\cdot\\0\end{pmatrix}"'] \arrow[rru, "\alpha_2\cdot"]                          &  & \rho^{-1}_{h_i}(P_i) \oplus \rho^{-1}_{\iota(h_i)}(P_i)            
    \end{tikzcd}
    $$
   is null-homotopic in $\mathcal{K}^b(A)$.
    
    \item There exists a new arrow $[\alpha_1\alpha_2]$ from the vertex corresponding to $\rho^{-1}_{h_i}(P_i)$ to the vertex corresponding to $\rho_{h_i}(P_i)$, given by the morphism 
    $$
    \begin{tikzcd}
    \rho_{h_i}(P_i) \arrow[rr, "\alpha_1\alpha_2\cdot"] &  & \rho^{-1}_{h_i}(P_i)
    \end{tikzcd}
    $$
    and such a morphism does not factor through any object in $\mathrm{add}(T / (\rho^{-1}_{h_i}(P_i) \oplus \rho_{h_i}(P_i)))$.
\end{enumerate}

	Consider the angle from $\rho^{-1}_{\iota(\rho(h_i))}\rho^{-1}_{h_i}(P_i)$ to  $ \rho^{-1}_{h_i}(P_i)$ around the vertex $\vv_3$.
	\begin{center}

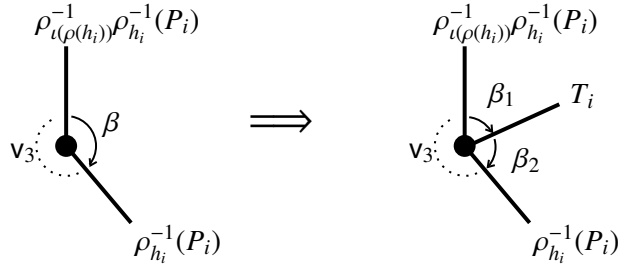
\tikzset{every picture/.style={line width=0.75pt}} %set default line width to 0.75pt        

\begin{tikzpicture}[x=0.75pt,y=0.75pt,yscale=-1,xscale=1]
%uncomment if require: \path (0,235); %set diagram left start at 0, and has height of 235

%Shape: Circle [id:dp9707066477236279] 
\draw  [fill={rgb, 255:red, 0; green, 0; blue, 0 }  ,fill opacity=1 ] (59.4,94.6) .. controls (59.4,91.84) and (61.64,89.6) .. (64.4,89.6) .. controls (67.16,89.6) and (69.4,91.84) .. (69.4,94.6) .. controls (69.4,97.36) and (67.16,99.6) .. (64.4,99.6) .. controls (61.64,99.6) and (59.4,97.36) .. (59.4,94.6) -- cycle ;
%Straight Lines [id:da6665669522172712] 
\draw [line width=1.5]    (64.4,94.6) -- (97,133.2) ;
%Straight Lines [id:da6931373066710853] 
\draw [line width=1.5]    (64.4,44.6) -- (64.4,94.6) ;
%Shape: Arc [id:dp5102379576178959] 
\draw  [draw opacity=0][dash pattern={on 0.84pt off 2.51pt}] (70.61,108.26) .. controls (68.72,109.12) and (66.61,109.6) .. (64.4,109.6) .. controls (56.12,109.6) and (49.4,102.88) .. (49.4,94.6) .. controls (49.4,88.23) and (53.37,82.78) .. (58.98,80.61) -- (64.4,94.6) -- cycle ; \draw  [dash pattern={on 0.84pt off 2.51pt}] (70.61,108.26) .. controls (68.72,109.12) and (66.61,109.6) .. (64.4,109.6) .. controls (56.12,109.6) and (49.4,102.88) .. (49.4,94.6) .. controls (49.4,88.23) and (53.37,82.78) .. (58.98,80.61) ;  
%Shape: Arc [id:dp012657155951965704] 
\draw  [draw opacity=0] (266.67,80.2) .. controls (271.22,80.74) and (275.15,83.31) .. (277.51,86.98) -- (264.9,95.1) -- cycle ; \draw   (266.67,80.2) .. controls (271.22,80.74) and (275.15,83.31) .. (277.51,86.98) ;  
\draw   (279.54,104.26) -- (275.98,105.22) -- (276.57,101.58) ;
%Shape: Circle [id:dp7499140287470254] 
\draw  [fill={rgb, 255:red, 0; green, 0; blue, 0 }  ,fill opacity=1 ] (259.4,94.6) .. controls (259.4,91.84) and (261.64,89.6) .. (264.4,89.6) .. controls (267.16,89.6) and (269.4,91.84) .. (269.4,94.6) .. controls (269.4,97.36) and (267.16,99.6) .. (264.4,99.6) .. controls (261.64,99.6) and (259.4,97.36) .. (259.4,94.6) -- cycle ;
%Straight Lines [id:da11086111616299932] 
\draw [line width=1.5]    (264.4,94.6) -- (297,133.2) ;
%Straight Lines [id:da3408169693515175] 
\draw [line width=1.5]    (264.4,44.6) -- (264.4,94.6) ;
%Shape: Arc [id:dp288322688344143] 
\draw  [draw opacity=0][dash pattern={on 0.84pt off 2.51pt}] (270.61,108.26) .. controls (268.72,109.12) and (266.61,109.6) .. (264.4,109.6) .. controls (256.12,109.6) and (249.4,102.88) .. (249.4,94.6) .. controls (249.4,88.23) and (253.37,82.78) .. (258.98,80.61) -- (264.4,94.6) -- cycle ; \draw  [dash pattern={on 0.84pt off 2.51pt}] (270.61,108.26) .. controls (268.72,109.12) and (266.61,109.6) .. (264.4,109.6) .. controls (256.12,109.6) and (249.4,102.88) .. (249.4,94.6) .. controls (249.4,88.23) and (253.37,82.78) .. (258.98,80.61) ;  
%Shape: Arc [id:dp19825269849568516] 
\draw  [draw opacity=0] (67.01,79.83) .. controls (74.05,81.06) and (79.4,87.21) .. (79.4,94.6) .. controls (79.4,98.13) and (78.18,101.38) .. (76.13,103.95) -- (64.4,94.6) -- cycle ; \draw   (67.01,79.83) .. controls (74.05,81.06) and (79.4,87.21) .. (79.4,94.6) .. controls (79.4,98.13) and (78.18,101.38) .. (76.13,103.95) ;  
\draw   (79.76,102.62) -- (76.55,104.44) -- (76.21,100.77) ;
%Straight Lines [id:da7792790188484628] 
\draw [line width=1.5]    (264.9,95.1) -- (311.9,73.7) ;
%Shape: Arc [id:dp36244427817917924] 
\draw  [draw opacity=0] (279.37,91.15) .. controls (279.72,92.41) and (279.9,93.73) .. (279.9,95.1) .. controls (279.9,98.76) and (278.59,102.12) .. (276.41,104.72) -- (264.9,95.1) -- cycle ; \draw   (279.37,91.15) .. controls (279.72,92.41) and (279.9,93.73) .. (279.9,95.1) .. controls (279.9,98.76) and (278.59,102.12) .. (276.41,104.72) ;  
\draw   (277.9,83.87) -- (278.22,87.54) -- (274.74,86.32) ;

% Text Node
\draw (35,91.8) node [anchor=north west][inner sep=0.75pt]   [align=left] {$\vv_3$};
% Text Node
\draw (48,22) node [anchor=north west][inner sep=0.75pt]   [align=left] {$\rho^{-1}_{\iota(\rho(h_i))}\rho^{-1}_{h_i}(P_i)$};
% Text Node
\draw (97.8,132.8) node [anchor=north west][inner sep=0.75pt]   [align=left] {$\rho^{-1}_{h_i}(P_i)$};
% Text Node
\draw (275.3,61.3) node [anchor=north west][inner sep=0.75pt]   [align=left] {$\beta_1$};
% Text Node
\draw (315.77,62.17) node [anchor=north west][inner sep=0.75pt]   [align=left] {$T_i$};
% Text Node
\draw (235,91.8) node [anchor=north west][inner sep=0.75pt]   [align=left] {$\vv_3$};
% Text Node
\draw (248,22) node [anchor=north west][inner sep=0.75pt]   [align=left] {$\rho^{-1}_{\iota(\rho(h_i))}\rho^{-1}_{h_i}(P_i)$};
% Text Node
\draw (297.8,132.8) node [anchor=north west][inner sep=0.75pt]   [align=left] {$\rho^{-1}_{h_i}(P_i)$};
% Text Node
\draw (82.1,75.9) node [anchor=north west][inner sep=0.75pt]   [align=left] {$\beta$};
% Text Node
\draw (154,74.8) node [anchor=north west][inner sep=0.75pt]  [font=\LARGE] [align=left] {$\Longrightarrow$};
% Text Node
\draw (287.63,94.3) node [anchor=north west][inner sep=0.75pt]   [align=left] {$\beta_2$};

\end{tikzpicture}
	\end{center}
Then the following observations hold around the vertex $\vv_3$.
\begin{enumerate}
    \item There exists a new arrow $\beta_1$ from the vertex corresponding to $\rho^{-1}_{\iota(\rho(h_i))}\rho^{-1}_{h_i}(P_i)$ to the vertex corresponding to $T_i$, given by the morphism 
    $$
\begin{tikzcd}
P_i \arrow[rr, "0"] \arrow[d, "{\begin{pmatrix}\alpha_1\cdot\\\text{*}\end{pmatrix}}"'] &  & 0 \arrow[d, "0"] \\
 \rho^{-1}_{h_i}(P_i) \oplus \rho^{-1}_{\iota(h_i)}(P_i)  \arrow[rr, "{(\beta\cdot,0)}"]               &  & \rho^{-1}_{\iota(\rho(h_i))}\rho^{-1}_{h_i}(P_i)            
\end{tikzcd}
    $$
and such a morphism does not factor through any object in $\mathrm{add}(T / (\rho^{-1}_{\iota(\rho(h_i))}\rho^{-1}_{h_i}(P_i) \oplus T_i))$.
\item There exists a new arrow $\beta_2$ from the vertex corresponding to $T_i$ to the vertex corresponding to $ \rho^{-1}_{h_i}(P_i)$, given by the morphism 
$$
\begin{tikzcd}
0 \arrow[rr, "0"] \arrow[d, "0"'] &  & P_i \arrow[d, "{\begin{pmatrix}\alpha_1\cdot\\\text{*}\end{pmatrix}}"] \\
\rho^{-1}_{h_i}(P_i) \arrow[rr, "{\begin{pmatrix}1\\0\end{pmatrix}}"]                                &  & \rho^{-1}_{h_i}(P_i) \oplus \rho^{-1}_{\iota(h_i)}(P_i)                          
\end{tikzcd}$$
and such a morphism does not factor through any object in $\mathrm{add}(T / (\rho^{-1}_{h_i}(P_i) \oplus T_i))$.
    
    \item There is no arrow in $Q_{B'}$ from the vertex corresponding to $\rho^{-1}_{\iota(\rho(h_i))}\rho^{-1}_{h_i}(P_i)$ to the vertex corresponding to $\rho^{-1}_{h_i}(P_i)$, since the original morphism $\beta$ factors through the summand $T_i$ as follows.
    $$
\begin{tikzcd}
\rho^{-1}_{h_i}(P_i) \arrow[rr, "\beta\cdot"] \arrow[rd, "\begin{pmatrix}0\\\begin{pmatrix}1\\0\end{pmatrix}\end{pmatrix}"'] &                                                           & \rho^{-1}_{\iota(\rho(h_i))}\rho^{-1}_{h_i}(P_i) \\
                                                                                                                   & T_i \arrow[ru, "{\begin{pmatrix}0\\(\beta\cdot,0)\end{pmatrix}}"'] &  
\end{tikzcd}
    $$
\end{enumerate}

Therefore, by the above discussion, we conclude that the quiver of $B$ coincides with the quiver of $B'$. In fact, the above construction also shows that the generating relations of the ideal $I$ in $B = kQ / I$ (as defined in Subsection~\ref{subsec:def-fms-BGA}) appear as relations of the corresponding ideal in $B'$ as well. For example, consider a commutativity relation $\alpha p = q$ in $I$, where $\alpha$ is an arrow starting from the vertex corresponding to $\rho^{-1}_{h_i}(P_i)$, and $p, q$ are paths in the quiver $Q$. According to the construction of the arrows in $B'$ described above, the commutativity relation $\alpha_1\alpha_2p = q$ in $A$ naturally induces the relation $[\alpha_1\alpha_2]p = q$ in $B'$. Hence, considering the canonical surjection from $kQ$ onto $B'$, the ideal $I$ is contained in the kernel of this map. Therefore, there exists a natural algebra epimorphism from $B$ to $B'$.

We now show that $\mathrm{dim}_k B = \mathrm{dim}_k B'$. We first verify that the multiplicities of the corresponding vertices in the two graphs are equal. 
By Proposition~\ref{prop:multiplicity-eqv-def}, we have $m(\vv) = \dfrac{F(\vv)}{n(\vv)}$. Moreover, the quantity $n(\ww)$ appearing in this formula remains unchanged, since the number of elements in each $\langle\nu\rangle$-orbit is preserved under the mutation. Likewise, $\alpha(w)$ is invariant, because the number of $\nu$-fans traversed by the Nakayama permutation at each vertex does not vary. Consequently, the multiplicity at every vertex remains the same.

For each indecomposable summand $P$ of the tilting complex $T$, denote by $e_P$ the corresponding idempotent in the quiver algebra $B$, and by $\bar{h}_P = \{h_P, \iota(h_P)\}$ the corresponding edge in the biserial FBG associated with $B$. Then, for any summands $P, P'$ of $T$, we verify that 
\[
\mathrm{dim}_k(e_P B e_{P'}) = \mathrm{dim}_k \mathrm{Hom}_{\mathcal{K}^b(\mathrm{proj}\,A)}(P', P).
\]

For the biserial FBGA, it can be verified that:
\begin{itemize}
  \item if $s(h_P) = s(h_{P'})$ and $s(\iota(h_P)) = s(\iota(h_{P'}))$, and the angle from $h_P$ to $\rho^{d(s(h_P))}(h_P)$ passes through $h_{P'}$, then
  \[
            \mathrm{dim}_k(e_P B e_{P'}) =
            \begin{cases}
               \lceil m(s(h_P)) \rceil + \lceil m(s(\iota(h_P))) \rceil, & \text{if } \mathrm{top}(P')\not\cong \soc(P), \\[2mm]
                \lceil m(s(h_P)) \rceil + \lceil m(s(\iota(h_P))) \rceil-1, & \text{if } \mathrm{top}(P')\cong \soc(P),
            \end{cases}
            \]
  where we denote by $\lceil m(s(h_P)) \rceil$ the smallest integer greater than or equal to $m(s(h_P))$.
  \item if $s(h_P) = s(\iota(h_{P'}))$ and $s(\iota(h_P)) = s(h_{P'})$, and the angle from $h_P$ to $\rho^{d(s(h_P))}(h_P)$ passes through $\iota(h_{P'})$, then
  \[
            \mathrm{dim}_k(e_P B e_{P'}) =
            \begin{cases}
               \lceil m(s(h_P)) \rceil + \lceil m(s(\iota(h_P))) \rceil, & \text{if } \mathrm{top}(P')\not\cong \soc(P), \\[2mm]
                \lceil m(s(h_P)) \rceil + \lceil m(s(\iota(h_P))) \rceil-1, & \text{if } \mathrm{top}(P')\cong \soc(P),
            \end{cases}
            \]
  \item if $s(h_P) = s(h_{P'})$ but $s(\iota(h_P)) \neq s(\iota(h_{P'}))$, and the angle from $h_P$ to $\rho^{d(s(h_P))}(h_P)$ passes through $h_{P'}$, then $ \mathrm{dim}_k(e_P B e_{P'}) = \lceil m(s(h_P)) \rceil$;
  \item if $s(h_P) = s(\iota(h_{P'}))$ but $s(\iota(h_P)) \neq s(h_{P'})$, and the angle from $h_P$ to $\rho^{d(s(h_P))}(h_P)$ passes through $\iota(h_{P'})$, then $\mathrm{dim}_k(e_P B e_{P'}) = \lceil m(s(h_P)) \rceil;$
  \item otherwise, $ \mathrm{dim}_k(e_P B e_{P'}) = 0.$
\end{itemize}

We next compute $\mathrm{dim}_k \mathrm{Hom}_{\mathcal{K}^b(\mathrm{proj}\,A)}(P', P)$.
\begin{enumerate}
	\item Let $P = T_i$, which is given by the complex
	\[
 P_i \longrightarrow \rho^{-1}_{h_i}(P_i) \oplus \rho^{-1}_{\iota(h_i)}(P_i).
\]
	and let $P' = Q_j$, where $Q_j$ denotes an indecomposable projective module not isomorphic to any $P_i$. If the corresponding edges in the biserial FBG are not incident to the same vertex, or if the angle around $\vv_1$ (resp. $\vv_2$) from the half-edge $h$ corresponding to $T_i$ to the half-edge $\rho^{d(\vv_1)}h$ (resp. $\rho^{d(\vv_2)}h$) does not pass through any half-edge corresponding to $Q_j$, then $\mathrm{dim}_k \mathrm{Hom}_{\mathcal{K}^b(\mathrm{proj}\,A)}(P', P) = 0.$
	Otherwise, without loss of generality, assume that the corresponding edges in the biserial FBG are incident to the same vertex $\vv_1$. 
	If no half-edge corresponding to $Q_j$ is incident to $\vv_2$, then
	\begin{align*}
	\mathrm{dim}_k \mathrm{Hom}_{\mathcal{K}^b(\mathrm{proj}\,A)}(P', P) 
	&= \mathrm{dim}_k \mathrm{Hom}_A(Q_j,\rho^{-1}_{h_i}(P_i) \oplus \rho^{-1}_{\iota(h_i)}(P_i))
	- \mathrm{dim}_k \mathrm{Hom}_A( Q_j,P_i) \\
	&= \mathrm{dim}_k \mathrm{Hom}_A( Q_j,\rho_{h_i}(P_i))
	- \mathrm{dim}_k \mathrm{Hom}_A( Q_j,P_i)=0.
	\end{align*}
	If, on the other hand, there also exists a half-edge corresponding to $Q_j$ incident to $\vv_2$, then
	$$
	\mathrm{dim}_k \mathrm{Hom}_{\mathcal{K}^b(\mathrm{proj}\,A)}(P', P) 
	= \mathrm{dim}_k \mathrm{Hom}_A( Q_j,\rho^{-1}_{\iota(h_i)}(P_i)) = \lceil m(\vv_2) \rceil.
$$
	
	\item Let $P = Q_i$ and $P' = T_j$, where $T_j$ is given by the complex
	\[
 P_j \longrightarrow \rho^{-1}_{h_j}(P_j) \oplus \rho^{-1}_{\iota(h_j)}(P_j).
\]
	The same reasoning applies symmetrically. If the corresponding edges in the biserial FBG are not incident to the same vertex, or if the angle (induced by Nakayama permutation) around $\vv_1$ (resp. $\vv_2$) from the half-edge $\rho^{-d(\vv_1)}h$ (resp. $\rho^{-d(\vv_2)}h$) to the half-edge $h$ corresponding to $Q_i$ does not pass through any half-edge corresponding to $P_j$, then $	\mathrm{dim}_k \mathrm{Hom}_{\mathcal{K}^b(\mathrm{proj}\,A)}(P', P) = 0.$
	Otherwise, without loss of generality, assume that the corresponding edges in the biserial FBG are incident to the same vertex $\vv_1$. 
	If no half-edge corresponding to $Q_i$ is incident to $\vv_2$, then $	\mathrm{dim}_k \mathrm{Hom}_{\mathcal{K}^b(\mathrm{proj}\,A)}(P', P) = 0$.
	If, on the other hand, there also exists a half-edge corresponding to $Q_j$ incident to $\vv_2$, then $
	\mathrm{dim}_k \mathrm{Hom}_{\mathcal{K}^b(\mathrm{proj}\,A)}(P', P) 
	= \lceil m(\vv_2) \rceil.$
	
	\item Now suppose both $P$ and $P'$ are of the form $T_j$ and $T_i$, respectively. Then the corresponding edges share common vertices $\vv_1$ and $\vv_2$. Hence
	\begin{align*}
	\mathrm{dim}_k \mathrm{Hom}_{\mathcal{K}^b(\mathrm{proj}\,A)}(P', P) 
	&= \mathrm{dim}_k \mathrm{Hom}_A(\rho^{-1}_{h_i}(P_i) \oplus \rho^{-1}_{\iota(h_i)}(P_i), \rho^{-1}_{h_j}(P_j) \oplus \rho^{-1}_{\iota(h_j)}(P_j)) + \mathrm{dim}_k \mathrm{Hom}_A(P_i, P_j) 
	\\& - \mathrm{dim}_k \mathrm{Hom}_A(\rho^{-1}_{h_i}(P_i) \oplus \rho^{-1}_{\iota(h_i)}(P_i), P_j)-\mathrm{dim}_k \mathrm{Hom}_A(P_i, \rho^{-1}_{h_j}(P_j) \oplus \rho^{-1}_{\iota(h_j)}(P_j)) \\
	&= \mathrm{dim}_k \mathrm{Hom}_A(P_i, P_j) \\
	&= \begin{cases}
               \lceil m(s(h_P)) \rceil + \lceil m(s(\iota(h_P))) \rceil, & \text{if } \mathrm{top}(P')\not\cong \soc(P), \\[2mm]
                \lceil m(s(h_P)) \rceil + \lceil m(s(\iota(h_P))) \rceil-1, & \text{if } \mathrm{top}(P')\cong \soc(P).
            \end{cases}
	\end{align*}
	
	\item Finally, if both $P$ and $P'$ are of the form $Q_i$ and $Q_j$, that is, indecomposable projective $A$-modules not isomorphic to any $P_i$, then $
\mathrm{dim}_k(e_{Q_i} B e_{Q_j}) 
= \mathrm{dim}_k \mathrm{Hom}_A(Q_j, Q_i)$,
since the edges corresponding to these modules remain unchanged under the graph mutation.
\end{enumerate}

Combining all the above cases, we conclude that for every pair of indecomposable summands $P, P'$ of the tilting complex $T$, $\mathrm{dim}_k(e_P B e_{P'}) 
= \mathrm{dim}_k \mathrm{Hom}_{\mathcal{K}^b(\mathrm{proj}\,A)}(P', P)$.
Hence, $\mathrm{dim}_k B = \mathrm{dim}_k B'$. Moreover, since there exists a natural algebra epimorphism from $B$ to $B'$, we deduce that $B \cong B'$ as algebras. Therefore, the tilting mutation on the $\langle \nu_A \rangle$-orbit $\{\bar{h}_1, \ldots, \bar{h}_n\}$ is described by the first graph in Figure~\ref{fig:1-KM-fms}.
\end{proof}


\begin{thebibliography}{88}

		\bibitem{Ada} T. Adachi, The classification of $\tau$-tilting modules over Nakayama algebras. J. Algebra \textbf{452} (2016), 227--262.

\bibitem{AAC} T. Adachi, T. Aihara, A. Chan, Classification of two-term tilting complexes over Brauer graph algebras. Math. Z. \textbf{290} (2018), 1--36.

\bibitem{AIR} T. Adachi, O. Iyama, I. Reiten, $\tau$-tilting theory. Compos. Math. \textbf{150} (2014), 415--452.

\bibitem{AK} T. Adachi, R. Kase, Examples of tilting-discrete self-injective algebras which are not silting-discrete. Publ. Res. Inst. Math. Sci. \textbf{60} (2024), 373--411.

\bibitem{Ai} T. Aihara, Tilting-connected symmetric algebras. Algebr. Represent. Theory \textbf{16} (2013), 873--894.

\bibitem{AI} T. Aihara, O. Iyama, Silting mutation in triangulated categories. J. Lond. Math. Soc. \textbf{85} (2012), 633--668.

\bibitem{AM} T. Aihara, Y. Mizuno, Classifying tilting complexes over preprojective algebras of Dynkin type. Algebra Number Theory \textbf{11} (2017), 1287--1315.

\bibitem{AR} S. Al-Nofayee, J. Rickard, Rigidity of tilting complexes and derived equivalence for self-injective algebras. arXiv:1311.0504.

\bibitem{AP} C. Amiot, P. G. Plamondon, The cluster category of a surface with punctures via group actions. Adv. Math. \textbf{389} (2021), 107884.

\bibitem{APS} C. Amiot, P.G. Plamondon, S. Schroll, A complete derived invariant for gentle algebras via winding numbers and Arf invariants. Selecta Math. \textbf{29} (2023), no. 2, Paper No. 30.

\bibitem{AP2} C. Amiot, P. G. Plamondon, skew group $A_{\infty}$-categories as Fukaya categories of orbifolds. arXiv:2405.15466.

\bibitem{AZ} M. Antipov, A. Zvonareva, Brauer graph algebras are closed under derived equivalence. Math. Z. \textbf{301} (2022), 1963--1981.

\bibitem{AY} T. Aoki, T. Yurikusa, Complete gentle and special biserial algebras are $g$-tame. J. Algebr. Comb. \textbf{57} (2023), 1103--1137.

\bibitem{Asa1} H. Asashiba, The derived equivalence classification of representation-finite selfinjective algebras. J. Algebra \textbf{214} (1999), 182--221.

\bibitem{Asa2} H. Asashiba, On a lift of an individual stable equivalence to a standard derived equivalence for representation-finite self-injective algebras. Algebr. Represent. Theory \textbf{6} (2003), 427--447.

\bibitem{Asa3} H. Asashiba, A generalization of Gabriel's Galois covering functors and derived equivalences. J. Algebra \textbf{334} (2011), 109--149.

\bibitem{AD} J. August, A. Dugas, Silting and tilting for weakly symmetric algebras. Algebr. Represent. Theory \textbf{26} (2023), 169--179.

\bibitem{BSW} S. Barmeier, S. Schroll, Z. Wang, Partially wrapped Fukaya categories of orbifold surfaces. arXiv:2407.16358.

\bibitem{BC} K. Baur, R. Coelho Sim\~oes, A geometric model for the module category of a string algebra. arXiv:2403.07810.

\bibitem{BG} K. Bongartz, P. Gabriel, Covering spaces in representation theory. Invent. Math. \textbf{65} (1982), 331--378.

\bibitem{CKL} A. Chan, S. Koenig, Y. Liu, Simple-minded systems, configurations and mutations for representation-finite self-injective algebras. J. Pure Appl. Algebra \textbf{219} (2015), 1940--1961.

\bibitem{CS} W. Chang, S. Schroll, A geometric realization of silting theory for gentle algebras. Math. Z. \textbf{303} (2023), 67.

\bibitem{D} E. C. Dade, Blocks with cyclic defect groups. Ann. of Math. \textbf{84} (1966), 20--48.

\bibitem{DF} P. W. Donovan, M. R. Freislich, The indecomposable modular representations of certain groups with dihedral Sylow subgroup. Math. Ann. \textbf{238} (1978), 207--216.

\bibitem{EJR} F. Eisele, G. Janssens, T. Raedschelders, A reduction theorem for $\tau$-rigid modules. Math. Z. \textbf{290} (2018), 1377--1413.

\bibitem{EGV} A. Elsener, V. Guazzelli, Y. Valdivieso, Skew-Brauer graph algebras. arXiv:2410.01942.

\bibitem{GS1} E. L. Green, S. Schroll, Multiserial and special multiserial algebras and their representations. Adv. Math. \textbf{302} (2016), 1111--1136.

\bibitem{J} G. J. Janusz, Indecomposable modules for finite groups. Ann. of Math. \textbf{89} (1969), 209--241.

\bibitem{JSW} H. Jin, S. Schroll, Z. Wang, A complete derived invariant and silting theory for graded gentle algebras. arXiv:2303.17474.

\bibitem{K} M. Kauer, Derived equivalences of graph algebras. Contemp. Math. \textbf{229} (1998), 201--213.

\bibitem{KV} B. Keller, D. Vossieck, Aisles in derived categories. Bull. Soc. Math. Belg. Sér. A \textbf{40} (1988), 239--253.

\bibitem{KKKMM} Y. Kimura, R. Koshio, Y. Kozakai, H. Minamoto, Y. Mizuno, $\tau$-tilting theory and silting theory of skew group algebra extensions. Ann. Represent. Theory \textbf{2} (2025), 599--637.

\bibitem{LL} N. Li, Y. Liu, Fractional Brauer configuration algebras I: definitions and examples. J. Algebra \textbf{692} (2026), 336--378.

\bibitem{LL2} N. Li, Y. Liu, Fractional Brauer configuration algebras II: covering theory. arXiv:2412.13445.

\bibitem{LL3} N. Li, Y. Liu, Fractional Brauer configuration algebras III: fractional Brauer graph algebras in type MS. arXiv:2412.13449.

\bibitem{Miz} Y. Mizuno, On mutations of selfinjective quivers with potential. J. Pure Appl. Algebra \textbf{219} (2015), 1742--1760.

\bibitem{Oku} T. Okuyama, Some examples of derived equivalent blocks of finite groups. Unpublished manuscript.

\bibitem{OPS} S. Opper, P. G. Plamondon, S. Schroll, A geometric model for the derived category of gentle algebras. arXiv:1801.09659.

\bibitem{OZ} S. Opper, A. Zvonareva, Derived equivalence classification of Brauer graph algebras. Adv. Math. \textbf{402} (2022), 108341.

\bibitem{RR} I. Reiten, C. Riedtmann, Skew group algebras in the representation theory of Artin algebras. J. Algebra \textbf{92} (1985), 224--282.

\bibitem{Ric} J. Rickard, Morita theory for derived categories. J. Lond. Math. Soc. \textbf{39} (1989), 436--456.

\bibitem{Ric2} J. Rickard, Derived categories and stable equivalence. J. Pure Appl. Algebra \textbf{61} (1989), 303--317.

\bibitem{Ried} C. Riedtmann, Representation-finite selfinjective algebras of class $A_n$. In: Representation Theory II (Ottawa, 1979), 449--520.

\bibitem{Sch} S. Schroll, Trivial extensions of gentle algebras and Brauer graph algebras. J. Algebra \textbf{444} (2015), 183--200.

\bibitem{SS} S. Schroll, Brauer graph algebras. In: Homological Methods, Representation Theory, and Cluster Algebras, Springer, 2018, pp. 177--223.

\bibitem{So1} V. Soto, Generalized Kauer moves and derived equivalences of Brauer graph algebras. J. Algebra \textbf{657} (2024), 514--548.

\bibitem{So2} V. Soto, Tilting mutations as generalized Kauer moves for (skew) Brauer graph algebras with multiplicity. arXiv:2406.10634.

\bibitem{W} J. Wei, Semi-tilting complexes. Israel J. Math. \textbf{194} (2013), 871--893.

\bibitem{X} B. Xing, Quasi-biserial algebras, special quasi-biserial algebras and symmetric fractional Brauer graph algebras. arXiv:2408.03778.

\bibitem{Z} A. Zimmermann, \textit{Representation theory: a homological algebra point of view}. Springer, 2014.
	
\end{thebibliography}
\end{document}